\documentclass[12pt, a4paper,
oneside, headinclude,footinclude]{scrartcl}
\usepackage[utf8]{inputenc}

\usepackage[
nochapters, % Turn off chapters since this is an article        
pdfspacing, % Makes use of pdftex’ letter spacing capabilities via the microtype package
dottedtoc % Dotted lines leading to the page numbers in the table of contents
]{classicthesis} % The layout is based on the Classic Thesis style
\usepackage{amsopn}
\usepackage[T1]{fontenc} % Use 8-bit encoding that has 256 glyphs
\usepackage{multirow}
\usepackage[utf8]{inputenc} % Required for including letters with accents
\usepackage{hhline}
\usepackage{graphicx} % Required for including images
\graphicspath{{Figures/}} % Set the default folder for images
\usepackage{indentfirst}
\usepackage{enumitem} % Required for manipulating the whitespace between and within lists
\usepackage{savesym}
\usepackage{mathrsfs}
\usepackage{geometry}
\usepackage{esint}
\usepackage{longtable}

\usepackage[
    backend=biber,
    style=numeric,
    natbib=false,
    url=false, 
    doi=false,
    eprint=false,
    maxnames=50
]{biblatex}

\usepackage{tikz}
\usetikzlibrary{cd}
\usetikzlibrary{decorations.markings}

\usepackage{subfig} % Required for creating figures with multiple parts (subfigures)

\usepackage{amsmath, amssymb,amsthm,amsfonts} % For including math equations, theorems, symbols, etc

\usepackage[english,capitalize]{cleveref}

\usepackage{thmtools}
\usepackage{thm-restate}

\usepackage{hyperref}
\newcommand{\vertiii}[1]{{\left\vert\kern-0.25ex\left\vert\kern-0.25ex\left\vert #1 
    \right\vert\kern-0.25ex\right\vert\kern-0.25ex\right\vert}}

\theoremstyle{plain}
\newtheorem{teorema}{Theorem}[section]
\newtheorem{proposizione}[teorema]{Proposition}
\newtheorem{congettura}{Conjecture}[section]
\newtheorem{lemma}[teorema]{Lemma}
\newtheorem{corollario}[teorema]{Corollary}
\newtheorem*{theorem*}{Theorem}

\theoremstyle{definition}
\newtheorem{definizione}{Definition}[section]

\theoremstyle{remark}
\newtheorem{osservazione}{Remark}[section]

\newcommand{\N}{\mathbb{N}}

\newcommand{\Q}{\mathbb{Q}}
\newcommand{\R}{\mathbb{R}}

\newcommand{\res}

\newcounter{const}
\newcommand{\newC}{\refstepcounter{const}\ensuremath{C_{\theconst}}}
\newcommand{\oldC}[1]{\ensuremath{C_{\ref{#1}}}}

\newcounter{eps}
\newcommand{\newep}{\refstepcounter{eps}\ensuremath{\varepsilon_{\theeps}}}
\newcommand{\oldep}[1]{\ensuremath{\varepsilon_{\ref{#1}}}}

\DeclareMathOperator*{\diam}{diam}

\DeclareMathOperator*{\dist}{dist}
\hypersetup{
colorlinks=true, breaklinks=true,bookmarksnumbered,
urlcolor=webbrown, linkcolor=RoyalBlue, citecolor=webgreen, % Link colors
pdftitle={}, % PDF title
pdfauthor={\textcopyright}, % PDF Author
pdfsubject={}, % PDF Subject
pdfkeywords={}, % PDF Keywords
pdfcreator={pdfLaTeX}, % PDF Creator
pdfproducer={LaTeX with hyperref and ClassicThesis} % PDF producer
} 

\title{\normalfont\spacedallcaps{Carnot rectifiability and Alberti representations}} 
\author{\spacedlowsmallcaps{G. Antonelli\textsuperscript{*} and E. Le Donne\textsuperscript{**} and A. Merlo\textsuperscript{***}}}

\date{} 

\begin{document}

\renewcommand{\sectionmark}[1]{\markright{\spacedlowsmallcaps{#1}}} 
\lehead{\mbox{\llap{\small\thepage\kern1em\color{halfgray} \vline}\color{halfgray}\hspace{0.5em}\rightmark\hfil}} 
\pagestyle{scrheadings}
\maketitle 
\setcounter{tocdepth}{2}

\paragraph*{Abstract}
A metric measure space is said to be Carnot-rectifiable if it can be covered up to a null set by countably many biLipschitz images of compact sets of a fixed Carnot group. In this paper, we give several characterisations of such notion of rectifiability both in terms of Alberti representations of the measure and in terms of differentiability of Lipschitz maps with values in Carnot groups.

In order to obtain this characterisation, we develop and study the analogue of the notion of Lipschitz differentiability space by Cheeger, using Carnot groups and Pansu derivatives as models. We call such metric measure spaces \textit{Pansu differentiability spaces} (PDS).

{\let\thefootnote\relax\footnotetext{* \textit{Courant Institute Of Mathematical Sciences (NYU), 251 Mercer Street, 10012, New York, USA,
\href{ga2434@nyu.edu}{ga2434@nyu.edu}}}}
{\let\thefootnote\relax\footnotetext{**
\textit{University of Fribourg, Chemin du Mus\'ee~23, 1700 Fribourg, Switzerland \&  University of Jyv\"askyl\"a, Department of Mathematics and Statistics, P.O. Box (MaD), FI-40014, Finland, \href{enrico.ledonne@unifr.ch}{enrico.ledonne@unifr.ch}}
}}
{\let\thefootnote\relax\footnotetext{*** \textit{Departamento de Matem\'aticas, Universidad del Pa\' is Vasco, Barrio Sarriena s/n 48940 Leioa, Spain, \href{andrea.merlo@ehu.eus}{andrea.merlo@ehu.eus}}}}



\tableofcontents

\section{Introduction}

{
\subsection{Lipschitz differentiability spaces and rectifiability}
As of today, it is understood that there is a tight connection between the rectifiability of metric spaces and the differentiability of Lipschitz functions. 

In \cite{CheegerGAFA} J. Cheeger introduced and studied the notion of \textit{Lipschitz differentiability space}: namely, a metric measure space $(X,d,\mu)$ with countably many Lipschitz charts with values in Euclidean spaces such that every real-valued Lipschitz function on $X$ is $\mu$-almost everywhere differentiable with respect to every chart. In \cite{CheegerGAFA}, see also \cite{Keith, KleinerMackay}, it is proved that every metric measure space that is doubling and supports a Poincaré inequality is a Lipschitz differentiable space, and the dimension of the target of the Lipschitz charts is bounded from above by a constant only depending on the doubling and Poincaré constants. For a partial converse to the previous theorem; see \cite{ErikssonBiqueGAFA} and references therein. In the recent years there has been a flourishing literature on the subject, see, e.g.,  \cite{CheegerGAFA, Keith, CK9, KleinerMackay, BateJAMS, GuyCDavidLDS, Schioppa16, Schioppa16Bis, CKS16, bateli, BateLi2, ErikssonBiqueGAFA, KleinerDavid, BateOrp, BateActa, BateInv}.

One of the remarkable contributions to the study of Lipschitz differentiability spaces was recently given by D. Bate in \cite{BateJAMS}: Bate proved that a metric measure space $(X,d,\mu)$ is a Lipschitz differentiability space if and only if it can be written as a countable union of Borel sets $X=\cup U_i$ such that each $\mu\llcorner U_i$ possesses a finite collection of (universal) Alberti representations. Roughly speaking, an Alberti representation is a way to express the measure on the space as an integral of the 1-dimensional Hausdorff measure $\mathcal{H}^1$ on distinguished curves.

The first to realize the connection between the differentiability of Lipschitz functions and the decomposition of measures in rectifiable curves was Alberti in the early 2000s. Alberti introduced the notion of decomposability bundle\footnote{In \cite{alberti1993} the decomposition of the derivatives of BV functions was obtained in terms of decompositions with $1$-codimensional surfaces, which is a strictly stronger notion than that of Alberti representations.} in his celebrated paper \cite{alberti1993}, where he proved the Rank-One Property for the singular part of the derivative of vector-valued BV functions defined on an open subset of a Euclidean space conjectured by De Giorgi and Ambrosio \cite{DeGiorgiAmbrosio}. See also \cite{ACP10}, and \cite{AlbertiMarchese}. These ideas jump-started a collaboration with D. Preiss and M. Cs\"ornyei that led to the characterisations of non-differentiability sets of Lipschitz maps, see for instance \cite{ACP10} for the statement of the results, introducing many of the notions and techniques that are now considered standard in the study of non-differentiability properties of Lipschitz functions in metric spaces.

\smallskip

In \cite{bateli} Bate--Li exploited the aforementioned result by Bate in \cite{BateJAMS}, to characterize the $n$-rectifiability of a metric space. More precisely, they showed that a metric measure space $(X,d,\mu)$ is $n$-rectifiable, according to the classical Federer's definition, if and only if it can be $\mu$-almost everywhere decomposed as a countable union of Borel sets $U_i$, with $n$-density properties with respect to $\mu$, such that, equivalently, either $(U_i,d,\mu)$ is an $n$-Lipschitz differentiable space, or each $\mu\llcorner U_i$ has $n$ independent Alberti representations with respect to some Lipschitz map $\varphi_i:X\to\mathbb R^n$.
\smallskip

\subsection{Rectifiability in Carnot groups}

For surveys on Carnot groups, we refer the reader to \cite{SC16, LD17}. Carnot groups coincide with the Lie groups equipped with geodesic distances that admit dilations.
Carnot groups have a rich algebraic structure that interacts with the metric structure, see \cref{sec2}.

From the geometric viewpoint, Carnot groups represent a different world with respect to the Euclidean one since their Hausdorff dimension is strictly greater than their topological dimension unless they are Abelian. Anyway, Carnot groups are natural objects to consider not only because they arise as infinitesimal models in sub-Riemannian geometry \cite{Mitchell} but also because they are the asymptotic cones of connected Riemannian nilpotent Lie groups, see Pansu's work \cite{PansuCroissance}.  They are also infinitesimal models of doubling geodesic metric spaces with unique GH-tangent, see \cite{zbMATH05977088}

The study of Geometric Measure Theory and rectifiability in Carnot groups was pioneered by the works of Ambrosio--Kirchheim \cite{AK00} and Franchi--Serapioni--Serra Cassano \cite{Serapioni2001RectifiabilityGroup}. The research in this topic has been very active in the last two decades, see, e.g., \cite{MagnaniUnrect, step2, ambled, Mag13, MagnaniTowardArea, MatSerSC, JNGV20, DLDMV19, Vittone20, DDFO20, antonelli2020rectifiable2, antonelli2020rectifiableA, antonelli2020rectifiableB}, and the Introduction of \cite{PhDAntonelli} for further references.
In Carnot groups, the natural generalisation of Federer's notion of rectifiability using Carnot groups as models was first studied by Pauls and Cole--Pauls \cite{Pauls04, CP06}. See also \cite{ALD}. It is worth noting that, in general, Lipschitz and biLipschitz rectifiability with Carnot models might differ; see \cite{LeDonneLiRajala}.

\subsection{Main results}

In this paper, a {\em metric measure space} $(X,d,\mu)$ is a triple of a complete and separable metric space $(X,d)$ endowed with a Radon measure $\mu$. In order to state our main result, we need to introduce some notation. 
First, we define the analog, in the realm of Carnot groups, of the definition of Lipschitz differentiability spaces, see \cite{CheegerGAFA, Keith}, modeled on Pansu differentiability theorem for Lipschitz maps between Carnot groups \cite{Pansu}.

\begin{definizione}[Lipschitz chart]\label{def:Lipschitzchart}
Let $\mathbb G,\mathbb H$ be Carnot groups. Let $(X,d,\mu)$ be a metric measure space. Let $U\subset X$ be a Borel set, and let $\varphi: X\to\mathbb G$ be a Lipschitz function. 

We say that $(U,\varphi)$ is a \textit{Lipschitz chart with target $\mathbb G$ for $\mathbb H$-valued maps}, or simply \textit{chart}, when the following holds. 
Every Lipschitz function $f:X\to\mathbb H$ is {\em differentiable} $\mu$-almost everywhere in $(U,\varphi)$; i.e., for $\mu$-almost every $x_0\in U$ there exists a unique homogeneous homomorphism $Df(x_0):\mathbb G\to\mathbb H$ such that
    \begin{equation}\label{eqn:PANSUDIFF}
        \limsup_{X\ni x\to x_0}\frac{d_{\mathbb H}(f(x),f(x_0)\cdot Df(x_0)(\varphi(x_0)^{-1}\cdot\varphi(x)))}{d(x,x_0)}=0.
    \end{equation}
\end{definizione}

\begin{definizione}[Pansu differentiability space]\label{def:PansuDifferentiabilitySpaceINTRO}
Let $\mathbb H$ be a Carnot group. A metric measure space $(X,d,\mu)$ is said to be a \emph{Pansu differentiability space with respect to $\mathbb H$} if there exist Borel sets $U_i\subset X$ such that 
\[
\mu  (X\setminus\cup_{i\in\mathbb N}U_i)=0,
\]
and there exist Carnot groups $\mathbb G_i$, with homogeneous dimension uniformly bounded, and Lipschitz functions $\varphi_i:X\to\mathbb G_i$,  such that $(U_i,\varphi_i)$ is a chart for every $i\in\mathbb N$.
\end{definizione}

When in the previous definition $\mathbb G_i$ is constantly equal to some Carnot group $\mathbb G$, we say that $(X,d,\mu)$ is a \textit{$(\mathbb G,\mathbb H)$-(Pansu) differentiability space}. 
Notice that by definition, an $n$-Lipschitz differentiability space in the sense of \cite{CheegerGAFA} is a $(\mathbb R^n,\mathbb R^m)$-Lipschitz differentiability space for $m=1$, and consequently for every $m\geq 1$.
We are finally ready to state the main result of the paper.

\begin{teorema}\label{thm:SeanLiINTRO}
Let $\mathbb G$ be a Carnot group of homogeneous dimension $Q$. Let $(X,d,\mu)$ be a metric measure space. The following facts are equivalent:
\begin{itemize}
    \item[(i)] $(X,d)$ is biLipschitz $\mathbb{G}$-rectifiable, see \cref{def:GbiLipRect}, and $\mu$ is mutually absolutely continuous with respect to the $Q$-dimensional Hausdorff measure $\mathcal{H}^Q$ on $X$;
    \item[(ii)]
    For every Carnot group $\mathbb H$, the following holds. There is a countable collection $\{U_i\}_{i\in\mathbb N}$\footnote{Actually, the collection $\{U_i\}_{i\in\mathbb N}$ can be taken to be independent on $\mathbb H$.} of Borel sets of $X$ with
    \begin{equation}\label{eqn:munull}
    \mu  (X\setminus\cup_{i\in\mathbb N}U_i)=0, \end{equation}
    such that,
    \begin{equation}\label{eqn:Density}
    \forall i\in\mathbb N,\quad 0<\Theta^{Q}_*(\mu\llcorner U_i,x)\leq \Theta^{Q,*}(\mu\llcorner U_i,x)<\infty, \qquad \text{for $\mu$-almost every $x\in U_i$},
    \end{equation}
 and each $(U_i,d,\mu)$ is a $(\mathbb{G},\mathbb H)$-differentiability space. The densities $\Theta^{Q}_*(\mu,\cdot)$ and $\Theta^{Q,*}(\mu,\cdot)$ are introduced in \cref{def:density};
    \item[(iii)] There exists a Carnot group $\mathbb H$ for which item (ii) holds;
    \item[(iv)] There exists a countable collection $\{U_i\}_{i\in\mathbb N}$ of Borel sets of $X$ such that \eqref{eqn:munull} holds, \eqref{eqn:Density}
    holds, and, for every $i\in \N$, there exist a Lipschitz function $\varphi_i:U_i\to\mathbb G$ such that each $\mu\llcorner U_i$ has $n_1$ $\varphi_i$-independent Alberti representations, where $n_1$ is the rank of $\mathbb G$. See \cref{sec:Alberti} for relevant definitions about Alberti representations.
    \item[(v)] There exists a countable collection $\{U_i\}_{i\in\mathbb N}$ of Borel sets of $X$ such that \eqref{eqn:munull} holds, \eqref{eqn:Density}
    holds, and  each $U_i$ satisfies  David Condition with respect to a Lipschitz map $\varphi_i:U_i\to\mathbb G$, see \cref{def:DavidsCondition}.
\end{itemize}
\end{teorema}

Let us stress that, as shown by the example constructed in \cite{LeDonneLiRajala}, see \cref{rem:PerForzaBilip}, in the item (i) above, it is not possible to substitute biLipschitz $\mathbb G$-rectifiable with $\mathbb G$-Lipschitz rectifiable. We notice that a distinguished class of $\mathbb G$-biLipschitz rectifiable spaces is that of sub-Riemannian manifolds with constant nilpotentization $\mathbb G$, see \cite{LDY19}.

Notice that, up to a Borel decomposition, \cref{thm:SeanLiINTRO} implies that whenever the dimensionality condition \eqref{eqn:Density} holds, $X$ is a $(\mathbb G,\mathbb R)$-differentiability space if and only if it is a $(\mathbb G,\mathbb H)$-differentiability space for every Carnot group $\mathbb H$. 
However, the previous equivalence might not be true in general, as the following example shows. See \cref{rem:NotDropComplete} for more details. Take the first Heisenberg group $(\mathbb H^1,d,\mathcal{H}^4)$, where $d$ is a Carnot--Carathéodory distance. This is an $(\mathbb R^2,\mathbb R)$-differentiability space, but it is not an $(\mathbb R^2,\mathbb H^1)$-differentiability space. If it were, take $\varphi:\mathbb H^1\to\mathbb R^2$ one of the charts that makes it an $(\mathbb R^2,\mathbb H^1)$-differentiability space, and $U$ a domain of this $\varphi$, with $\mathcal{H}^4(U)>0$. Then for $\mathcal{H}^4$-a.e. $x\in U$, $\varphi$ is Pansu differentiable at $x$, and then, by homogeneity, there is a non-trivial direction along which the Pansu differential at $x$ of $\varphi$ vanishes. Then, along this direction, one cannot differentiate the identity map $\mathrm{id}_{\mathbb H^1}:\mathbb H^1\to \mathbb H^1$ in the chart $\varphi$, contradicting the fact that $(\mathbb H^1,d,\mathcal{H}^4)$ is an $(\mathbb R^2,\mathbb H^1)$-differentiability space.

In \cref{thm:SeanLiINTRO}, the implication (i)$\Rightarrow$(ii) is a consequence of the Pansu differentiability theorem in charts, while the implication (iii)$\Rightarrow$(iv) is a consequence of the following characterisation of PDS in terms of Alberti representations.

\begin{teorema}\label{chardiffspacesINTRO}
Let $(X,d,\mu)$ be a metric measure space. Let $\mathbb G$ and $\mathbb H$ be Carnot groups, and let $n_1$ denote the dimension of the first stratum of $\mathbb G$.  
Then the following are equivalent:
\begin{enumerate}
    \item $(X,d,\mu)$ is a $(\mathbb G,\mathbb H)$-differentiability space.
    \item There exist countably many Lipschitz functions $\varphi_i:X\to\mathbb G$, and countably many Borel subsets $U_i$ of $X$, for $i\in \N$, such that 
    \begin{itemize}
        \item $\mu(X\setminus \cup_{i\in\mathbb N} U_i)=0$;
        \item for every $i\in\mathbb N$ there exists $\rho_i>0$ and there exist $\varphi_i$-independent horizontally $\rho_i$-universal Alberti representations $\mathcal{A}_1,\dots,\mathcal{A}_{n_1}$ of $\mu\llcorner U_i$, see \cref{def:IndependentAlberti} and \cref{def:UniversalRep};
        \item for every $i\in\mathbb N$ and every Lipschitz function $f:X\to\mathbb H$, there exist Borel sets $U_i^j\subset U_i$ with $\mu(U_i\setminus \cup_{j\in\mathbb N}U_i^j)=0$ such that, for every $j\in\mathbb N$, the pair $(U_i^j,\varphi_i)$ is $(\mathbb G,\mathbb H)$-complete with respect to $f$, see \cref{def:GHcomplete}.
    \end{itemize}
\end{enumerate}
    Moreover, if $(X,d,\mu)$ is a $(\mathbb G,\mathbb H)$-differentiability space, then one can choose a set of charts $\{(U_i,\varphi_i)\}_{i\in\mathbb N}$ for $(X,d,\mu)$ such that the three bullets of item 2 hold.
\end{teorema}

The challenging part of the proof of \cref{chardiffspacesINTRO} is the implication 2.$\Rightarrow$1., and in particular the step of constructing a differential defined on $\mathbb G$ when having at disposal only the \textbf{horizontal fragments} given by the Alberti representations $\mathcal{A}_1,\ldots,\mathcal{A}_{n_1}$. 
The requirement of completeness of charts in the statement of \cref{chardiffspacesINTRO} cannot be dropped in any non-Abelian Carnot group; see \cref{rem:NotDropComplete}. Completeness ensures that, at almost every point $x\in U$, if the infinitesimal Lipschitz constant of the chart $\varphi$ is $0$ along some directions, then so it is for the function $f$ with respect to which the chart is complete and this ensures certain compatibility properties of concatenations of curves with respect to charts.
See \cref{def:GHcomplete} for details. For a more thorough review of the proof, we refer the reader to the beginning of \cref{sec3}, and in particular to the proof of \cref{thm:Fondamentale2}.

\medskip

We shall remark that every Pansu differentiability space is a Lipschitz differentiability space in the sense of \cite{CheegerGAFA}. This can be seen by projecting a chart $\varphi: X\to\mathbb G$ on the first stratum of $\mathbb G$. For a more detailed explanation, see \cref{rem:Importante}. Thus, we are studying a distinguished sub-class of Lipschitz differentiability spaces. As it happens for Pansu differentiability theorem \cite{Pansu}, which gives more information with respect to Cheeger's differentiability theorem on PI spaces \cite{CheegerGAFA} when applied to Carnot groups, our results deal with a notion of differentiability that is stronger than the relative notion one has when one considers the Pansu differentiability spaces as a Lipschitz differentiability space.

Going back to the discussion of the proof of our main result, the most interesting and difficult implication of \cref{thm:SeanLiINTRO}
is (iv)$\Rightarrow$(v), proved in  \cref{sec6}. 
The key steps to achieve this are in contained
\cref{l:david-inclusion} and
\cref{l:induction-alternative}.
We shall remark that our construction is very delicate. 
In order to make the proof work we have to construct suitable tilings of Carnot groups. Building on \cite{MR2466427, tilings}, given $n_1$ narrow horizontal independent cones, we construct self-similar tilings such that the centers of the dyadic subcubes $p_j$ enjoy the following additional property: for every basis of the horizontal stratum in those narrow independent cones, from every $p_j$ we can reach points sufficiently close to the origin by flowing along the horizontal basis with uniform lower bounds on time, see \cref{sec82}. Under the hypothesis in item (iv) of \cref{thm:SeanLiINTRO} above, after refining the representations, we can assume to work with independent representations that go in the direction of narrow cones, to which it is attached a fixed cube as discussed before, see item (v) of \cref{cubofigo}. Then, we get the alternatives in \cref{l:induction-alternative}, using also the estimate from \cref{lemma:drift}. Finally, \cref{l:induction-alternative} is used to prove \cref{l:david-inclusion}, which is basically telling us in a quantitative way that having independent Alberti representations implies having big projections in charts.

Finally, let us highlight that in \cref{sec5} we prove that a $(\mathbb G,\mathbb H)$-differentiability space satisfies an analogue of De Philippis--Marchese--Rindler's resolution of Cheeger's conjecture.
The following \cref{cheegerpansuINTRO} extends \cite[Theorem~4.1.1]{DePhilippisMarcheseRindler} to the realm of Carnot groups. Our proof leverages on the deep results obtained in \cite{ReversePansu}, where the authors prove the analog of the reverse Rademacher theorem on Carnot groups. We stress that our result holds for every chart.

\begin{teorema}\label{cheegerpansuINTRO}
Let $(X,d,\mu)$ be a Pansu-differentiability space with respect to some Carnot group $\mathbb H$. Let $U\subset X$ such that $(U,\varphi)$ is a Lipschitz chart with target $\mathbb G$ for $\mathbb H$-valued maps, where $\mathbb G$ is a Carnot group. Hence $\varphi_\sharp(\mu\llcorner U)$ is absolutely continuous with respect to every Haar measure of $\mathbb{G}$.
\end{teorema}

This work also lays the theoretical foundations to approach \cref{cong}. \cref{cong} would be a way of connecting the classical notion of rectifiability, i.e., in terms of images of (bi)Lipschitz maps, with the infinitesimal regularity given by prescribing the uniqueness of the pmGH-tangents. For the intrinsic version of this local regularity of measures in Carnot groups, we refer to \cite{antonelli2020rectifiableA,antonelli2020rectifiableB,antonelli2020rectifiable2}. 
The reason for which our main result is relevant for the solution of \cref{cong} is that the characterisation of rectifiability by means of Alberti representations, and thus in terms of differentiability of Lipschitz maps, is a fundamental step to understanding the structure of the images of metric spaces under Lipschitz maps. The next step towards the proof of \cref{cong} will be a Besicovitch-projection-type theorem for generic (typical) Lipschitz maps and the study of weak tangential properties of $\mathbb{G}$-purely unrectifiable metric spaces in a similar fashion to what has been done in \cite{BateActa}, with the huge obstacle that in our setting Lipschitz maps defined on subsets of Carnot groups might not be extended. 
Finally, besides the specific applications, we think the notions introduced here have an independent interest. Thus it would be interesting, for instance, to construct further examples of Pansu differentiability spaces modeled on Carnot groups, see \cref{rk:laakso}, and to better understand if the notion of Pansu differentiability space is related to some form of Poincaré inequality, or Lip-lip inequality, see \cref{rem:LipLip}.

\medskip

We finally remark here that the previous \cref{thm:SeanLiINTRO}, \cref{chardiffspacesINTRO}, and \cref{cheegerpansuINTRO} are stated with $\mathbb H$ being a Carnot group just for the sake of exposition. The analogous versions hold without any changes in the proofs with $\mathbb H$ being a Banach-homogeneous group with the Radon--Nikodym property \cite{MagnaniRajala}, or with more general structures like scalable groups with the Radon--Nikodym property, see \cite{Scalable}. Indeed, it is enough that the target metric group $\mathbb H$ has a homogeneous structure and that every Lipschitz fragment is almost everywhere differentiable with respect to this homogeneous structure. See \cref{rem:MoreGeneraltarget} for further details.

\medskip

\textbf{Acknowledgments}. 
A part of this work was done while the authors were at 
the University of Fribourg. The excellent work conditions are acknowledged.
During the writing of this work, all three authors were partially supported by the Swiss National Science Foundation
(grant 200021-204501 `\emph{Regularity of sub-Riemannian geodesics and applications}')
and by the European Research Council  (ERC Starting Grant 713998 GeoMeG `\emph{Geometry of Metric Groups}').  
E.L.D was also partially supported by the Academy of Finland 
 (grant 322898
`\emph{Sub-Riemannian Geometry via Metric-geometry and Lie-group Theory}').
 A.M. 
was also partially supported by the European Union’s Horizon Europe research and innovation programme under the Marie Sk\l odowska-Curie grant agreement no 101065346.

\section{Preliminaries}\label{sec2}
In this section, we discuss some preliminary facts about Carnot groups and differentiability spaces. 

\subsection{Carnot groups}

In this subsection, we set the basic definitions and notation about Carnot groups. Basic references on Carnot groups are \cite{MR2363343, LD17}. 

A {\em stratifiable group} $\mathbb{G}$ of step $s$\label{num:step} is a simply connected Lie group whose Lie algebra $\mathfrak g$ admits a stratification $\mathfrak g=V_1\, \oplus \, V_2 \, \oplus \dots \oplus \, V_s$. We say that $V_1\, \oplus \, V_2 \, \oplus \dots \oplus \, V_s$ is a {\em stratification} of $\mathfrak g$ if $\mathfrak g = V_1\, \oplus \, V_2 \, \oplus \dots \oplus \, V_s$ as vector spaces, and moreover
\[
[V_1,V_i]=V_{i+1}, \quad \text{for every $i=1,\dots,s-1$}, \quad  \quad [V_1,V_s]=\{0\}, \quad \text{and}\,\, V_s\neq\{0\},
\]
where $[A,B]:=\mathrm{span}\{[a,b]:a\in A,b\in B\}$. We call $V_1$ the \emph{horizontal stratum} of $\mathbb G$. We usually denote by $n$ the topological dimension of $\mathfrak g$, and by $n_j$ the dimension of $V_j$ for every $j=1,\dots,s$.
Furthermore, we define $\pi_i:\mathfrak g\to V_i$ to be the projection maps with respect to the direct sum decomposition on the $i$-th stratum. 
We will often shorten the notation to $v_i:=\pi_iv$. Sometimes, with an abuse of notation, we denote $\pi_1:=\pi_{\mathbb G}$.

A {\em Carnot group} $\mathbb G$, or {\em stratified group}, is a stratifiable group $\mathbb G$ on which we fix a stratification. The identity element of $\mathbb G$ will be denoted by $0$ when we are identifying $\mathbb G$ with $\mathbb R^n$ by means of exponential coordinates, which we are now going to introduce.

For a Carnot group $\mathbb G$, the exponential map $\exp:\mathfrak g \to \mathbb{G}$ is a global diffeomorphism from $\mathfrak g$ to $\mathbb{G}$.
Hence, if we choose a basis $\{X_1,\dots , X_n\}$ of $\mathfrak g$,  every $p\in \mathbb{G}$ can be written in a unique way as 
\begin{equation}\label{eqn:ExponentialCoordinates}
p=\exp (p_1X_1+\dots +p_nX_n).
\end{equation}
This means that we can identify $p\in \mathbb{G}$ with the $n$-tuple $(p_1,\dots , p_n)\in \R^n$ and the group $\mathbb{G}$ itself with $\R^n$ endowed with the group operation $\cdot$ determined by the Baker-Campbell-Hausdorff formula. When we say that $(X_1,\dots,X_n)$ is an \textit{adapted basis} of $\mathfrak g$ we mean that $\{X_1,\dots,X_n\}$ is a basis of $\mathfrak g$, and $\{X_{1+\sum_{i=1}^{j-1} n_i},\dots,X_{n_{j}+\sum_{i=1}^{j-1} n_i}\}$ is a basis of $V_j$ for every $j=1,\dots,s$ (where, by convention, $\sum_{i=1}^0 n_i:=0$).

\textbf{In what follows, we usually identify $\mathbb G$ with $\mathbb R^n$ by means of a choice of an adapted basis $(X_1,\ldots,X_n)$ of $\mathfrak g$}. After this identification, we endow $\mathbb R^n$ with the standard scalar product. Then, when we say that we chose an orthonormal basis, we mean that we are choosing an orthonormal basis with respect to this scalar product. 

For every $p\in \mathbb{G}$, we define the left translation $\tau _p:\mathbb{G} \to \mathbb{G}$ as the map
\begin{equation*}
q \mapsto \tau _p q := p\cdot q.
\end{equation*}

The stratification of $\mathfrak{g}$ carries with it a natural family of dilations $\delta_\lambda :\mathfrak{g}\to \mathfrak{g}$, that are Lie algebra automorphisms of $\mathfrak{g}$ and are defined by\label{intrdil}
\begin{equation}
     \delta_\lambda (v_1,\dots , v_s):=(\lambda v_1,\lambda^2 v_2,\dots , \lambda^s v_s), \quad \text{for every $\lambda\in\mathbb R\setminus{0}$}.
     \nonumber
\end{equation}
We will also denote with $\delta_\lambda$ the automorphism on $\mathbb G$ defined as $\exp\circ\delta_\lambda\circ\exp^{-1}$. A homomorphism $f:\mathbb G\to \mathbb H$ between two Carnot groups is said to be \textit{homogeneous} if it commutes with $\delta_\lambda$ for every $\lambda>0$.

As already remarked above, the group operation $\cdot$ is determined by the ~Baker-Campbell-Hausdorff formula, and, in exponential coordinates, it has the form (see \cite[Proposition~2.1]{step2})
\begin{equation*}
p\cdot q= p+q+\mathscr{Q}(p,q), \quad \mbox{for all }\, p,q \in  \R^n,
\end{equation*} 
where $\mathscr{Q}=(\mathscr{Q}_1,\dots , \mathscr{Q}_s):\R^n\times \R^n \to \R^n$, and the $\mathscr{Q}_i$'s have the following properties. For every $i=1,\ldots s$ and every $p,q\in \mathbb{G}$ we have\label{tran}
\begin{itemize}
    \item[(i)]$\mathscr{Q}_i(\delta_\lambda p,\delta_\lambda q)=\lambda^i\mathscr{Q}_i(p,q)$,
    \item[(ii)] $\mathscr{Q}_i(p,q)=-\mathscr{Q}_i(-q,-p)$,
    \item[(iii)] $\mathscr{Q}_1=0$ and $\mathscr{Q}_i(p,q)=\mathscr{Q}_i(p_1,\ldots,p_{i-1},q_1,\ldots,q_{i-1})$.
\end{itemize}
Thus, we can represent the product $\cdot$ as
\begin{equation}\label{opgr}
p\cdot q= (p_1+q_1,p_2+q_2+\mathscr{Q}_2(p_1,q_1),\dots ,p_s +q_s+\mathscr{Q}_s (p_1,\dots , p_{s-1} ,q_1,\dots ,q_{s-1})). 
\end{equation}

We recall that a sub-algebra $\mathfrak{h}$ of $\mathfrak{g}$ is said to be \textit{homogeneous} if it is $\delta_{\lambda}$-invariant for every $\lambda>0$. We recall that, given an algebra $\mathfrak{h}$, a direct-sum decomposition $\mathfrak h=W_1\oplus\dots\oplus W_{s}$ as vector spaces is a {\em grading} of $\mathfrak h$ if $[W_i,W_j]\subseteq W_{i+j}$ for every $1\leq i,j\leq s$, where we mean that $W_\ell:=\{0\}$ for every $\ell > s$. The stratification of the Lie algebra $\mathfrak{g}$  naturally induces a grading on each of its homogeneous Lie sub-algebras $\mathfrak{h}$, i.e.,
\begin{equation}
    \mathfrak{h}=V_1\cap \mathfrak{h}\oplus\ldots\oplus V_s\cap \mathfrak{h}.
    \label{eq:intr1}
\end{equation}

\begin{definizione}[Homogeneous subgroups]\label{homsub}
A subgroup $\mathbb V$ of $\mathbb{G}$ is said to be \emph{homogeneous} if it is a Lie subgroup of $\mathbb{G}$ that is invariant under the dilations $\delta_\lambda$. Given a homogeneous subgroup $\mathbb V$ of $\mathbb G$, we denote with $\mathrm{Lie}(\mathbb V)$ its Lie algebra.
\end{definizione}

Homogeneous Lie subgroups of $\mathbb{G}$ are in bijective correspondence through $\exp$ with the homogeneous Lie sub-algebras of $\mathfrak{g}$.
For every Lie algebra $\mathfrak{h}$ with grading $\mathfrak h= W_1\oplus\ldots\oplus W_{s}$, we define its \emph{homogeneous dimension} as
\[
\text{dim}_{\mathrm{hom}}(\mathfrak{h}):=\sum_{i=1}^{s} i\cdot\text{dim}(W_i).
\]
Thanks to \eqref{eq:intr1} we infer that, if $\mathfrak{h}$ is a homogeneous Lie sub-algebra of $\mathfrak{g}$, we have $\text{dim}_{\mathrm{hom}}(\mathfrak{h}):=\sum_{i=1}^{s} i\cdot\text{dim}(\mathfrak{h}\cap V_i)$. We now introduce the class of homogeneous and left-invariant distances.

\begin{definizione}[Homogeneous left-invariant distance]
A metric $d:\mathbb{G}\times \mathbb{G}\to \R$ is said to be homogeneous and left-invariant if for every $x,y\in \mathbb{G}$ we have
\begin{itemize}
    \item[(i)] $d(\delta_\lambda x,\delta_\lambda y)=\lambda d(x,y)$ for every $\lambda>0$,
    \item[(ii)] $d(\tau_z x,\tau_z y)=d(x,y)$ for every $z\in \mathbb{G}$.
\end{itemize}
\end{definizione}

We remark that two homogeneous left-invariant distances on a Carnot group are biLipschitz equivalent, and moreover, they induce the manifold topology on $\mathbb G$, see \cite{LDNG19}.
It is well-known that the Hausdorff dimension (for a definition of Hausdorff dimension see for instance \cite[Definition 4.8]{Mattila1995GeometrySpaces}) of a graded Lie group $\mathbb G$ with respect to an arbitrary left-invariant homogeneous distance coincides with the homogeneous dimension of its Lie algebra. For a reference for the latter statement, see \cite[Theorem~4.4]{LDNG19}. Given a Carnot group we will usually denote with $Q$ its homogeneous dimension.

We recall that a \emph{homogeneous norm} $\|\cdot\|$ on $\mathbb G$ is a function $\|\cdot\|:\mathbb G\to [0,+\infty)$ such that $\|x\|=\|x^{-1}\|$ for every $x\in\mathbb G$; $\|\delta_\lambda x\|=\lambda\|x\|$ for every $\lambda>0$ and $x\in\mathbb G$; $\|x\cdot y\|\leq \|x\|+\|y\|$ for every $x,y\in\mathbb G$; and $\|x\|=0$ if and only if $x=e$. Given an arbitrary homogeneous norm $\|\cdot\|$ on $\mathbb G$, the homogeneous left-invariant distance $d$ induced by $\|\cdot\|$ is defined as follows
$$
d(x,y):=\lVert x^{-1}\cdot y\rVert.
$$
Vice-versa, given a homogeneous left-invariant distance $d$, it induces a homogeneous norm through the equality $\|x\|:=d(x,e)$ for every $x\in\mathbb G$, where $e$ is the identity element of $\mathbb G$.
\medskip

We now introduce a distinguished homogeneous norm on $\mathbb G$.

\begin{definizione}[Box norm]\label{smoothnorm}
Let $\mathbb G$ be a Carnot group. Let $\mathcal{B}:=\{X_1,\dots,X_n\}$ be a basis of $\mathfrak g$, and, for every $i=1\dots,s$, let us identify each vector space $V_i$ with a vector subspace of $\mathbb R^n$ by means of the exponential map and the coordinates associated to $\mathcal{B}$. We will denote with $|\cdot|$ the standard Euclidean norms on such vector subspaces. Then there exist $\varepsilon_2,\ldots \varepsilon_{s}$ depending on the group $\mathbb{G}$ such that, if we define
\[
\lVert g\rVert:=\max\{\lvert g_1\rvert,\varepsilon_2\lvert g_2\rvert^{1/2},\ldots, \varepsilon_{s}\lvert g_s\rvert^{1/{s}}\} \qquad \text{for all $g\in\mathbb G$},
\]
then $\lVert\cdot\rVert$ is a homogeneous norm on $\mathbb{G}$ that induces a left-invariant homogeneous distance. We refer to \cite[Lemme II.1]{Guivarch1973} or \cite[Section 5]{step2} for a proof of this fact. 
\end{definizione}

There is a distinguished class of left-invariant homogeneous distances on Carnot groups, known as {\em Carnot-Carathéodory distances}. If we fix a norm $\|\cdot\|_1$ on the first stratum $V_1$ of the Lie algebra $\mathfrak g$ of $\mathbb G$, we can extend it left-invariantly to the horizontal bundle 
	\begin{equation}\label{eqn:HorizontalBundle}
	\mathbb V_1(x):=(\tau_x)_{*}V_1,
	\end{equation}
	for $x\in\mathbb G$, where $\tau_x$ is the left translation by $x$, and $V_1$ is seen as a subspace of $T_e\mathbb G\equiv \mathfrak g$. We say that an absolutely continuous curve $\gamma:[0,1]\to \mathbb G$ is {\em horizontal} if
	$$
	\gamma'(t)\in \mathbb V_1(\gamma(t)), \qquad \mbox{for almost every} \quad  t\in [0,1].
	$$
	We define
	\begin{equation}\label{eqn:CCDistance}
	d_{\mathrm{cc}}^{\|\cdot\|_1}(x,y):= \inf\left\{\int_0^1 \|\gamma'(t)\|_1dt: \quad \gamma(0)=x, \quad \gamma(1)=y, \quad \gamma \quad \mbox{horizontal}  \right\}.
	\end{equation}
	The Chow-Rashevskii theorem states that this distance is finite. It is clearly homogeneous and left-invariant. 

Let us recall some other definitions and Pansu--Rademacher theorem.
\begin{definizione}[Definition of $\mathrm{Lip}f$]
Given a Lipschitz $f:(X,d_X)\to (Y,d_Y)$ between metric spaces, we define 
\[
\mathrm{Lip}f(x):=\limsup_{r\to 0}\left\{\frac{d_Y(f(x),f(y))}{d_X(x,y)}:0<d_X(x,y)<r\right\},
\]
and
\[
\mathrm{Lip}f:=\sup\left\{\frac{d_Y(f(x),f(y))}{d_X(x,y)}:x,y\in X, x\neq y\right\}.
\]
\end{definizione}
\begin{definizione}[Norm of a homogeneous homomorphism]
Let $\mathbb G,\mathbb H$ be Carnot groups endowed with homogeneous norms $\|\cdot\|_{\mathbb G},\|\cdot\|_{\mathbb H}$. Let $L:\mathbb G\to\mathbb H$ be a homogeneous homomorphism. Then 
\[
\|L\|:=\sup_{\|v\|_{\mathbb G}=1}\|Lv\|_{\mathbb H}.
\]
\end{definizione}
\begin{teorema}[Pansu--Rademacher theorem, see \cite{Pansu, MagnaniPhD}]
Let $f:K\subset \mathbb G\to\mathbb H$ be a Lipschitz function, where $K$ is a measurable set, and $\mathbb G,\mathbb H$ are Carnot groups endowed with homogeneous norms $\|\cdot\|_{\mathbb G},\|\cdot\|_{\mathbb H}$. Then $f$ is Pansu differentiable almost everywhere, i.e., for almost every $x\in\mathbb G$ there is  a homogeneous homomorphism $Df(x):\mathbb G\to\mathbb H$ such that 
    \[
    0=\lim_{K\ni y\to x}\frac{\|(Df(x)(x^{-1}y))^{-1}f(x)^{-1}f(y)\|_{\mathbb H}}{\|x^{-1}y\|_{\mathbb G}}.
    \]
\end{teorema}
{
Finally, let us record here the following result, which is classical in control theory. In the setting of Carnot groups, the following proposition is a direct consequence of \cite[Lemma~3.33]{ABB}, see also \cite[Lemma~3.10, and Lemma~4.5]{ALDPolynomial}. For an analogous statement, see \cite[Theorem~19.2.1]{MR2363343}. 
\begin{lemma}\label{propdecomposizione}
	Let $\mathbb G$ be a Carnot group of topological dimension $n$ and let $\mathfrak{V}:=\{v_1,\ldots,v_{n_1}\}$ be a basis of $V_1(\mathbb G)$. Then there exists an open bounded set $ V\subseteq (0,1)^n$ and $(\hat s_1,\dots,\hat s_n)\in (0,1)^n$ such that 
	\begin{equation}\label{eqn:HatPsi}
	F_{\mathfrak{V}}: V\to U, \qquad \text{where}\qquad F_{\mathfrak{V}}(s_1,\dots,s_n):= \delta_{s_1}(v_{i_1})\dots\delta_{s_{n}}(v_{i_{n}})\delta_{-\hat s_{n}}(v_{i_{n}})\dots\delta_{-\hat s_{1}}(v_{i_{1}}),
	\end{equation}
	is a diffeomorphism of class $C^\infty$ and $U$ is a neighourhood of $0$.
 
    In particular, there exists a constant $c_0$ depending on $\mathfrak{V}$ and on the homogeneous norm on $\mathbb G$ such that the following holds. Let $M:=2n$. For all $v\in \mathbb G$ there exist $s_1,\dots,s_{M}\in\mathbb R$ 
    such that 
	\begin{equation}\label{eqn:Writev}
	v=\delta_{s_1}(v_{i_1})\dots\delta_{s_{M}}(v_{i_{M}}),
	\end{equation}
with 
\[
|s_i|\|v_{i_1}\|\leq c_0\|v\|, \forall i=1,\ldots,M. 
\]
\end{lemma}
}

\subsection{Differentiability spaces}
In this section, we will restate the definition of Pansu-differentiability space given in the introduction \cref{def:PansuDifferentiabilitySpaceINTRO} in slightly different terms. First, let us recall the classical definition of Lipschitz differentiability space \cite{CheegerGAFA, Keith}.

\begin{definizione}
A metric measure space $(X,d,\mu)$ is said to be a \emph{Lipschitz differentiability space} if, for every $i\in\mathbb N$, there exist Borel sets $U_i\subset X$, natural numbers $n_i$, and Lipschitz functions $\varphi_i:X\to\mathbb R^{n_i}$,  such that the three following items hold. 
\begin{enumerate}
    \item We have $\mu  (X\setminus\cup_{i\in\mathbb N}U_i)=0$;
    \item We have $n_i$ uniformly bounded;
     \item We have that $(U_i,\varphi_i)$ is a Lipschitz chart with target $\mathbb R^{n_i}$ for $\mathbb R$-valued maps, see \cref{def:Lipschitzchart}. I.e., for every Lipschitz function $f:X\to\mathbb R$, for every $i\in\mathbb N$, and for $\mu$-almost every $x_0\in U_i$ there exists a linear map $Df(x_0):\mathbb R^{n_i}\to\mathbb R$ such that
    \begin{equation}\label{eqn:FINE}
        \limsup_{X\ni x\to x_0}\frac{|f(x)-f(x_0)-Df(x_0)[\varphi_i(x)-\varphi_i(x_0)]|}{d(x_0,x)}=0.
    \end{equation}
\end{enumerate}
\end{definizione}
When in the previous definition $n_i$ is constantly equal to $n$ we say that $(X,d,\mu)$ is an \textit{$n$-Lipschitz differentiability space}.

\subsubsection{Pansu differentiability space: definition and basic properties}

In this section, after the definition of Pansu differentiability space, see \cref{def:PansuDifferentiabilitySpace}, we will show that in every big portion of a chart of a Pansu differentiability space, one has the strict differentiability of $f$, see \cref{diffunif}, and \cref{prop:GoingTo0InChart}. These last two propositions will eventually imply that, as a consequence of the very definition, we can always choose charts that are \textbf{complete} with respect to a given Lipschitz function, in the sense of \cref{def:GHcomplete}, see \cref{rem:GoingTo0}.

As we discussed in the previous section, every couple of homogeneous left-invariant norms are biLipschitz equivalent on $\mathbb G$. Without loss of generality, every time that a Carnot group is taken into account, we fix on it a homogeneous left-invariant norm that agrees with some Euclidean norm on the first stratum of the stratification, see \cref{smoothnorm}.

\begin{definizione}\label{def:PansuDifferentiabilitySpace}
Let $\mathbb H$ be a Carnot group. A metric measure space $(X,d,\mu)$ is said to be a \emph{Pansu differentiability space with respect to $\mathbb H$} if, for every $i\in\mathbb N$, there exist Borel sets $U_i\subset X$, Carnot groups $\mathbb G_i$, and Lipschitz functions $\varphi_i:X\to\mathbb G_i$,  such that the three following items hold.
\begin{enumerate}
    \item We have $\mu  (X\setminus\cup_{i\in\mathbb N}U_i)=0$;
    \item We have $\mathrm{dim}_{\mathrm{hom}}\mathbb G_i$ uniformly bounded;
     \item We have that $(U_i,\varphi_i)$ is a Lipschitz chart with target $\mathbb G_i$ for $\mathbb H$-valued maps, see \cref{def:Lipschitzchart}.
\end{enumerate}
\end{definizione}

When in the previous definition $\mathbb G_i$ is constantly equal to some Carnot group $\mathbb G$, we say that $(X,d,\mu)$ is a \textit{$(\mathbb G,\mathbb H)$-(Pansu) differentiability space}. 
Notice that an $n$-Lipschitz differentiability space is a $(\mathbb R^n,\mathbb R^m)$-Lipschitz differentiability space for every $m\geq 1$.

\begin{osservazione}[Pansu differentiability spaces and Lipschitz differentiability spaces]\label{rem:Importante}
If $(X,d,\mu)$ is a Pansu differentiability space with respect to some Carnot group $\mathbb H$, by composing the Lipschitz chart $\varphi:X\to \mathbb G$ with the projection $\pi_{1}:\mathbb G\to (V_{1},d_{\mathrm{eu}})$, where $V_{1}$ is the first stratum of the stratification of $\mathbb G$, we have that $(X,d,\mu)$ is a Lipschitz differentiability space. Indeed, if $f:X\to \mathbb R$ is Lipschitz, by isometrically embedding $\mathbb R$ into $\mathbb H$ (in the first stratum of $\mathbb H$), we have a map $\widetilde f:X\to \mathbb H$. 
Differentiating, we get the Pansu differential $D\widetilde f:\mathbb G\to \mathbb H$. We claim that applying \eqref{eqn:PANSUDIFF}, restricted to the first stratum of $\mathbb G$, and by using the fact that $Df(x_0)$ is a homogeneous homomorphism, one gets \eqref{eqn:FINE}. Indeed, let us explain this with a diagram and a computation.
\begin{center}
\begin{tikzcd} 
		X \arrow{r}{f}  
  \ar{d}{\varphi} 
  & \R \arrow[hookrightarrow]{r}  & \mathbb H \ar[two heads]{r}{\pi} & V_1(\mathbb H) \\
\mathbb G \ar{rru}{Df}   \arrow[two heads]{d}{\pi} &  & &
   \\
 \R^{n_1}  \simeq V_1(\mathbb G) 
 \qquad \ar{rrruu}{Df|_{V_1}}  &  &  &
   .
	\end{tikzcd}
\end{center}
{
Using that $ Df(x_0)$ is a group morphism and hence commutes with the abelianization $\pi_1$ on $\mathbb G$ and $\mathbb H$, and that $\pi_1: \mathbb H \to V_1(\mathbb H)$ is 1-Lipschitz, we get
 \begin{eqnarray}\label{0eqn:PANSUDIFF}
           d\left(f(x) ,f(x_0) \cdot Df(x_0)|_{V_1}[(\pi\circ\varphi)(x_0)^{-1}\cdot (\pi\circ\varphi)(x)]  \right) \\
           \qquad = 
  d\left(f(x) ,f(x_0) \cdot \pi (Df(x_0) [ \varphi (x_0)^{-1}\cdot  \varphi (x)] ) \right)\\
         \qquad   \qquad  \leq
        d\left(f(x) ,f(x_0)  \cdot Df(x_0)[\varphi(x_0)^{-1}\cdot \varphi(x)],  \right)
        \end{eqnarray} 
and then, \eqref{eqn:PANSUDIFF} implies \eqref{eqn:FINE}.
}
\end{osservazione}

Let us now show that on every big portion of a chart of a Pansu differentiability space, one has the strict differentiability of $f$, see \cref{diffunif}. Let us first notice, in the following \cref{propcartestrutturate}, that we can restrict to consider $\lambda$-structured chart in the horizontal directions (cf. \cite{BateJAMS}), and moreover, we have a uniform bound on $\|Df\|$. 

{
\begin{proposizione}\label{propcartestrutturate}
Suppose $(X,d,\mu)$ is a Pansu differentiability space with respect to $\mathbb H$. Then, we can find charts $(U_i,\varphi_i)$ for $X$ satisfying items
1, 2 and 3 of Definition~\ref{def:PansuDifferentiabilitySpace},
such that for every $i\in\N$ we have $\mu(U_i)<\infty$, and there exists a $\lambda_i>0$ such that the following holds. For every $R>0$ there exists an $r>0$ such that for every $x_0\in U_i$ there are points $x_1,\ldots,x_{n_1(\mathbb G_i)}\in U$ such that $r<d(x_j,x_0)<R$ and 
\begin{equation}
    \max_{1\leq j\leq n_1(\mathbb G_i)}\frac{\lvert \langle\pi_{1,i}\circ \varphi_i(x_j)-\pi_{1,i}\circ \varphi_i(x_0), v\rangle\rvert}{d(x_j,x_0)}\geq \lambda_i,
    \label{lambdastructuredchart}
\end{equation}
for every $v\in \mathbb{S}^{n_1(\mathbb{G}_i)-1}$, where $\mathbb S^{n_1(\mathbb G_i)-1}$ is the Euclidean sphere in the horizontal stratum of $\mathbb G_i$, and $\langle\cdot,\cdot\rangle$ is the standard scalar product.
In particular, for every Lipschitz function $f:X\to\mathbb H$ there exists a constant $C_i=C(f,\lambda_i)>0$ depending on $\mathrm{Lip}f$,  $\lambda_i$, $\mathbb{G}_i$, and its homogeneous norm $\lVert\cdot\rVert_i$ such that for $\mu$-almost every $x\in U_i$ we have $\lVert Df(x)\rVert\leq C_i$.
\end{proposizione}

\begin{proof}
Let $(U_i,\varphi_i)$ be the charts of $(X,d,\mu)$ that make it a Pansu differentiability space. As remarked above, see \cref{rem:Importante}, the space $(X,d,\mu)$ with the charts $(U_i,\pi_{1,i}\circ \varphi_{1,i})$ is a Lipschitz differentiability space and hence, thanks to \cite[Lemma~3.7]{BateJAMS} we can decompose each $U_i$ into countably many Borel sets $U_i^j$ and a $\mu$-null set $N$ such that  $U_i=\bigcup_{j\in\N} U_i^j\cup N$ , the charts $\{(U_i^j,\varphi_i)\}_{i,j\in\N}$ clearly satisfy items 1, 2 and 3 of Definition~\ref{def:PansuDifferentiabilitySpace} and they satisfy \eqref{lambdastructuredchart} for some $\lambda_i>0$.

Moreover, since $(X,d,\mu)$ is a Lipschitz differentiability space with the charts $\{(U_i^j,\pi_{1,i}\circ\varphi_i)\}_{i,j\in\N}$, see \cref{rem:Importante}, we know that for every $i,j\in\mathbb N$, and for every Lipschitz function $g:X\to \R^{n_1(\mathbb H)}$ and $\mu$-almost every $x_0\in U_i^j$  we have there exists an $L_g(x_0):\mathbb{R}^{n_1(\mathbb{G}_i)}\to\R^{n_1(\mathbb{H})}$ such that 
\begin{equation}\label{eqn:lipPANSUDIFF}
        \limsup_{X\ni x\to x_0}\frac{\big |g(x)-g(x_0)- L_g(x_0)[\pi_{1,i}\circ\varphi_i(x)-\pi_{1,i}\circ\varphi_i(x_0)]\big |}{d(x_0,x)}=0.
    \end{equation}
and thanks to \cite[Lemma~3.4]{BateJAMS} we know that $\lVert L_g(x_0)\rVert \leq \mathrm{Lip}(g,x_0)/\lambda_i$. On the other hand for every $x_0\in U_i^j$ and every Lipschitz map $f:X\to\mathbb H$
\begin{equation}
\begin{split}
       0= &\limsup_{X\ni x\to x_0}\frac{\big\|\big(Df(x_0)(\varphi_i(x_0)^{-1}\ \varphi_i(x))\big)^{-1}\  f(x_0)^{-1}\  f(x)\big\|_{\mathbb H}}{d(x_0,x)}\\
       \geq &\limsup_{X\ni x\to x_0}\frac{\big |\pi_{1,\mathbb{H}}\circ f(x)-\pi_{1,\mathbb{H}}\circ f(x_0)-\pi_{1,\mathbb{H}}(Df(x_0)(\varphi_i(x_0)^{-1}\ \varphi_i(x)))\big |}{d(x_0,x)}\\
       =&\limsup_{X\ni x\to x_0}\frac{\big |\pi_{1,\mathbb{H}}\circ f(x)-\pi_{1,\mathbb{H}}\circ f(x_0)-\pi_{1,\mathbb{H}}\circ Df(x_0)(\pi_{1,i}\circ \varphi (x)-\pi_{1,i}\circ \varphi_i(x_0))\big |}{d(x_0,x)},
\end{split}
    \end{equation}
    where the last identity above is a consequence of the homogeneity of $Df(x_0)$. Finally, by the uniqueness of the differential, we infer that $L_{\pi_{1,\mathbb{H}}\circ f}(x_0)=\pi_{1,\mathbb{H}}\circ Df(x_0)\circ \pi_{1,i}=Df(x_0)\circ \pi_{1,i}$ and thus 
    $$
    \lVert Df(x_0)\circ \pi_{1,i}\rVert\leq \mathrm{Lip}(f,x_0)/\lambda_i.
    $$
    Thanks to Proposition~\ref{propdecomposizione}, there exist an absolute constant $M_i\in \N$ depending only on $\mathbb{G}_i$ and a constant $c_{0,i}>0$ depending only on $\mathbb{G}_i$ and on the norm $\lVert\cdot\rVert_i$ on it,
    such that for every $v\in\mathbb{G}$ we can find $\mu_1,\ldots,\mu_M>0$ such that $v=\delta_{\mu_1}(e_{i_1})\cdot\ldots\cdot\delta_{\mu_M}(e_{i_M})$ and $| \mu_{j}|\|e_{i_j}\|_i\leq c_{0,i}\lVert v\rVert_i$ for every $j=1,\dots,M$. Thus
    \begin{equation}
        \lVert Df(x_0)[v]\rVert\leq \sum_{\ell=1}^M\lVert Df(x_0)[\delta_{\mu_{\ell}}e_{i_\ell}]\rVert=\sum_{\ell=1}^M|\mu_{\ell}|\lVert Df(x_0)[\pi_{1,i}(e_{i_\ell})]\rVert\leq M_ic_{0,i}\|v\|_i\mathrm{Lip}(f,x_0)/\lambda_i.
    \end{equation}
\end{proof}
}

\begin{proposizione}\label{propmisuDf}
Suppose $(X,d,\mu)$ is a Pansu differentiability space with respect to $\mathbb{H}$ and suppose $(\varphi,U)$ is one of its charts with $\varphi:X\to \mathbb{G}$ satisfying the conclusion of Proposition~\ref{propcartestrutturate}. Let $f:X\to\mathbb H$ be Lipschitz.
Then, the map $x\mapsto Df(x)$ on $U$ is a $\mu$-measurable map, where we have endowed the space of homogeneous homomorphisms $\mathrm{HHom}(\mathbb{G},\mathbb{H})$ with the distance 
$$d_\mathrm{HHom}(L_1,L_2):=\sup_{\lVert g\rVert_{\mathbb G}= 1}d_{\mathbb H}(L_1(g),L_2(g)).$$
In addition, as $k\in\mathbb N$, the maps 
$$
U\ni x_0\mapsto R_k(x_0):=\sup_{d(x_0,x)\leq k^{-1}}\frac{\lVert Df(x_0)[\varphi(x_0)^{-1}\varphi(x)]^{-1}f(x_0)^{-1}f(x)\rVert_\mathbb{H}} {d(x,x_0)},
$$
are measurable and they pointwise converge to $0$ $\mu$-almost everywhere on $U$ as $k\to+\infty$.
\end{proposizione}

\begin{proof}
Denote by $\tilde{U}$ a Borel subset of $U$ of full $\mu$-measure where \eqref{eqn:PANSUDIFF} is satisfied. The map $\Xi:\tilde{U}\times \mathrm{HHom}(\mathbb{G},\mathbb{H})\to \tilde{U}\times \mathrm{HHom}(\mathbb{G},\mathbb{H})\times \R$ defined as 
$$
\Xi(x_0,L):=\Big(x_0,L,\lim_{k\to \infty}\sup_{d(x_0,x)\leq k^{-1}}\frac{\lVert L(\varphi(x_0)^{-1}\varphi(x))^{-1}f(x_0)^{-1}f(x)\rVert_\mathbb{H}} {d(x,x_0)}\Big),
$$
is Borel. Indeed, since for every $j,k\geq 1$ and for every $L\in \mathrm{HHom}(\mathbb G,\mathbb H)$ the function 
$$
\tilde U\ni x_0\mapsto \mathfrak{R}_{j,k}(x_0):=\sup_{j^{-1}k^{-1}\leq d(x_0,x)\leq k^{-1}}\frac{\lVert L(\varphi(x_0)^{-1}\varphi(x))^{-1}f(x_0)^{-1}f(x)\rVert_\mathbb{H}} {d(x,x_0)},$$
is lower semicontinuous on $X$, it is immediate to see that $\Xi$ is the pointwise limit of this sequence of Borel functions, and thus, it is Borel. In addition, its image is a Borel set since it is the graph of the Borel function $x_0\mapsto\lim_{k\to\infty}\sup_{j\in\N}\mathfrak{R}_{j,k}(x_0)$. Thanks to the choice of $\tilde{U}$, for every $j\in\N$ we see that the projection of $\mathrm{im}(\Xi)\cap \tilde{U}\times \mathrm{HHom}(\mathbb{G},\mathbb{H})\times [0,j^{-1})$ onto $\tilde{U}$ is a set of full $\mu$-measure and hence thanks to Jankov-von Neumann Selection Theorem, see
\cite[Theorem~18.1]{Kechris}, we conclude that, for every $j$, we can find a Borel subset $\tilde{\tilde{U}}$ of full $\mu$-measure in $\tilde U$, and a $\mu$-measurable map $\mathfrak{D}_j:\tilde{\tilde{U}}\to \mathrm{HHom}(\mathbb{G},\mathbb{H})$ such that for every $x_0\in\tilde{\tilde{U}}$ 
$$\lim_{k\to \infty}\sup_{d(x,x_0)\leq k^{-1}}\frac{\lVert \mathfrak{D}_j(x_0)[\varphi(x_0)^{-1}\varphi(x)]^{-1}f(x_0)^{-1}f(x)\rVert_\mathbb{H}} {d(x,x_0)}\leq j^{-1}.$$
Since we have assumed that \eqref{eqn:PANSUDIFF} holds at the points of $\tilde{\tilde{U}}$, we infer that 
\begin{equation}
    \begin{split}
      j^{-1}\geq &  \lim_{k\to \infty}\sup_{d(x,x_0)\leq k^{-1}}\frac{\lVert \mathfrak{D}_j(x_0)[\varphi(x_0)^{-1}\varphi(x)]^{-1}f(x_0)^{-1}f(x)\rVert_\mathbb{H}} {d(x,x_0)}\\
      \geq &\lim_{k\to \infty}\sup_{d(x,x_0)\leq k^{-1}}\frac{\lVert\mathfrak{D}_j(x_0)[\varphi(x_0)^{-1}\varphi(x)]^{-1}Df(x_0)[\varphi(x_0)^{-1}\varphi(x)]\rVert_{\mathbb H}} {d(x,x_0)}\\
      &\qquad-\lim_{k\to \infty}\sup_{d(x,x_0)\leq k^{-1}}\frac{\lVert Df(x_0)[\varphi(x_0)^{-1}\varphi(x)]^{-1}f(x_0)^{-1}f(x)\rVert_\mathbb{H}} {d(x,x_0)}\\
      \geq &\max_{1\leq \ell\leq n_1(\mathbb{G})}\limsup_{k\to \infty}\frac{\lVert\mathfrak{D}_j(x_0)[\varphi(x_0)^{-1}\varphi(x_\ell(k))]^{-1}Df(x_0)[\varphi(x_0)^{-1}\varphi(x_\ell(k))]\rVert}{d(x_\ell(k),x_0)}
    \end{split}
\end{equation}
where $\{x_\ell(k)\}_{\ell=1,\ldots,n_1(\mathbb{G})}$ is the sequence of points given by Proposition~\ref{propcartestrutturate}. Without loss of generality we can assume that the vectors $\delta_{1/d(x_\ell(k),x_0)}(\varphi(x_0)^{-1}\varphi(x_\ell(k)))$ converge to $\omega_\ell\in B(0,\mathrm{Lip}(\varphi)(x_0))$ for every $\ell\in\{1,\ldots,n_1(\mathbb{G})\}$ and 
\begin{equation}
    \max_{1\leq \ell\leq n_1(\mathbb{G})}\langle v,\pi_1(\omega_\ell)\rangle\geq \lambda\qquad\text{for every  }v\in\mathbb{S}^{n_1(\mathbb{G})-1}.
    \label{eq:basisgoodlambda}
\end{equation}
Hence 
$$j^{-1}\geq \max_{1\leq \ell\leq n_1(\mathbb{G})}\lVert\mathfrak{D}_j(x_0)[\omega_\ell]^{-1}Df(x_0)[\omega_\ell]\rVert\geq \mathfrak{C}^{-1}\max_{1\leq \ell\leq n_1(\mathbb{G})}\lvert \mathfrak{D}_j(x_0)[\pi_1(\omega_\ell)]-Df(x_0)[\pi_1(\omega_\ell)]\rvert,$$
where $\mathfrak{C}$ is the constant that gives the equivalence between $\lVert \cdot\rVert$ and the box norm, see \cref{smoothnorm}. Thanks to Proposition~\ref{propdecomposizione} and to \eqref{eq:basisgoodlambda} that there exists an absolute constant $M\in \N$ depending only on $\mathbb{G}$ and a constant $c_0>0$ depending only on $\mathbb{G}$, the norm $\lVert\cdot\rVert$ and the vectors $\{\pi_1(\omega_\ell)\}_{\ell=1,\ldots,n_1}$
    such that for every $v\in\mathbb{G}$ we can find $\lambda_1,\ldots,\lambda_M$ such that $$v=\delta_{\lambda_1}(v_1)\cdot\ldots\cdot\delta_{\lambda_M}(v_M),$$
    with $v_1,\ldots,v_M\in\{\pi_1(\omega_1),\ldots,\pi_1(\omega_{n_1(\mathbb{G})})\}$, and $| \lambda_{\ell}|\|v_\ell\|\leq c_0\lVert v\rVert$ for every $\ell=1,\ldots,M$. Notice that, for every $\|v\|=1$, we thus have $|\lambda_\ell|\leq c_0/\max\{\|v_\ell\|\}=:\tilde C$, where the constant $C$ depends on $c_0,x_0$.
     Thus if $L_1,L_2$ in the following denote respectively $Df(x_0),\mathfrak{D}_j(x_0)$, then for every $\lVert v\rVert=1$, we have
\begin{equation}
\begin{split}
     d(L_1(v),L_2(v))=\lVert\delta_{\lambda_M}(L_1(v_M))^{-1}\ldots\delta_{\lambda_1}(L_1(v_1))^{-1}\delta_{\lambda_1}(L_2(v_1))\ldots \delta_{\lambda_M}(L_2(v_M))\rVert\\
     \leq C\sum_{\ell=1}^M\lVert L_1(v_\ell)^{-1}L_2(v_\ell)\rVert^{1/s}\leq C\sum_{\ell=1}^M| L_2(v_\ell)-L_1(v_\ell)|^{1/s^2}\leq C M j^{-1/s^2},
\end{split}
\end{equation}
where the constant $C$ above may change from line to line, and it is a constant depending on $M$, $c_0$, $\lambda$, $\tilde C$, and $\|Df(x_0)\|$ given by repeatedly applying \cite[Lemma~2.1]{antonelli2020rectifiableA}.
Since the above bound does not depend on $v$ such that $\|v\|=1$ (but it might depends on $x_0$) we infer, thanks to this, that for $\mu$-almost every $x_0\in\tilde U$, $\mathfrak{D}_j(x_0)$ is defined for all $j$ and converge pointwise to $Df(x_0)$. This concludes the proof of the fact that $U\ni x\mapsto Df(x)$ is a $\mu$-measurable map. 

In addition, arguing as above, we know that the maps 
$$U\ni x_0\mapsto R_k(x_0)=\sup_{d(x,x_0)\leq k^{-1}}\frac{\lVert Df(x_0)[\varphi(x_0)^{-1}\varphi(x)]^{-1}f(x_0)^{-1}f(x)\rVert_\mathbb{H}} {d(x,x_0)},$$
are $\mu$-measurable since $x_0\mapsto Df(x_0)$ is $\mu$-measurable and they pointwise converge to $0$ for $\mu$-almost every $x_0\in U$ by the definition of Pansu differentiability space. 
\end{proof}

The following proposition is a standard use of Lusin and Severini--Egoroff's theorem.
\begin{proposizione}\label{diffunif}
Suppose $(X,d,\mu)$ is a Pansu differentiability space with respect to $\mathbb{H}$, and suppose $(\varphi,U)$ is one of its charts with $\varphi:X\to \mathbb{G}$ satisfying the conclusion of Proposition~\ref{propcartestrutturate}. Let $f:X\to\mathbb H$ be Lipschitz. Then, for every $\eta>0$ there exists a compact set $\tilde{U}\subseteq U$ such that $\mu(U\setminus \tilde{U})\leq \eta\mu(U)$ and the following holds. For every $\varepsilon>0$ there exists a $\delta>0$ such that 
$$\frac{\lVert Df(x_0)[\varphi(x)^{-1}\varphi(y)]^{-1}f(x)^{-1}f(y)\rVert_\mathbb{H}} {d(x,y)}\leq \varepsilon,$$
for every $x_0\in\tilde U$, $x\in \tilde{U}\cap B(x_0,\delta)$, and every $y\in B(x,\delta)$.
\end{proposizione}

\begin{proof}
Since we have assumed that $\mu(U)<\infty$, thanks to Proposition~\ref{propmisuDf}, Severini-Egoroff's and Lusin's theorem we know that for every $\eta>0$ there exists a compact subset $\tilde{U}$ of $U$ such that $\mu(U\setminus \tilde U)\leq \eta\mu (U)$, and
\begin{itemize}
    \item[(i)] for every $\varepsilon>0$ there exists a $k_0$ such that 
$$\sup_{d(x_0,x)\leq k^{-1}}\frac{\lVert Df(x_0)[\varphi(x_0)^{-1}\varphi(x)]^{-1}f(x_0)^{-1}f(x)\rVert_\mathbb{H}} {d(x,x_0)}\leq \varepsilon$$ 
for every $k\geq k_0$ and every $x_0\in\tilde{U}$;
\item[(ii)] for every $\varepsilon>0$ there exists a  $\delta>0$ such that $$\text{$d_{\mathrm{HHom}}(Df(x_0),Df(x))\leq \varepsilon$ for every $x_0\in \tilde U$ and $x\in B(x_0,\delta)\cap \tilde{U}$. }$$
\end{itemize}
Let us fix $\varepsilon>0$ and define $\bar\delta:\min\{\delta,k_0^{-1},1\}/10$, where $\delta,k_0$ are chosen as above and depend on $\varepsilon$. Thus for every $x_0\in\tilde U$, for every $x\in B(x_0,\bar{\delta})\cap \tilde U$, and $y\in B(x,\bar{\delta})$ we have 
\begin{equation}
\begin{split}
      &\qquad\qquad\frac{\lVert Df(x_0)[\varphi(x)^{-1}\varphi(y)]^{-1}f(x)^{-1}f(y)\rVert_\mathbb{H}} {d(x,y)}\\
    &\leq d_{\mathrm{HHom}}(Df(x),Df(x_0))\mathrm{Lip}(\varphi)+\frac{\lVert Df(x)[\varphi(x)^{-1}\varphi(y)]^{-1}f(x)^{-1}f(y)\rVert_\mathbb{H}} {d(x,y)}\\
    &\leq \varepsilon(1+\mathrm{Lip}(\varphi)).
\end{split}
\end{equation}
This concludes the proof.
\end{proof}

The previous strict differentiability in \cref{diffunif} allows us to show that in a Pansu differentiability space with respect to $\mathbb H$ we can choose {\em good charts} $(U,\varphi)$, where  $\varphi: X\to\mathbb G$, in the sense that whenever $f: X\to\mathbb H$ is Lipschitz and $U\ni x_i,y_i\to x_0\in U$ is such that 
\[
\frac{\|\varphi(x_i)^{-1}\varphi(y_i)\|_{\mathbb G}}{d(x_i,y_i)}\to 0,
\]
hence 
\[
\frac{\|f(x_i)^{-1}f(y_i)\|_{\mathbb H}}{d(x_i,y_i)}\to 0.
\]
We give a more quantitative version of this statement in the following \cref{prop:GoingTo0InChart}, which leads to the definition of $(\mathbb G,\mathbb H)$-completeness, see \cref{def:GHcomplete}.

In the following proposition we prove that the charts of Pansu differentiability spaces have the following property. Whenever a chart $\varphi$ collapses around a point, the same happens for every Lipschitz function $X\to\mathbb H$ around the same point.

\begin{proposizione}\label{prop:GoingTo0InChart}
Suppose $(X,d,\mu)$ is a Pansu differentiability space with respect to $\mathbb{H}$. Let $(\varphi,U)$ be a chart as in Proposition~\ref{propcartestrutturate}. Let $f:X\to\mathbb H$ be Lipschitz. Then, there exists a constant $C>0$ such that for every $\eta>0$ there exists a compact set $\tilde{U}\subseteq U$
for which
\begin{itemize}
    \item[(i)]$\mu(U\setminus \tilde{U})\leq \eta\mu(U)$;
    \item[(ii)] for every $\varepsilon>0$ there exists a $\rho>0$ such that if $x_0\in\tilde U$, $x\in B(x_0,\rho)\cap \tilde{U}$ are such that $\lVert\varphi(x)^{-1}\varphi(y)\rVert_{\mathbb G} \leq \varepsilon d(x,y)$ for some $y\in B(x,\rho)$, then
    $$
    \lVert f(x)^{-1}f(y)\rVert_{\mathbb H}\leq C\varepsilon d(x,y).
    $$
\end{itemize}
\end{proposizione}

\begin{proof}
Thanks to Proposition~\ref{diffunif}, we know that there exists a compact set $\tilde{U}\subseteq U$ such that $\mu(U\setminus \tilde{U})\leq \eta\mu(U)$ and for every $\varepsilon>0$ there exists a $\rho>0$ such that 
$$
\frac{\lVert Df(x_0)[\varphi(x)^{-1}\varphi(y)]^{-1}f(x)^{-1}f(y)\rVert_\mathbb{H}} {d(x,y)}\leq \varepsilon,
$$
for every $x_0\in\tilde U$, every $x\in \tilde{U}\cap B(x_0,\rho)$, and every $y\in B(x,\rho)$.
If however $\lVert\varphi(x)^{-1}\varphi(y)\rVert\leq \varepsilon d(x,y)$ for  some $y\in B(x,\rho)$, then 
\begin{equation}
    \lVert f(x)^{-1}f(y)\rVert\leq \lVert Df(x_0)\rVert\lVert \varphi(x)^{-1}\varphi(y)\rVert+ \varepsilon d(x,y)\leq \varepsilon(1+\mathrm{Lip}(\varphi)C(f,U))d(x,y),
\end{equation}
where here the constant $C(f,U)$ is that given by the second half of Proposition~\ref{propcartestrutturate}. This concludes the proof of the proposition.
\end{proof}

{
\begin{definizione}[$(\mathbb G,\mathbb H)$-completeness]\label{def:GHcomplete}
Let $\mathbb G,\mathbb H$ be Carnot groups. Let $\varphi:X\to\mathbb G$ be Lipschitz and let $U\subset X$ be Borel. Let $\mathcal{F}$ a collection of Lipschitz functions from $X$ to $\mathbb H$. In the following, we say that the pair $(U,\varphi)$ is  $(\mathbb{G},\mathbb{H})$-\emph{complete with respect to $\mathcal{F}$} if the following holds.
For every Lipschitz map $f\in \mathcal{F}$, if there exist $x_i,y_i\to x_0\in U$ such that $x_i\in U$ and
\[
\lim_{i\to +\infty}\frac{\|\varphi(x_i)^{-1}\varphi(y_i)\|_{\mathbb G}}{d(x_i,y_i)}= 0,
\]
then 
\[
\lim_{i\to +\infty}\frac{\|f(x_i)^{-1}f(y_i)\|_{\mathbb H}}{d(x_i,y_i)}= 0.
\]
\end{definizione}

In the following proposition, we show that one can choose the atlas of a Pansu differentiability space in such a way that every chart is complete with respect to a given Lipschitz function. Moreover, we remark here that one cannot take the atlas to be made of complete charts for \textit{every} Lipschitz function, see \cref{rem:NonTuttoCompleto}.

\begin{proposizione}\label{rem:GoingTo0}
Let $(X,d,\mu)$ be a Pansu differentiability space with respect to $\mathbb{H}$, and let $f: X\to\mathbb H$ be a Lipschitz function. Then, for some Carnot groups $\mathbb G_i$, $X$ admits $(\mathbb{G}_i,\mathbb{H})$-\emph{complete charts with respect to $f$} satisfying items 1, 2 and 3 of Definition~\ref{def:PansuDifferentiabilitySpace}.
\end{proposizione}

\begin{proof}
 Let us take $\{(U_i,\varphi_i)\}_{i\in\mathbb N}$ as in \cref{def:PansuDifferentiabilitySpace}. By subsequently using \cref{propcartestrutturate}, and \cref{prop:GoingTo0InChart} we can find, for every $i\in\mathbb N$, compact sets $\{U_i^j\}_{j\in\mathbb N}$ such that $\mu(U_i\setminus\cup_{j\in\mathbb N} U_i^j)=0$, item (ii) of \cref{prop:GoingTo0InChart} holds with $U_i^j$, and items 1, 2, 3 of \cref{def:PansuDifferentiabilitySpace} hold with charts $\{(U_i^j,\varphi_i)\}_{i,j}$. Item (ii) of \cref{prop:GoingTo0InChart} directly implies $(\mathbb G_i,\mathbb H)$-completeness of $(U_i^j,\varphi_i)$ with respect to $f$.
\end{proof}
}

\subsubsection{Alberti representations}\label{sec:Alberti}

In this section, we recall the basic definitions about Alberti representations on a metric measure space $(X,d,\mu)$. We define the \textit{horizontal} speed of a representation, the notion of independent and universal representations, and the notion of separated cones.

\begin{definizione}[The space $\Gamma(X)$ of fragments]\label{lipcurvez}
Let $(X,d)$ be a metric space. Let $\Gamma(X)$ be the set of biLipschitz embeddings $\gamma:\mathrm{Dom}(\gamma)\to X$, with $\mathrm{Dom}(\gamma)\subset \mathbb R$ non-empty and compact such that 
$$
2^{-1}\lvert s-t\rvert\leq d(\gamma(s),\gamma(t))\leq 2\lvert s-t\rvert\qquad\text{for every $s,t\in \mathrm{Dom}(\gamma)$}.
$$
\end{definizione}

\begin{definizione}[Definition of $A(X)$ and $\mathfrak{D}(K)$]\label{def:curvespace}
For a metric space $(X,d)$ we define $H(X)$ to be
the collection of non-empty compact subsets of the metric space $\mathbb R\times X$
with the Hausdorff metric induced by the Euclidean product distance, so that $H(X)$ is complete and separable. 
In the following we shall identify every element $\gamma\in \Gamma(X)$
with its graph
$$
\mathrm{gr}(\gamma):=\{(t,\gamma(t))\in \R\times X: t\in \mathrm{Dom}(\gamma)\}.
$$
We
also identify $\Gamma(X)$ with its isometric image in $H(X)$ via
$\gamma\mapsto\mathrm{gr}(\gamma)$ and set
\[A(X)=\{(x,\gamma)\in X\times\Gamma(X):\exists\ t\in\mathrm{Dom}(\gamma),\ \gamma(t)=x\}.\]
Finally, for every $K\subset X$, we define the set
\[\mathfrak{D}(K):=\{(x,\gamma)\in A(X): \gamma^{-1}(x) \text{ is a density
  point of } \gamma^{-1}(K)\}.\]
\end{definizione}

\begin{definizione}[Alberti representation]\label{def:AlbertiRepresentation}
Let $(X,d,\mu)$ be a metric measure space. For every measurable set $A\subset X$ we say that $ \mathcal{A}=(\{\mu_\gamma\}_{\gamma\in\Gamma(X)},\mathbb P)$ is an {\em Alberti representation of} $\mu\llcorner A$ if $\mathbb P$ is a probability measure on $\Gamma(X)$ and
\begin{itemize}
    \item[(i)] for every $\gamma\in \Gamma(X)$ the measure $\mu_\gamma$ is absolutely continuous with respect to $\mathcal{H}^1\llcorner \gamma$;
    \item[(ii)] for every Borel set $Y\subset A$, the curve $\gamma\to \mu_\gamma(Y)$ is Borel measurable and
\[
\mu(Y)=\int_{\Gamma(X)}\mu_\gamma(Y)d\mathbb P.
\]
\end{itemize}
\end{definizione}

Let us fix now a Carnot group $\mathbb G$ with stratification
\[
\mathfrak g=V_1\oplus\dots\oplus V_s,
\]
and recall that $\pi_1$ is the projection of $\mathfrak g$ onto $V_1$. After choosing an adapted basis for $\mathfrak g$, and up to identifying $\mathbb G$ with $\mathbb R^n$ through this basis, we denote with $\langle\cdot,\cdot\rangle$ the standard Euclidean scalar product, and with $|\cdot|$ the standard Euclidean norm.

\begin{definizione}[$C$-curves]
\label{C-curves}
Let $e\in V_1$ be a unit vector and $\sigma\in (0,1)$. We denote by $C(e,\sigma)$ the one-sided, closed, convex cone with axis $e$ and opening $\sigma$, namely\label{s-cones}
$$C(e,\sigma):=\{x\in V_1:\langle x,e\rangle\geq(1-\sigma^2)|x|\}.$$
Let $B$ be a bounded Borel subset of the real line. A Lipschitz curve $\gamma:B\to \mathbb{G}$, is said to be a {\em $C(e,\sigma)$-curve} if
   \begin{equation}
       \text{ $\pi_1(\gamma(s))-\pi_1(\gamma(t))\in C(e,\sigma)\setminus\{0\}$ for every $t,s\in B$ with $t<s$.}
       \label{conecondition}
   \end{equation}
\end{definizione}

\begin{definizione}[Independent Alberti representations]\label{def:IndependentAlberti}
Let $\gamma\in\Gamma(X)$, and let $\varphi:X\to\mathbb G$ be Lipschitz. Let $C(e,\sigma)$ be as in \cref{C-curves}. We say that {\em $\gamma$ goes in the $\varphi$-direction of $C(e,\sigma)$} if $\varphi\circ\gamma$ is a $C(e,\sigma)$-curve.

We say that an Alberti representation $\mathcal{A}:=(\{\mu_\gamma\}_{\gamma\in\Gamma(X)},\mathbb P)$ {\em is in the $\varphi$-direction of $C(e,\sigma)$} if $\mathbb P$-almost every $\gamma\in\Gamma(X)$ goes in the $\varphi$-direction of $C(e,\sigma)$.

We say that the cones $C(e_1,\sigma_1),\dots,C(e_\ell,\sigma_\ell)$ are {\em independent} if every $v_1,\dots,v_\ell\in V_1$ such that $v_i\in C(e_i,\sigma_i)\setminus\{0\}$ for every $i=1,\dots,\ell$ are linearly independent. 

We say that $\ell$ Alberti representations $\mathcal{A}_1,\dots,\mathcal{A}_\ell$ are {\em $\varphi$-independent} if there exist $\ell$ independent cones $C(e_1,\sigma_1),\dots,C(e_\ell,\sigma_\ell)$ such that for every $i=1,\dots,\ell$ we have that $\mathcal{A}_i$ goes in the $\varphi$-direction of $C(e_i,\sigma_i)$.
\end{definizione}

\begin{definizione}[Horizontal speed of a representation]\label{def:horizontalspeed}
Let $(X,d)$ be a metric space and $\varphi:X\to\mathbb G$ be Lipschitz. Let $\delta>0$. We say that $\gamma\in \Gamma(X)$ has {\em horizontal $\varphi$-speed $\delta$} if for almost every $t_0\in\mathrm{Dom}(\gamma)$ we have that
\[
\|D(\varphi\circ\gamma)(t_0)\|_{\mathbb G} \geq \delta \mathrm{Lip}(\pi_{H,\mathbb G}\circ\varphi,\gamma(t_0))\mathrm{Lip}(\gamma,t_0).
\]
Recall that $Df$ stands for the Pansu differential of the function $f$, which exists for almost every $t_0\in\mathrm{Dom}(\gamma)$ due to the Pansu--Rademacher theorem. We say that an Alberti representation $\mathcal{A}=(\{\mu_\gamma\}_{\gamma\in\Gamma(X)},\mathbb P)$ {\em has horizontal $\varphi$-speed $\delta$} if for $\mathbb P$-almost every $\gamma\in \Gamma(X)$, $\gamma$ has horizontal $\varphi$-speed $\delta$.
\end{definizione}

\begin{definizione}[Universal representation]\label{def:UniversalRep}
Let $(X,d)$ be a metric space with $U\subset X$. Let $\varphi:X\to\mathbb G$ be Lipschitz. Let $\mathcal{A}_1,\dots,\mathcal{A}_n$ be $\varphi$-independent Alberti representations of $\mu\llcorner U$, each with strictly positive horizontal $\varphi$-speed. 
Fix $\rho>0$. We say that $\{\mathcal{A}_1,\dots,\mathcal{A}_n\}$ is {\em horizontally $\rho$-universal with respect to} $\mathbb H$ if for every Lipschitz $f:X\to\mathbb H$ there exists a decomposition $X=X_1\cup\dots\cup X_n$ in Borel sets such that the Alberti representation of $\mu\llcorner (U\cap X_i)$ induced by $\mathcal{A}_i$ has horizontal $f$-speed $\rho$. Notice that if $\mathbb G=\mathbb R^n$, $\mathbb H=\mathbb R$ this definition coincides with \cite[Definition 7.1]{BateJAMS}.
\end{definizione}

We recall the following, which is \cite[Definition 7.3]{BateJAMS}.
\begin{definizione}[Separated cones]\label{coniseparati}
For $\xi>0$, we say that $v_1,\dots,v_m\in \mathbb R^n$ are {\em $\xi$-separated} if for every $\lambda\in\mathbb R^m\setminus\{0\}$ we have 
\[
\left|\sum_{i=1}^m\lambda_i v_i\right| > \xi \max_{1\leq i\leq m} |\lambda_i v_i|.
\]
We say that cones $C_1,\dots,C_\ell$ are $\xi$-separated if every $v_i\in C_i\setminus\{0\}$ are $\xi$-separated. A finite collection of Alberti representation $\mathcal{A}_i$ is {\em $\xi$-separated} with respect to $\varphi$ if there are $\xi$-separated cones $C_i$ such that $\mathcal{A}_i$ is in the $\varphi$-direction of $C_i$.
\end{definizione}

\subsubsection{Horizontal differentiability}

This section aims to prove that whenever we have $n_1(\mathbb G)$ horizontally $\rho$-universal (with respect to $\mathbb H$) independent Alberti representations on $(X,d,\mu)$, then $X$ is a Euclidean Lipschitz differentiable space, and we can define a good notion of horizontal differential almost everywhere in chart, see \cref{prophordiff}. Moreover, we can relate the derivative of a Lipschitz function $f:X\to\mathbb H$ along a curve $\gamma$ with the derivative of the chart $\varphi:X\to\mathbb G$ along $\gamma$ by means of this horizontal differential. The result of this section relies in a fundamental way on the main result in \cite[Proposition~7.8]{BateJAMS} by D. Bate. 

In this section $\mathbb G,\mathbb H$ will be  Carnot groups endowed with homogeneous norms $\|\cdot\|_{\mathbb G}$, $\|\cdot\|_{\mathbb H}$ in such a way that they agree with some Euclidean norm on the first stratum and the horizontal projections $\pi_\mathbb G,\pi_\mathbb H$ are $1$-Lipschitz, see, e.g., \cref{smoothnorm}. 

\begin{definizione}[Horizontal differential $D_Hf$]\label{def:HorDiff}
Suppose $(X,d,\mu)$ is a metric measure space. Suppose $\varphi:X\to \mathbb{G}$ is Lipschitz.
For every $f\in \mathrm{Lip}(X,\mathbb{H})$ we say that {a linear map $D_Hf(x):V_1(\mathbb{G})\to V_1(\mathbb H)$} is the \emph{horizontal Pansu differential} of $f$ at $x\in X$  with respect to $\varphi$, if for every $\gamma\in\Gamma(X)$ for which $t_0:=\gamma^{-1}(x)$ is a density point for $\mathrm{Dom}(\gamma)$, and for which $\varphi\circ \gamma$ and $f\circ \gamma$ are Pansu differentiable at $t_0$, we have
\begin{equation}\label{eqn:HorizDiff}
    0=\lim_{t\to 0}\frac{\lVert \pi_{\mathbb H}  ( (D_Hf(x)[\pi_{\mathbb{G}}(\varphi(\gamma(t_0+t)))-\pi_{\mathbb{G}}(\varphi(x))])^{-1}f(x)^{-1}f(\gamma(t_0+t)))\rVert_{\mathbb{H}}}{d(x,\gamma(t_0+t))}.
    \nonumber
\end{equation}
\end{definizione}

\begin{proposizione}\label{prophordiff}
Let $(X,d,\mu)$ be a metric measure space and let $U$ be a Borel subset of $X$. 
Suppose that $\varphi:X\to \mathbb{G}$ is a Lipschitz map and that $\mu\llcorner U$ has a horizontally $\rho$-universal family of $\varphi$-independent representations $\{\mathcal{A}_1,\ldots,\mathcal{A}_{n_1}\}$ with respect to $\mathbb{H}$. Then, the following hold.
\begin{itemize}
    \item[(i)] Every Lipschitz map $f\in \mathrm{Lip}(X,\mathbb{H})$ has a horizontal Pansu differential {with respect to $\varphi$} at $\mu$-almost every $x\in U$.
    \item[(ii)] $(U,d,\mu)$ is a $n_1(\mathbb G)$-Lipschitz differentiability space, 
    and in particular $\mu$ is asymptotically doubling.
    \item[(iii)] For $\mu$-almost every $x\in U$ the following holds. If $\gamma\in\Gamma(X)$, $t_0:=\gamma^{-1}(x)$ is a density point of $\mathrm{Dom}(\gamma)$, and the values $D(f\circ\gamma)(t_0)$ and $D(\varphi\circ\gamma)(t_0)$ exist, then we have 
\begin{equation}\label{eqn:LAltraEquazione}
    D( f\circ\gamma)(t_0) = D_Hf(x)\big(D(\varphi\circ \gamma)(t_0)\big).
\end{equation} 
\end{itemize} 
\end{proposizione}

\begin{proof}
Let us first prove that the representations $\{\mathcal{A}_1,\ldots,\mathcal{A}_{n_1}\}$ are $(\pi_{\mathbb{G}}\circ \varphi)$-independent  and $\rho$-universal with respect to $\mathbb R$. Let us fix $f:X\to\mathbb H$. The preliminary fact to prove is the identity 
$$
(\pi_{\mathbb{H}}\circ f\circ\gamma)'(s)=D(f\circ \gamma)(s),
$$
at all the points $s$ where the Pansu derivative $D(f\circ\gamma)(s)$ exists.
In order to see why this is the case, notice that, by definition of Pansu differential, there exists a $v\in V_1(\mathbb{H})$ such that
$$v=\lim_{t\to 0}\delta_{1/t}((f\circ\gamma)(s)^{-1}(f\circ\gamma)(s+t)).$$
This shows in particular by the continuity of $\pi_{\mathbb{H}}$ that 
\begin{equation}
    \begin{split}
        D(f\circ \gamma)(s)&=v=\pi_{\mathbb{H}}(v)\\
        &=\lim_{t\to 0}\frac{(\pi_{\mathbb{H}}\circ f\circ\gamma)(s+t)-(\pi_{\mathbb{H}}\circ f\circ\gamma)(s)}{t}=(\pi_{\mathbb{H}}\circ f\circ\gamma)'(s),
        \label{pansudifflungocurve}
    \end{split}
\end{equation}
where in the second equality above we are using that $D(f\circ\gamma)(s)$ is a horizontal vector. The above computation, together with the horizontal universality, shows therefore that, up to possibly decompose the space on pieces on which the induced Alberti representations have horizontal $f$-speed $\rho$, we have that for every $i$, $\mathbb P_i$-a.e. $\gamma\in\Gamma(X)$, where $\mathbb P_i$ is the probability measure associated to $\mathcal{A}_i$, and almost every $t\in\mathrm{Dom}(\gamma)$ we have
$$
\lvert (\pi_{\mathbb{H}}\circ f\circ\gamma)'(t)\rvert = \lVert D(f\circ \gamma)(t)\rVert_{\mathbb{H}}
\geq \rho\mathrm{Lip}(\pi_{\mathbb{H}}\circ f,\gamma(t))\mathrm{Lip}(\gamma,t),
$$
and hence the representations $\{\mathcal{A}_1,\ldots,\mathcal{A}_{n_1}\}$ are $\rho$-universal with respect to $\R$, and are $\pi_\mathbb{G}\circ \varphi$-independent by hypothesis. By \cite[Proposition~7.8]{BateJAMS} this implies that $(U,d,\mu)$ is an $n_1(\mathbb{G})$-Lipschitz differentiability space and hence by \cite[Corollary 8.4]{BateJAMS} the measure $\mu$ is asymptotically doubling. Thus item (ii) is proved. Now by definition of Lipschitz differentiability space, for $\mu\llcorner U$-almost every $x\in U$ there exists a linear map $L_x:\R^{n_1(\mathbb{G})}\to \R^{n_1(\mathbb{H})}$ such that for every $\gamma\in\Gamma(X)$ for which $t_0$ is a density point of $\mathrm{Dom}(\gamma)$, we have
\begin{equation}
    \begin{split}
           & 0=\lim_{t\to 0}\frac{\lvert -L_x[\pi_{\mathbb{G}}(\varphi(\gamma(t_0+t)))-\pi_{\mathbb{G}}(\varphi(x))]-\pi_{\mathbb{H}}\circ f(x)+\pi_{\mathbb{H}}\circ f(\gamma(t_0+t))\rvert}{d(x,\gamma(t_0+t))}\\
    =&\lim_{t\to 0}\frac{\lVert\pi_{\mathbb{H}}\Big( (L_x[\pi_{\mathbb{G}}(\varphi(\gamma(t_0+t)))-\pi_{\mathbb{G}}(\varphi(x))])^{-1}f(x)^{-1} f(\gamma(t_0+t))\Big)\rVert_{\mathbb{H}}}{d(x,\gamma(t_0+t))}.
    \end{split}
    \nonumber
\end{equation}
Thus item (1) is proved.

Let us now fix $x,\gamma$ as in item (iii).
Thanks to the above discussion, by exploiting the triangle inequality, and by exploiting the existence of $D(f\circ\gamma)(t_0)$, and the fact that $\gamma$ is $2$ biLipschitz, we have
\begin{equation}
\begin{split}
    0&=\lim_{t\to 0}\frac{\lVert\pi_{\mathbb{H}}\Big( (L_x[\pi_{\mathbb{G}}(\varphi(\gamma(t_0+t)))-\pi_{\mathbb{G}}(\varphi(x))])^{-1}f(x)^{-1} f(\gamma(t_0+t))\Big)\rVert_{\mathbb{H}}}{d(x,\gamma(t_0+t))}\\
    &\geq \lim_{t\to 0}\frac{\lVert\pi_{\mathbb{H}}\Big( (L_x[\pi_{\mathbb{G}}(\varphi(\gamma(t_0+t)))-\pi_{\mathbb{G}}(\varphi(x))])^{-1}\delta_t(D(f\circ \gamma)(t_0))\Big)\rVert_{\mathbb{H}}}{d(x,\gamma(t_0+t))}\\
    &\qquad\qquad\qquad-\lim_{t\to0}\frac{\lVert \pi_{\mathbb H}\left(  \delta_t(D(f\circ \gamma)(t_0))^{-1}f(x)^{-1} f(\gamma(t_0+t)  \right)\rVert_\mathbb{H}}{d(x,\gamma(t_0+t))}\\
     &\geq \frac{1}{2}\lim_{t\to 0}\frac{\lVert\pi_{\mathbb{H}}\Big( (L_x[\pi_{\mathbb{G}}(\varphi(\gamma(t_0+t)))-\pi_{\mathbb{G}}(\varphi(x))])^{-1}\delta_t(D(f\circ \gamma)(t_0))\Big)\rVert_{\mathbb{H}}}{t}.
    \nonumber
\end{split}
\end{equation}
The above computation implies, in particular, that
\begin{equation}\label{eqn:RightBefore}
    \begin{split}
         \lim_{t\to 0}L_x\left[\frac{\pi_{\mathbb{G}}(\varphi(\gamma(t_0+t)))-\pi_{\mathbb{G}}(\varphi(x))}{t}\right]&=\pi_{\mathbb H}[D( f\circ\gamma)(t_0)]\\
         &=D( f\circ\gamma)(t_0),
    \end{split}
\end{equation}
and this, together with the fact that 
\[
\lim_{t\to 0}\left[\frac{\pi_{\mathbb{G}}(\varphi(\gamma(t_0+t)))-\pi_{\mathbb{G}}(\varphi(x))}{t}\right]=D(\varphi\circ\gamma)(t_0)
\]
concludes the proof of \eqref{eqn:LAltraEquazione}.
\end{proof}
{
\begin{osservazione}\label{rem:NotReallyDerivative}
Assume we are in the setting of \cref{prophordiff}. Let $x\in U$, and take $\gamma\in\Gamma(X)$ for which $\gamma^{-1}(x)$ is a density point of $\mathrm{Dom}(\gamma)$. If the quantity $D(\varphi\circ\gamma)(\gamma^{-1}(x))=(\pi_{\mathbb G}\circ \varphi\circ\gamma)'(\gamma^{-1}(x))$ exists, we cannot infer immediately that $D(f\circ\gamma)(\gamma^{-1}(x))$ exists. Indeed, in \cref{prophordiff} we have only proved the {\em horizontal} differentiability of $f$ in chart, which only implies that $(\pi_{\mathbb H}\circ f\circ \gamma)'(\gamma^{-1}(x))$ exists.

Nevertheless, we proved that $(U,d,\mu)$ is a $n_1(\mathbb G)$-Lipschitz differentiability space, where the charts can be taken to be $(U_i,\pi_{\mathbb G}\circ\varphi)$, where $U_i\subset U$. Hence, if we take an arbitrary Lipschitz function $g:X\to\mathbb R^\ell$, and we assume that $D(\varphi\circ\gamma)(\gamma^{-1}(x))=(\pi_{\mathbb G}\circ \varphi\circ\gamma)'(\gamma^{-1}(x))$ exists, the same computations right before \eqref{eqn:RightBefore} give that $(g\circ\gamma)'(\gamma^{-1}(x))$ exists and it is equal to $Dg(x)(\pi_{\mathbb G}\circ \varphi\circ\gamma)'(\gamma^{-1}(x))$.
\end{osservazione}
}

\subsection{Alberti representations and Borel selections}

This section aims to use measurable selection arguments to prove the following result. Whenever there are $n_1(\mathbb G)$ horizontally $\rho$-universal (with respect to $\mathbb H$) independent Alberti representations on a Borel subset $U$ of $(X,d,\mu)$, and a Lipschitz function $f:X\to\mathbb H$ is given, then for $\mu$-almost every point $x$ in the domain $U$ of the chart $\varphi$ of the representation we can choose curves $\gamma_i$ starting from $x$ for which $D(f\circ\gamma_i)(\gamma_i^{-1}(x)), D(\varphi\circ\gamma_i)(\gamma_i^{-1}(x))$ are defined. Moreover, all the choices can be made in a measurable way; see \cref{constructionvectorfields} for a detailed statement. 

The forthcoming results are generalizations of results by D. Bate for $\R$-valued maps in  \cite{BateJAMS}. 
The proofs are verbatim adaptations, thus we omit them.

The following \cref{lem:measurability} extends \cite[Lemma~2.8]{BateJAMS}.
For the notation $A(X),H(X),\mathfrak{D}(K)$, recall \cref{def:curvespace}. 

\begin{proposizione}\label{lem:measurability}
Let $(X,d,\mu)$ be a metric measure space. Let $\mathbb{H}$ be a Carnot group, and let $e_1\in\mathbb H$ be of norm $1$. Let $f\in \mathrm{Lip}(X,\mathbb{H})$ and let $\mathfrak{G}:A(X)\subset H(X)\to \mathbb{H}$ be the map
\begin{equation}
    \mathfrak{G}_f(x,\gamma):=\begin{cases}
    D(f\circ \gamma)(\gamma^{-1}(x))&\text{if it exists,}\\
    \delta_{4\mathrm{Lip}(f)}(e_1)&\text{otherwise.}
    \end{cases}
    \label{eq:Gf}
\end{equation}
Then $\mathfrak{G}_f$ is Borel measurable.
In addition, for every compact set $K\subseteq X$, we have that $\mathfrak{D}(K)$ is a Borel subset of $A(X)$. Notice that, since $\gamma$ is $2$-Lipschitz, if $D(f\circ\gamma)(\gamma^{-1}(x))$ exists, then it cannot be equal to $\delta_{4\mathrm{Lip}(f)}(e_1)$.
\end{proposizione}

The following \cref{prop:curvefullmeas} is a corollary of \cite[Proposition~2.9]{BateJAMS}.

\begin{proposizione}\label{prop:curvefullmeas}
Let $(X,d,\mu)$ be a metric measure space and $M\subset X$ measurable
such that $\mu\llcorner M$ has an Alberti representation $\mathcal A$ and let $\mathbb{H}$ be a Carnot group. 

Let $B\subset \mathbb{H}$ be Borel and $\mathfrak{G}\colon A(X)\to \mathbb{H}$ be a Borel function such
that, for almost every $\gamma\in\Gamma(X)$ and almost every
$t\in\mathrm{Dom}\gamma$, we have $\mathfrak{G}(\gamma(t),\gamma)\in B$.  Then the set
\begin{equation}\label{eqn:PM}
P(M):=\{x\in M:\exists\ \gamma\in\Gamma(X),\ (x,\gamma)\in \mathfrak{D}(M),\ \mathfrak{G}(x,\gamma)\in B\},
\end{equation}
is a set of full measure in $M$ and for each $x\in P(M)$ there exists a
$\gamma^x\in\Gamma(X)$ that satisfies the conditions given in the
definition of $P(M)$ such that $x\mapsto \gamma^x$ is {$(\mu,\mathfrak{Q}(H(X)))$-measurable, where $\mathfrak{Q}(H(X))$ is the $\sigma$-algebra of Borel sets in $H(X)$, and we recall that we are identifying $\Gamma(X)$ with its isometric image in $H(X)$, see \cref{def:curvespace}.}
\end{proposizione}

\begin{proposizione}\label{constructionvectorfields}
Let $(X,d,\mu)$ be a metric measure space, suppose $\varphi:X\to \mathbb{G}$ and $f:X\to \mathbb{H}$ are Lipschitz maps, and let $U\subseteq X$ be a Borel subset of $X$. Suppose that $\mathcal{A}_1,\ldots,\mathcal{A}_{n_1}$ are a family of horizontally $\rho$-universal and $\varphi$-independent representations of $\mu\llcorner U$ such that each $\mathcal{A}_i$ goes in the $\varphi$-direction of a cone $C(e_i,\sigma_i)$, and these cones are independent. We denote with $\mathfrak{Q}(H(X))$ the $\sigma$-algebra of Borel sets in $H(X)$, and we are identifying $\Gamma(X)$ with its isometric image in $H(X)$.

Then, we can find $(\mu,\mathfrak{Q}(H(X)))$-measurable maps $\chi_1,\ldots,\chi_{n_1}:U\to \Gamma(X)$, defined $\mu$-almost everywhere on $U$, such that for $\mu$-almost every $x\in U$ and every $i=1,\ldots,n_1(\mathbb{G})$ we have, denoting $\chi_{[i,x]}:=\chi_i(x)$, that 
\begin{itemize}
    \item[(i)] the Pansu derivatives $D(\varphi\circ \chi_{[i,x]})(\chi_{[i,x]}^{-1}(x))$ and $D(f\circ \chi_{[i,x]})(\chi_{[i,x]}^{-1}(x))$ exist;
    \item[(ii)] $(x,\chi_{[i,x]})\in \mathfrak{D}(U)$;
        \item[(iii)]$D(\varphi\circ \chi_{[i,x]})(\chi_{[i,x]}^{-1}(x))\in C(e_i,\sigma_i)$;
     \item[(iv)]$D(f\circ\chi_{[i,x]})(\chi_{[i,x]}^{-1}(x))=D_Hf(x)[D(\varphi\circ \chi_{[i,x]})(\chi_{[i,x]}^{-1}(x))]$, where {$D_Hf(x):V_1(\mathbb G)\to V_1(\mathbb H)$ is the horizontal differential of $f$}, see \cref{def:HorDiff}, which exists at $\mu$-a.e. $x\in U$ thanks to \cref{prophordiff}.
\end{itemize}
In addition 
\begin{itemize}
    \item[(v)] {the maps $U\ni x\mapsto D(\varphi\circ \chi_{[i,x]})(\chi_{[i,x]}^{-1}(x))$, $U\ni x\mapsto D(f\circ \chi_{[i,x]})(\chi_{[i,x]}^{-1}(x))$ are $(\mu,\mathfrak{Q}(\mathbb G))$-measurable, and $(\mu,\mathfrak{Q}(\mathbb H))$-measurable, respectively, where $\mathfrak{Q}(\mathbb G),\mathfrak{Q}(\mathbb H)$ are the $\sigma$-algebras of Borel sets in $\mathbb G,\mathbb H$, respectively}. 
\end{itemize}
\end{proposizione}

\begin{proof}
The Alberti representations $\mathcal{A}_1,\ldots,\mathcal{A}_{n_1}$ are by definition $(\pi_{\mathbb{G}}\circ\varphi)$-independent representations such that the $\mathcal{A}_i$'s go in the $\varphi$-direction of some independent cones $C(e_i,\sigma_i)$.
Defined $\mathfrak{G}:A(X)\to V_1(\mathbb H)\times V_1(\mathbb{G})$ the map
$$
\mathfrak{G}(x,\gamma):=(\mathfrak{G}_f(x,\gamma),\mathfrak{G}_{\varphi}(x,\gamma)),
$$
where $\mathfrak{G}_f$ and $\mathfrak{G}_{\varphi}$ are the maps defined in \eqref{eq:Gf} relative to $f$ and $\varphi$.

Thanks to \cref{lem:measurability} and Proposition~\ref{prop:curvefullmeas}, we infer that for every $i=1,\ldots, n_1(\mathbb{G})$ there exists a {$(\mu,\mathfrak{Q}(H(x)))$}-measurable  map $\chi_i:U\to \Gamma(X)$ such that, for $\mu$-almost every $x\in U$, $D(\varphi\circ \chi_{[i,x]})(\chi_{[i,x]}^{-1}(x))$ and $D(f\circ \chi_{[i,x]})(\chi_{[i,x]}^{-1}(x))$ exist, $D(\varphi\circ \chi_{[i,x]})(\chi_{[i,x]}^{-1}(x))$ belongs to $C(e_i,\sigma_i)$, the maps $U\ni x\mapsto D(\varphi\circ \chi_{[i,x]})(\chi_{[i,x]}^{-1}(x))$,  $U\ni x\mapsto D(f\circ \chi_{[i,x]})(\chi_{[i,x]}^{-1}(x))$ {are $(\mu,\mathfrak{Q}(\mathbb G))$-measurable, and $(\mu,\mathfrak{Q}(\mathbb H)$-measurable, respectively}, and $(x,\chi_{[i,x]})\in \mathfrak{D}(U)$. The last identity in item (iv) directly follows from item (iii) of \cref{prophordiff}.
\end{proof}

\section{Independent Alberti representations implies differentiability}\label{sec3}

{In this long section, we aim at proving our first main result, see \cref{thm:Fondamentale2}. This will give as a direct corollary the proof of \cref{chardiffspacesINTRO}, see \cref{sec4}.

In \cref{thm:Fondamentale2} we prove that when there are $n_1(\mathbb G)$ horizontally $\rho$-universal (with respect to $\mathbb H$) independent Alberti representations $(X,d,\mu)$, and a Lipschitz function $f:X\to\mathbb H$ is given such that $(U,\varphi)$ is $(\mathbb G,\mathbb H)$-complete with respect to $f$, then for $\mu$-almost every point $x\in U$ we can find a homogeneous homomorphism $Df(x)$ acting as the differential in the chart, see the third item of \cref{def:PansuDifferentiabilitySpace}. The rough idea of the proof is to concatenate the curves obtained in \cref{constructionvectorfields} to give a good definition of $Df(x)$. 

Since the proof is long and technical, we streamline the main steps here. We aim to concatenate the fragments obtained in \cref{constructionvectorfields}. Since those are just fragments defined on compact sets of $\mathbb R$, we have to construct, starting from \cref{constructionvectorfields}, a good family of \textit{full} fragments in the space, i.e., curves defined on compact intervals of $\mathbb R$, that are close to the original curves.

Restricting the attention on a compact set $U$, we construct such curves in a way that they approximate, at the first order, the original fragments, see \cref{prop:diff_impl_direct}. This is done by embedding the space in $L^\infty$, then by filling linearly the curves in the immersion, and finally projecting on the compact set $U$, see
\cref{prop:misurbillity}. Along the way, we do our construction in such a way that all the relevant functions are measurable, for which we need the technical \cref{lemmamisurabilitaproiezione}, and \cref{lemmadominimobili}.

Then, venturing in the proof of \cref{thm:Fondamentale2}, by a careful use of Lusin and Severini--Egoroff theorem, we show that on a compact subset of $U$ of almost full measure, the function $\varphi$ along the concatenation of the curves is differentiable, see items (a), (b), and (c) of Step 3 of \cref{thm:Fondamentale2}. By exploiting the latter, and the fact that also $f$ is differentiable along the concatenation of the curves, see Step 4 and 5 of \cref{thm:Fondamentale2}, we conclude the proof in Step 6 of \cref{thm:Fondamentale2}. Notice that the completeness is used in Step 6 to prove the existence of the differential, and in Step 7 to prove that the differential is well-defined.

}
\subsection{Lemmata}

{
\begin{lemma}\label{lemmamisurabilitaproiezione}
Le  t $(X,d)$ be a metric space,
and let $K$ be a compact subset of $X$. Then there exists a Borel map $\Pi_{K}:X\mapsto K$ such that 
$$
d(\Pi_{K}(x),x)=\mathrm{dist}(x,K).
$$
\end{lemma}

\begin{proof}
For every $x\in X$ let us denote with $\mathfrak{P}(x)$ the set of those $u\in K$ such that $d(x,u)=\mathrm{dist}(x,K)$. Clearly, for every $x$ the set $\mathfrak{P}(x)$ is closed.  In order to conclude the proof of the existence of the map $\Pi_{K}$ by means of the Kuratowski-Ryll-Nardzewski selection theorem, see \cite[Theorem~5.2.1]{Srivastava}, we just need to show that the multimap $x\mapsto \mathfrak{P}(x)$ is Borel measurable, see \cite[Section 5.1]{Srivastava}. In order to prove this, we just need to show that for every closed set $C$ in $X$ we have that the set $\{x\in X:\mathfrak{P}(x)\cap C\neq \emptyset\}$ is Borel, cf. \cite[Lemma~5.1.2]{Srivastava}, and \cite[Theorem~5.2.1]{Srivastava}. Let us note that 
$$\{x\in X:\mathfrak{P}(x)\cap C\neq\emptyset\}=\{x\in\ X:\mathrm{dist}(x,K)=\mathrm{dist}(x,K\cap C)\},$$
and thus $\{x\in X:\mathfrak{P}(x)\cap C\neq \emptyset\}$ is closed. This concludes the proof.
\end{proof}
}

{
\begin{lemma}\label{lemmadominimobili}
    Let $(X,d,\mu)$ be a complete metric space such that all closed balls are compact. 
Denote by    $\mathfrak{K}(X)$   the family of compact subsets of $X$ endowed with the Hausdorff distance,  by $\mathfrak{Q}(\mathfrak{K}(X))$   the $\sigma$-algebra of Borel sets in $\mathfrak{K}(X)$, and by
$\mathfrak{Q}(X)$   the $\sigma$-algebra of Borel sets of $X$.
    Let $\Psi:X\to \mathfrak{K}(X)$ be a $(\mu,\mathfrak{Q}(
    X))$-measurable map. 
     Then, for every $x\in X$ there exists a selection $\Pi_{\Psi(x)}$ of the minimal-distance projection onto $\Psi(x)$ such that the map $x\mapsto \Pi_{\Psi(x)}(x)$ is $(\mu,\mathfrak{Q}(X))$-measurable. 
\end{lemma}

\begin{proof}
It is well known that the metric space $\mathfrak{K}(X)$ is complete. In addition, since the balls are compact, we have that $\mathfrak{K}(B)$ is compact too for every ball $B$. 

Let us define $\mathfrak{P}(x,K)$ to be the set of those $y\in K$ such that $\mathrm{dist}(x,K)=d(y,x)$. For every couple $(x,K)$ the set $\mathfrak{P}(x,K)$ is closed. Let us now show that the multimap $(x,K)\mapsto \mathfrak{P}(x,K)$ has a Borel selection. In order to do so, we will apply the Kuratowski-Ryll-Nardzewski selection theorem, see \cite[Theorem~5.2.1]{Srivastava}, and thus we just need to show that the multimap $(x,K)\mapsto \mathfrak{P}(x,K)$ is Borel. This means that, taking \cite[Lemma~5.1.2]{Srivastava} into account, we need to show that for every closed set $C\neq \emptyset$ we have that 
\begin{equation}\label{eqn:ToCheck}
\{(x,K)\in X\times \mathfrak{K}(X):\mathfrak{P}(x,K)\cap C\neq \emptyset\}\qquad\text{is Borel.}
\end{equation}
{
However,
by the definition of the various objects, we infer that 
$$
\{(x,K)\in X\times \mathfrak{K}(X):\mathfrak{P}(x,K)\cap C\neq\emptyset\}=\{(x,K)\in X\times \mathfrak{K}(X):\mathrm{dist}(x,K)=\mathrm{dist}(x,K\cap C)\},
$$
and thus, in order to check \eqref{eqn:ToCheck}, it is enough to prove that the map 
\[
(x,K)\mapsto \mathrm{dist}(x,K\cap C)-\mathrm{dist}(x,K),
\]
is lower semicontinuous.
Indeed, note that for every $(x_1,K_1),(x_2,K_2)\in X\times \mathfrak{K}(X)$  we have that 
$$
\lvert \mathrm{dist}(x_1,K_1)-\mathrm{dist}(x_2,K_2)\rvert \leq d(x_1,x_2)+d_H(K_1,K_2),
$$
and thus the map $(x,K)\mapsto \dist(x,K)$ is continuous. Let us now finally prove that the map $(x,K)\mapsto \mathrm{dist}(x,K\cap C)$ is lower semicontinuous. Let $x_i\to x\in X$, and let $\{K_i\}_{i\in\N}\subseteq \mathfrak{K}(X)$ be a sequence of compact sets converging to some $K\in  \mathfrak{K}(X)$ in the Hausdorff metric. Since we are checking the lower semicontinuity, we can assume without loss of generality, up to possibly passing to subsequences, that
\[
\liminf_{i\to +\infty}\mathrm{dist}(x_i,K_i\cap C)=\lim_{i\to +\infty}\mathrm{dist}(x_i,K_i\cap C)<+\infty.
\]
Take $z_i\in K_i\cap C$ such that $\mathrm{dist}(x_i,K_i\cap C)=d(x_i,z_i)$.  Since the compact sets $K_i$ are converging to some compact set $K$, we know that all the sets $K_i$ are contained in some closed ball $B$ and this also implies that up to subsequences we can assume without loss of generality that the sequence $z_i$ converges to some $z\in  C\cap K$. Thus
$$
\mathrm{dist}(x,C\cap K)\leq d(x,z)=\lim_{i\to \infty} d(x_i,z_i)=\lim_{i\to\infty}\mathrm{dist}(x_i,K_i\cap C),
$$
thus completing the proof of the lower semicontinuity, which was sufficient to conclude \eqref{eqn:ToCheck}.

Hence, we can find a Borel map $\varphi:(x,K)\mapsto \Pi_K(x)$ such that $d(x,\Pi_K(x))=\mathrm{dist}(x,K)$. Finally considering the map $x\mapsto \varphi(x,\Psi(x))$ we obtain the sought measurable map in the assertion.}
\end{proof}
}

{
\begin{lemma}\label{prop:misurbillity}
Let $(X,d,\mu)$ be a metric measure space, suppose $\varphi:X\to \mathbb{G}$ and $f:X\to \mathbb{H}$ are Lipschitz maps, and let $U\subseteq X$ be a compact subset of $X$. Let $n_1$ be the dimension of the horizontal stratum of $\mathbb G$. Let $C(e_1,\sigma_1),\ldots, C(e_{n_1},\sigma_{n_1})$ be independent cones in the horizontal stratum $V_1(\mathbb G)$. Suppose that $\mathcal{A}_1,\ldots,\mathcal{A}_{n_1}$ are $\varphi$-independent horizontally $\rho$-universal representations of $\mu\llcorner U$ such that each $\mathcal{A}_j$ goes in the $\varphi$-direction of $C(e_j,\sigma_j)$. Let $\mathfrak{Q}$ be the $\sigma$-algebra of Borel sets in $X$.

If we denote with $\chi_{[i,x]}$ the fragments in $\Gamma(X)$ given by Proposition~\ref{constructionvectorfields}, for every $j=1,\ldots,n_1$ there exists $(\mathcal{L}^1\otimes\mu,\mathfrak{Q})$-measurable maps $\mathcal{X}_j:[-1,1]\times U\to U$ such that
\begin{equation}
    d(\mathcal{X}_j(t,x),\chi_{[j,x]}(t+\tau_x))=\mathrm{dist}(U,\chi_{[j,x]}(t+\tau_x))\quad\text{for every $t+\tau_x\in \mathrm{Dom}(\chi_{[j,x]})$.}
    \label{eq:distanza}
\end{equation}
\end{lemma}

\begin{proof}
Let
$\mathcal{B}$ be a Banach space in which the metric space $X$ is isometrically embedded and denote  $\iota:X\to \mathcal{B}$ the isometric embedding of $X$ into $\mathcal{B}$. It is classical, by Kuratowski embedding,  that one can take $\mathcal{B}=(L^\infty(X),\|\cdot\|_\infty)$.

The first step in the proof is to show that, for every $j=1,\dots,n_1(\mathbb G)$, there exists a $(\mu,\mathfrak{Q}(\mathcal{B}))$-measurable map $\tilde{J}_j:U\ni x\mapsto \tilde{\chi}_{[j,x]} $ taking values in the set of $2$-Lipschitz curves on $\mathcal{B}$ defined on $\R$ (which we denote $\Gamma(\mathcal{B})$) such that 
$$
\tilde{J}_j(x,t)=\iota(\chi_{[j,x]}(t))\qquad\text{ for every }t\in \mathrm{Dom}(\chi_{[j,x]}).
$$
In order to define $\tilde J_j$, for every $w\in X$ let $(a_j,b_j)$, with ${j\in \N}$, be open intervals such that 
\begin{equation}
    \mathrm{Conv}(\mathrm{Dom}(\chi_{[j,w]}))=\mathrm{Dom}(\chi_{[j,w]})\cup \bigcup_{j\in\N}(a_j,b_j),
    \nonumber
\end{equation}
where all the above unions are disjoint, and where for every $S\subseteq \R$ we denote with $\mathrm{conv}(S)$ the convex hull of $S$. Let us define on $\R$ the linear interpolation of $\iota(\chi_{[j,w]}(t))$, i.e.,
\begin{equation}
    \tilde{\chi}_{[j,w]}(t):=\begin{cases}
    \iota(\chi_{[j,w]}(\min(\mathrm{Dom}(\chi_{[j,w]})))) &\text{if }t< \min(\mathrm{Dom}(\chi_{[j,w]})),\\
    \iota(\chi_{[j,w]}(t)) &\text{if }t\in \mathrm{Dom}(\chi_{[j,w]}),\\
    \frac{b_j-t}{b_j-a_j}\iota(\chi_{[j,w]}(a_j))+\frac{t-a_j}{b_j-a_j}\iota(\chi_{[j,w]}(b_j))&\text{if }t\in(a_j,b_j)\text{ for some }j\in\N,\\
    \iota(\chi_{[j,w]}(\max(\mathrm{Dom}(\chi_{[j,w]})))) &\text{if }t> \max(\mathrm{Dom}(\chi_{[j,w]})).
    \end{cases}
\end{equation}
We define $\tilde J_j(w):=\tilde\chi_{[j,w]}$ for every $j=1,\dots,n_1(\mathbb G)$ and $w\in U$ such that $\chi_{[j,w]}$ exists (which is a subset of full measure in $U$).  Notice that $\tilde\chi_{[j,w]}$ is $2$-Lipschitz because it is defined as the linear interpolation of $\iota(\chi_{[j,w]})$. 

We aim now at showing that $\tilde J_j$ is $(\mu,\mathfrak{Q}(\mathcal{B}))$-measurable. Namely, we aim at proving that for every open set $\Omega$ of $2$-Lipschitz curves on $\mathcal{B}$ defined on $\mathbb R$ we have that $\tilde J_j^{-1}(\Omega)$ is $\mu$-measurable.
In order to do so, by Lusin theorem, for every $k\in\N$ there exists a compact set $C(k)$ such that $\mu(U\setminus C(k))\leq 1/k$ and on $C(k)$ the map $x\mapsto \chi_{[j,x]}$ is continuous. This is a consequence of Lusin's theorem and the fact that the maps $x\mapsto \chi_{[j,x]}$ are measurable, as a consequence of \cref{constructionvectorfields}.
In the following, for every $k\in\N$ we denote with $A(k)$ the set $A(k):=C(k)\setminus \bigcup_{j=0}^{k-1}A(j)$ and we note that if for every $k\in\N$ the set $\mathcal{A}(k):=A(k)\cap \tilde J_j^{-1}(\Omega)$  
is {Borel} then
$$
\mu  (\tilde J_j^{-1}(\Omega)\setminus \bigcup_{k\in\N} \mathcal{A}(k))=0,
$$
and hence $\tilde J_j^{-1}(\Omega)$ is $\mu$-measurable, which is what we wanted. Therefore, in order to conclude our argument, it is sufficient to show that $\mathcal{A}(k)$ is Borel.

We now claim that  $\mathcal{A}(k)$ is a relatively open set in $A(k)$, i.e., for every $x\in\mathcal{A}(k)$ there exists a $\delta>0$ such that $ B(x,\delta)\cap A(k)\subseteq \mathcal{A}(k)$. This would conclude that $\mathcal{A}(k)$ is  Borel. 

In order to prove the claim above, we start observing that, since the map $y\mapsto \chi_{[j,y]}$ is continuous on $A(k)$, for every $x\in\mathcal{A}(k)$ and every $\varepsilon>0$ there exists a $\delta>0$ such that 
\begin{equation}
    \textrm{Hausdorff distance}(\mathrm{graph}(\chi_{[j,x]}),\mathrm{graph}(\chi_{[j,z]}))<\varepsilon,
    \label{hausdorffgraphdistance}
\end{equation}
    whenever $z\in B(x,\delta)\cap A(k)$. 
Finally, since $\tilde \chi_{[j,\cdot]}$ above is defined by linear interpolation, it is readily seen that for every $\varepsilon>0$ there exists $\delta>0$ such that whenever $x,z\in A(k)$ satisfy $\textrm{Hausdorff dist}(\mathrm{graph}(\chi_{[j,x]}),\mathrm{graph}(\chi_{[j,z]}))<\delta$, then
\[
\textrm{Hausdorff dist}(\mathrm{graph}(\tilde\chi_{[j,x]}),\mathrm{graph}(\tilde\chi_{[j,z]}))< \varepsilon.
\]
The last inequality together with \eqref{hausdorffgraphdistance} proves that for every $x\in\mathcal{A}(k)$ and every $\varepsilon>0$, there exists $\delta>0$ small enough such that whenever $z\in B(x,\delta)\cap A(k)$ we have
\[
\textrm{Hausdorff distance}(\mathrm{graph}(\tilde J_j(x)),\mathrm{graph}(\tilde J_j(z))) <\varepsilon
\]
Hence, as a direct consequence of the previous inequality, since $x\in \mathcal{A}(k)\subset \tilde J_j^{-1}(\Omega)$, there is $\varepsilon>0$ small enough such that one can find $\delta>0$ for which whenever $z\in B(x,\delta)\cap A(k)$ we have that $\tilde J_j(z)\in\Omega$, and hence $z\in\mathcal{A}(k)$ from how $\mathcal{A}(k)$ is defined. Therefore, the claim of $\mathcal{A}(k)$ being relatively open in $A(k)$ is proved, and this concludes the measurability of the map $\tilde J_j$ as discussed before.

The second step is to show that the map 
\begin{equation}\label{eqn:MapMeasurable}
    \chi:U\ni x\mapsto \chi_{[j,x]}^{-1}(x)\in\mathbb R,
\end{equation} 
is $(\mu,\mathfrak{Q}(\mathbb R))$-measurable, where $\mathfrak{Q}(\mathbb R)$ are the Borel sets of $\mathbb R$. Recall that the map $U\ni x\mapsto \chi_{[j,x]}$ is {$(\mu,\mathfrak{Q}(H(X)))$-measurable} and thus, by Lusin theorem, we can find compact sets $C(m)$ such that $\mu(U\setminus\bigcup_{m\in\N} C(m))=0$ and such that $\chi_{[j,\cdot]}$ is continuous on $C(m)$ for every $m\in\mathbb N$.  Therefore, for every $m\in\mathbb N$, and for every $x,y\in C(m),$ we have that for every $\varepsilon>0$ there exists a $\delta>0$ such that if $d(x,y)<\delta$ then 
\[
    \textrm{Hausdorff distance}(\mathrm{graph}(\chi_{[j,x]}),\mathrm{graph}(\chi_{[j,y]}))<\varepsilon.
\]
From the latter inequality, and by exploiting the fact that the curves $\chi_{[j,\cdot]}$ are $2$-biLipschitz, we get that $|\chi_{[j,x]}^{-1}(x)-\chi_{[j,y]}^{-1}(y)|\leq 2\delta+3\varepsilon$.
Thus, in particular $x\mapsto \chi_{[j,x]}^{-1}(x)$ is continuous on $C(m)$ for every $m\in\mathbb N$, and hence $\chi$ is $(\mu,\mathfrak{Q}(\mathbb R))$-measurable.

Let us now finally define $\tilde{\mathcal{X}}_j(\cdot,\cdot)$. Thanks to Proposition~\ref{lemmamisurabilitaproiezione} and since the evaluation map $\mathfrak{e}:(t,\gamma)\mapsto \gamma(t)$ is continuous when the space
$\R\times \Gamma(\mathcal{B})$ is endowed with the product metric, we infer that the composition 
$$\tilde{\mathcal{X}}_j:[-1,1]\times U\ni (t,x)\mapsto\Pi_{\iota(U)}\mathfrak{e}(t+\chi_{[j,x]}^{-1}(x),\tilde{J}_j(x)),$$
is $(\mathcal{L}^1\otimes\iota_\#\mu,\mathfrak{Q}(\mathcal{B}))$-measurable. Clearly denoting $\pi_\iota$ the inverse of $\iota$ on $\iota(X)$, we immediately see that the map $\mathcal{X}_j(t,x):=\pi_\iota(\tilde{\mathcal{X}}_j(t,x))$ is $(\mathcal{L}^1\otimes\mu,\mathfrak{Q})$-measurable.

Let us check that \eqref{eq:distanza} holds. Let $\tilde t:=t+\chi_{[j,x]}^{-1}(x)\in \mathrm{Dom}( \chi_{[j,x]}(t))$ and note that
\begin{equation}
    \begin{split}
       d(\mathcal{X}_j(t,x),\chi_{[j,x]}(\tilde t))&=d(\iota(\mathcal{X}_j(t,x)),\iota(\chi_{[j,x]}(\tilde t)))=d(\tilde{\mathcal{X}}_j(t,x),\iota(\chi_{[j,x]}(\tilde t)))\\
       &=d(\Pi_{\iota(U)}\mathfrak{e}(\tilde t,\tilde{J}_j(x)),\iota(\chi_{[j,x]}(\tilde t)))=d(\Pi_{\iota(U)}(\iota(\chi_{[j,x]}(\tilde t))),\iota(\chi_{[j,x]}(\tilde t)))\\
       &=\mathrm{dist}(\iota(U),\iota(\chi_{[j,x]}(\tilde t)))=\mathrm{dist}(U,\chi_{[j,x]}(\tilde t)).
    \end{split}
\end{equation}
This concludes the proof of \cref{prop:misurbillity}.
\end{proof}
}

\begin{proposizione}\label{prop:diff_impl_direct}
Let $(X,d,\mu)$ be a metric measure space, suppose $\varphi:X\to \mathbb{G}$ is a Lipschitz map, and let $U\subseteq X$ be a compact subset of $X$. Let $n_1$ be the dimension of the horizontal stratum of $\mathbb G$. Let $C(e_1,\sigma_1),\ldots, C(e_{n_1},\sigma_{n_1})$ be independent cones in the horizontal stratum $V_1(\mathbb G)$. Suppose that $\mathcal{A}_1,\ldots,\mathcal{A}_{n_1}$ are $\varphi$-independent horizontally $\rho$-universal representations of $\mu\llcorner U$ such that each $\mathcal{A}_j$ goes in the $\varphi$-direction of $C(e_j,\sigma_j)$. Let $\chi_{[j,x]}$ and $ \mathcal{X}_j(t,x)$ be the maps constructed in \cref{constructionvectorfields}, and
\cref{prop:misurbillity}, respectively

Then for $\mu$-almost every $x\in U$ and for every $j=1,\ldots,n_1$ we have 
\begin{equation}
    \lim_{\mathrm{Dom}(\chi_{[j,x]})-\chi_{[j,x]}^{-1}(x)\ni t\to 0\\
    } \frac{d(\mathcal{X}_j(t,x),\chi_{[j,x]}(  \chi_{[j,x]}^{-1}(x)+ t))}{t}=0.
\end{equation}
\end{proposizione}

\begin{proof}
{Thanks to Proposition~\ref{constructionvectorfields}, and from Theorem~\ref{prophordiff}, see in particular the last part of \cref{rem:NotReallyDerivative},
we know that for $\mu$-almost every $x\in U$ the following hold}
\begin{itemize}
    \item[(i)] for every $j=1,\ldots,n_1(\mathbb{G})$ the derivative $D(\varphi\circ \chi_{[j,x]})(\chi_{[j,x]}^{-1}(x))$ exists, the map $x\mapsto D(\varphi\circ \chi_{[j,x]})(\chi_{[j,x]}^{-1}(x))$ {is $(\mu,\mathfrak{Q}(\mathbb G))$-measurable, where $\mathfrak{Q}(\mathbb G)$ is the $\sigma$-algebra of Borel sets of $\mathbb G$},  $(x,\chi_{[j,x]})\in \mathfrak{D}(U)$ and 
    $D(\varphi\circ \chi_{[j,x]})(\chi_{[j,x]}^{-1}(x))\in C(e_j,\sigma_j)$;
    \item[(ii)] for every Lipschitz function $f:X\to \R$ we have that $(f\circ \chi_{[j,x]})'(\chi_{[j,x]}^{-1}(x))$ exists and for every $j=1,\ldots,n_1$ we have
$$(f\circ\chi_{[j,x]})'(\chi_{[j,x]}^{-1}(x))=Df(x)[D(\varphi\circ \chi_{[j,x]})((\chi_{[j,x]})^{-1}(x))],$$
where $Df(x):V_1(\mathbb G)\to\mathbb R$ is the horizontal differential of $f$, see \cref{prophordiff}.
\end{itemize}
The first step of the proof is to show that the function
$$g(x):=\inf\{r>0:\mu(B(x,r)\cap U)>0\},$$
is a non-negative real-valued $1$-Lipschitz function. Let $x,y\in X$ and note that $B(y,r)\subseteq B(x,r+d(x,y))$. Therefore, for every $\varepsilon>0$ we have
\begin{equation}
\mu(B(x,g(y)+d(x,y)+\varepsilon)\cap U)\geq \mu(B(y,g(y)+\varepsilon)\cap U)>0.
\label{eq:mer1}
\end{equation}
Inequality \eqref{eq:mer1} implies that $g(x)\leq g(y)+d(x,y)$ and thus, interchanging $x$ and $y$, the claim is proved.
Suppose by contradiction that there is a $\mu$-positive compact set $K\subseteq U$ for which 
\begin{equation}
 \lim_{\mathrm{Dom}(\chi_{[j,x]})-\chi_{[j,x]}^{-1}(x)\ni t\to 0}\frac{\mathrm{dist}(U,\chi_{[j,x]}(\chi_{[j,x]}^{-1}(x)+t))}{t}=0.
 \label{eq:mer2}
\end{equation}
fails $\mu$-almost everywhere on $K$ and such that items (i) and (ii) of this proof hold at every point of $K$.
This means that for $\mu$-almost every $x\in K$ there is an infinitesimal sequence $s_i(x)$ and a $\lambda(x)>0$ such that
\begin{equation}
\dist(\chi_{[j,x]}(\chi_{[j,x]}^{-1}(x)+s_i(x)),U)\geq \lambda(x) \lvert s_i(x)\rvert,\qquad \text{ for every }i\in\mathbb{N}.
\label{eq:mer5}
\end{equation}

Let us now show that $g(z)=0$ for $\mu$-almost every $z\in K$. If this is not the case, we can find a $\mu$-positive subcompact $\tilde K$ such that $g>0$ on $\tilde K$. Then, by definition of $g$, for every $z\in \tilde K$ we can find an $r(z)>0$ such that $\mu( B(z,r(z))\cap U)=0$. However, since $\tilde K$ is compact, we can find  finitely many $z_i\in K$, with $i=1,\ldots,N$ such that 
$$
\tilde{K}\subseteq \bigcup_{i=1}^N B(z_i,r_i).
$$
This implies that
$$
\mu(\tilde{K})=\mu(\tilde{K}\cap U)\leq \sum_{i=1}^N\mu(B(z_i,r(z_i))\cap U)=0,
$$
which is in contradiction with the fact that $\tilde{K}$ had positive measure. This concludes the proof of the fact that $g(z)=0$ for $\mu$-almost every $z\in K$.

In order to discuss why \eqref{eq:mer5} is false, we shall fix a $z\in K$ where $g(z)=0$. Throughout the rest of the proof, we assume without loss of generality that $\chi_{[j,z]}^{-1}(z)=0$ for every $j=1,\ldots,n_1$.
Let us first derive a contradiction if we also have the following assumption, which strengthens \eqref{eq:mer5},
\begin{equation}
\liminf_{\mathrm{Dom}(\chi_{[j,z]})\ni r\to 0}\frac{\dist\big(\chi_{[j,z]}(r),U\big)}{\lvert r\rvert }\geq \eta(z),
\label{eq:mer3}
\end{equation}
for some positive $\eta(z)>0$ depending on $z$. {If the latter holds, we must conclude that
\begin{equation}
    \liminf_{\mathrm{Dom}(\chi_{[j,z]})\ni r\to 0}\frac{g(\chi_{[j,z]}(r))-g(z)}{\lvert r\rvert}=\liminf_{\mathrm{Dom}(\chi_{[j,z]})\ni r\to 0}\frac{g(\chi_{[j,z]}(r))}{\rvert r\rvert}\geq \eta(z).
    \label{eq:conseguenzaassenzacono}
\end{equation}
}
Since $X$ is in particular a Lipschitz differentiability space, see \cref{prophordiff}, we infer that 
\begin{equation}
    \begin{split}
        0&=\lim_{\mathrm{Dom}(\chi_{[j,z]})\ni t\to 0}\frac{\lvert g(\chi_{[j,z]}(t))- dg(z)[\pi_{\mathbb G}\varphi(\chi_{[j,z]}(t))-\pi_{\mathbb G}\varphi(z)]\rvert}{d(\chi_{[j,z]}(t),z)}\\&=\lim_{\mathrm{Dom}(\chi_{[j,z]})\ni t\to 0}\frac{\lvert g(\chi_{[j,z]}(t))-t dg(z)[\zeta(z)]+o(t)\rvert}{d(\chi_{[j,z]}(t),z)}\\
        &\geq  \lim_{\mathrm{Dom}(\chi_{[j,z]})\ni t\to 0}\frac{\lvert g(\chi_{[j,z]}(t))-t dg(z)[\zeta(z)]+o(t)\rvert}{2\lvert t\rvert},
        \label{formuladiffspaces}
    \end{split}
\end{equation}
where $\zeta(z):=D(\varphi\circ \chi_{[j,z]})(\chi_{[j,z]}^{-1}(z))$ and where the last inequality comes from the assumption of the $2$-biLipschitzianity of $\gamma$, see Definition~\ref{lipcurvez}.
However, this implies that $g(\chi_{[j,z]}(t))$ is differentiable at $t=0$, and this comes in contradiction with \eqref{eq:conseguenzaassenzacono}.
 Since this would contradict the $\mu$-almost sure differentiability on $K$ of the function $g$, we conclude that 
\begin{equation}
\liminf_{\mathrm{Dom}(\chi_{[j,z]})\ni r\to 0}\frac{\dist\big(\chi_{[j,z]}(r),U\big)}{\lvert r\rvert }=0,
    \label{eq:mer4}
\end{equation}
for $\mu$-almost every $z\in K$. However, \eqref{eq:mer5} and  \eqref{eq:mer4} imply in particular that
\begin{equation}
\liminf_{\mathrm{Dom}(\chi_{[j,z]})\ni r\to 0}\frac{g(\chi_{[j,z]}(r))}{\lvert r\rvert}=0,\qquad\text{and}\qquad\limsup_{\mathrm{Dom}(\chi_{[j,z]})\ni r\to 0}\frac{g(\chi_{[j,z]}(r))}{\lvert r\rvert}\geq\lambda(z),
\nonumber
\end{equation}
which contradicts \eqref{formuladiffspaces}, since $g(z)=0$. 
This shows in particular that for $\mu$-almost every $x\in U$ and every $j=1,\ldots,n_1(\mathbb{G})$ we have
\begin{align*}
   &\lim_{\mathrm{Dom}(\chi_{[j,x]})-\chi_{[j,x]}^{-1}(x)\ni t\to 0
    } \frac{d(\mathcal{X}_j(t,x),\chi_{[j,x]}(  \chi_{[j,x]}^{-1}(x)+ t))}{t}\\ &=\lim_{\mathrm{Dom}(\chi_{[j,x]})-\chi_{[j,x]}^{-1}(x)\ni t\to 0}\frac{\mathrm{dist}(U,\chi_{[j,x]}(\chi_{[j,x]}^{-1}(x)+t))}{t}=0,
\end{align*}
where the last inequality is true thanks to Proposition~\ref{prop:diff_impl_direct}. This concludes the proof.
\end{proof}

\subsection{Proof of \cref{thm:Fondamentale2}}
We next prove the main result of this section.
Starting with independent and horizontally universal Alberti representations in complete charts,
we construct differentials of Lipschitz functions. 
{
\begin{teorema}\label{thm:Fondamentale2}
Let $\mathbb G,\mathbb H$ be Carnot groups, and denote with $n_1$ the dimension of the horizontal stratum of $\mathbb G$. Let $(X,d,\mu)$ be a metric measure space, and let $f:X\to\mathbb H$, $\varphi:X\to \mathbb{G}$ be Lipschitz. Let $U\subseteq X$ be a  Borel subset of $X$, and assume $(U,\varphi)$ is $(\mathbb{G},\mathbb{H})$-complete with respect to $f$, see \cref{def:GHcomplete}. 

Suppose that there exist  $\varphi$-independent horizontally $\rho$-universal representations $\mathcal{A}_1,\ldots,\mathcal{A}_{n_1}$of $\mu\llcorner U$, for some $\rho>0$. Then for $\mu$-almost every $x_0\in U$, there exists a unique homogeneous homomorphism $Df(x_0):\mathbb G\to\mathbb H$ such that \eqref{eqn:PANSUDIFF} holds with $f,\varphi$.
\end{teorema}
}

\begin{proof}
{Without loss of generality, by inner regularity, we can assume that $U$ is compact.} We first need to define a candidate differential. In order to do so, we let $\chi_1,\ldots,\chi_{n_1}:U\to \Gamma(X)$ be the {$(\mu,\mathfrak{Q}(H(X)))$-measurable} map constructed in Proposition~\ref{constructionvectorfields} relative to the maps $\varphi$ and $f$. Having this at our disposal, first, we give a definition of the homogeneous homomorphism, which a priori might be ill-given. Then, we proceed to prove that if $Df(x)$ is well defined, it is a homogeneous homomorphism and that \eqref{eqn:PANSUDIFF} holds. Finally,
we prove that $Df(x)$ is well-defined for $\mu$-almost every $x\in U$. The uniqueness part is remarked, as an appendix of this proof, in \cref{rk_identity}. 

Throughout the proof for every $j=1,\ldots,n_1(\mathbb{G})$ and for $\mu\llcorner U$-almost every $x\in U$ we let, whenever they are defined,
\begin{itemize}
    \item[($\alpha$)] $\zeta_j(x):=D(\varphi\circ \chi_{[j,x]})(\chi_{[j,x]}^{-1}(x))$,
\item[($\beta$)] 
$\partial_j f(x):=D(f\circ \chi_{[j,x]})(\chi_{[j,x]}^{-1}(x))$.
\end{itemize}
Thanks to Proposition~\ref{constructionvectorfields}, recall that the maps $U\ni x\mapsto \zeta_j(x)$ and $U\ni x\mapsto \partial_j f(x)$ are {$(\mu,\mathfrak{Q}(\mathbb G))$-measurable, and $(\mu,\mathfrak{Q}(\mathbb H))$-measurable, respectively}.

\paragraph{Step 1.}
We define the Pansu differential of $f$ at $x$ with respect to $\varphi$ and the curves $\{\chi_{[j,\cdot]}\}_{j=1,\ldots,n_1(\mathbb{G})}$ to be the homogeneous homomorphism $Df(x):\mathbb G\to\mathbb H$ that acts as follows. First, recall that $\{\zeta_1(x),\ldots,\zeta_{n_1}(x)\}$ is a basis of $V_1(\mathbb G)$ for every $x\in U$. Then, for every $v\in\mathbb{G}$ we can find $M\in\N$, $\lambda_1,\ldots,\lambda_M\in\R$, and $i_1,\ldots,i_M\in\{1,\ldots,n_1(\mathbb{G})\}$, all depending on $x$ and $v$, such that 
\begin{equation}
    v=\delta_{\lambda_1}(\zeta_{i_1}(x))\cdot\ldots\cdot\delta_{\lambda_M}(\zeta_{i_M}(x)),
    \label{espressioneperv}
\end{equation}
see \cref{propdecomposizione}. Notice also that if $\|v\|$ is uniformly bounded, also the $|\lambda_i|$ are uniformly bounded as a consequence of \cref{propdecomposizione}. Now, we define the action of $Df(x)$ on $v$ as
\[
Df(x)[v]:=\delta_{\lambda_1}(\partial_{i_1} f(x))\cdot\ldots\cdot\delta_{\lambda_M}(\partial_{i_M} f(x)).
\]
Clearly, at this stage, it is not clear whether $Df(x)$ is well-defined $\mu$-almost everywhere or not because it might depend on the choices of the $\lambda_j$'s. However, it is clear that $Df(x)$ is definable on the whole $\mathbb G$, as a consequence of \cref{propdecomposizione}, and that \emph{if} $Df(x)$ is well defined, it must be a homogeneous homomorphism.

\paragraph{Step 2.} By a direct application of \cref{prop:misurbillity}, for every $j=1,\dots,n_1$, the maps 
\begin{equation}
    \begin{split}
       (t,x)\mapsto  \delta_{1/t}\big(\varphi(x)^{-1}\varphi(\mathcal{X}_j(t,x))\big),\\
    (t,x)\mapsto  \delta_{1/t}\big(f(x)^{-1}f(\mathcal{X}_j(t,x))\big),
        \label{eq:rapportiincrementalimisurabili}
    \end{split}
    \end{equation}
    are {$(\mu\otimes\mathcal{L}^1,\mathfrak{Q}(\mathbb G))$-measurable, and $(\mu\otimes\mathcal{L}^1,\mathfrak{Q}(\mathbb H))$-measurable, respectively, both on $[-1,0)\times U$, and $(0,1]\times U$.}
    
    Let us fix a $j=1,\dots,n_1$ from now on. 
    As an aim of this step,
    we claim that the maps in \eqref{eq:rapportiincrementalimisurabili} converge, for $\mu$-a.e. $x\in U$, as $t\to 0$, to $\zeta_j(x)$ and $\partial_j f(x)$ respectively. Namely, we shall prove that 
 \begin{equation}\label{eqn:SoughtConvergence}
     \lim_{t\to 0}\delta_{1/t}\big(\varphi(x)^{-1}\varphi(\mathcal{X}_j(t,x))\big)=\zeta_j(x),
 \end{equation}
 for $\mu$-almost every $x\in U$, the other equality being analogous. However, as a direct consequence of \cref{prop:diff_impl_direct}, for every $t\in [-1,1]$ such that $t+\chi_{[j,x]}^{-1}(x)\in \mathrm{Dom}(\chi_{[j,x]})$, for the $\Delta_x(t)$ defined by 
 \begin{equation}
     \begin{split}
        \delta_{1/t}\big(\varphi(x)^{-1}\varphi(\mathcal{X}_j(t,x))\big)=\zeta_j(x)\delta_{1/t}(\Delta_x(t)),
     \end{split}
 \end{equation}
 we have 
 \begin{equation}\label{eqn:Final0}
 \lim_{t\to 0, t+\chi_{[j,x]}^{-1}(x)\in \mathrm{Dom}(\chi_{[j,x]})}\lVert\Delta_x(t)\rVert/t=0.
 \end{equation}
 Since by construction $\chi_{[j,x]}^{-1}(x)$ is a density point of $\mathrm{Dom}(\chi_{[j,x]})$, {for every $\varepsilon>0$ there exists an $r(\varepsilon)>0$ such that for every $t$ with $\lvert t\rvert<r(\varepsilon)$ there exists a $\tau(t)+\chi_{[j,x]}^{-1}(x)\in \mathrm{Dom}(\chi_{[j,x]})$ such that $\lvert \tau(t)-t\rvert\leq \varepsilon \lvert t\rvert$} and thus
 \begin{equation}\label{eqn:Final}
     \begin{split}
         &\zeta_j(x)\delta_{1/\tau(t)}(\Delta_x(\tau(t)))=\delta_{1/\tau(t)}\big(\varphi(x)^{-1}\varphi(\mathcal{X}_j(\tau(t),x))\big)\\
         &=\delta_{t/\tau(t)}(\delta_{1/t}(\varphi(x)^{-1}\varphi(\mathcal{X}_j(t,x))))\cdot\\
         &\qquad\qquad\qquad\qquad\delta_{1/\tau(t)}(\varphi(\mathcal{X}_j(t,x))^{-1}\varphi(\mathcal{X}_j(\tau(t),x))).
     \end{split}
 \end{equation}
 However, let us note that 
 \begin{equation}\label{eqn:1}
         \lVert\delta_{1/\tau(t)}(\varphi(\mathcal{X}_j(t,x))^{-1}\varphi(\mathcal{X}_j(\tau(t),x)))\rVert\leq \mathrm{Lip}(\varphi)\cdot d(\mathcal{X}_j(t,x),\mathcal{X}_j(\tau(t),x))/\tau(t).
 \end{equation}
 For the ease of notation, let us rename $\chi_{[j,x]}^{-1}(x):=\tau_x$ Thanks to the triangle inequality we know that 
 \begin{equation}\label{eqn:2}
   d(\mathcal{X}_j(t,x),\mathcal{X}_j(\tau(t),x))
         \leq d(\mathcal{X}_j(t,x),\chi_{[j,x]}(\tau_x+\tau(t)))+d(\mathcal{X}_j(\tau(t),x),\chi_{[j,x]}(\tau_x+\tau(t))).
 \end{equation}
Thus thanks to Proposition~\ref{prop:diff_impl_direct}, and using the notation in the statement and the proof of \cref{prop:misurbillity} (e.g., $\mathfrak{e}$ is the evaluation map, $\iota$ is an isometric embedding into a Banach space, and $\Pi$ is the minimal-distance projection, $\tilde J_j$ is curve associated to $\chi_{[j,\cdot]}$ in the Banach space), we see that we just need to estimate 
\begin{equation}\label{eqn:3}
    \begin{split}
        &d(\mathcal{X}_j(t,x),\chi_{[j,x]}(\tau_x+\tau(t)))\\
        &=d(\Pi_{\iota(U)}\mathfrak{e}(t+\tau_x,\tilde{J}_j(x)),\iota\circ\chi_{[j,x]}(\tau_x+\tau(t)))\\
        &=d(\Pi_{\iota(U)}\tilde\chi_{[j,x]}(t+\tau_x),\tilde\chi_{[j,x]}(\tau_x+\tau(t)))\\
        &\leq d(\Pi_{\iota(U)}\tilde\chi_{[j,x]}(t+\tau_x),\tilde\chi_{[j,x]}(\tau_x+t))\\
        &\qquad\qquad\qquad\qquad+d(\tilde\chi_{[j,x]}(\tau_x+t),\tilde\chi_{[j,x]}(\tau_x+\tau(t)))\\
        &\leq d(\Pi_{\iota(U)}\tilde\chi_{[j,x]}(t+\tau_x),\tilde\chi_{[j,x]}(\tau_x+t))+2\varepsilon\lvert t\rvert\\
        &\leq \mathrm{dist}(\iota(U),\tilde\chi_{[j,x]}(\tau_x+\tau(t)))+d(\tilde\chi_{[j,x]}(\tau_x+\tau(t)),\tilde\chi_{[j,x]}(\tau_x+t))+2\varepsilon\lvert t\rvert\\
        &\leq \mathrm{dist}(\iota(U),\tilde\chi_{[j,x]}(\tau_x+\tau(t)))+4\varepsilon\lvert t\rvert=\mathrm{dist}(U,\chi_{[j,x]}(\tau_x+\tau(t)))+4\varepsilon\lvert t\rvert\\
     &=d(\mathcal{X}_j(\tau(t),x),\chi_{[j,x]}(\tau_x+\tau(t)))+4\varepsilon\lvert t\rvert.
    \end{split}
\end{equation}
We point out that in the previous inequalities, sometimes we are denoting with $d$ also the distance in the space in which $X$ is embedded in \cref{prop:misurbillity}. As a consequence of \eqref{eqn:1}, \eqref{eqn:2}, \eqref{eqn:3}, and \cref{prop:diff_impl_direct}, we have that for every $\varepsilon>0$ there exists $\delta>0$ small enough such that whenever $|t|<\delta$ and $\tau(t)$ is as above, we have 
\[
\lVert\delta_{1/\tau(t)}(\varphi(\mathcal{X}_j(t,x))^{-1}\varphi(\mathcal{X}_j(\tau(t),x))\rVert \leq \mathrm{Lip}(\varphi)  \left(\varepsilon+4\varepsilon\frac{|t|}{|\tau(t)|}+\varepsilon\right),
\]
and since by the definition of $\tau(t)$ we have that $1/(1+\varepsilon)\leq |t/\tau(t)|\leq 1/(1-\varepsilon)$, we conclude that, as $t\to 0$, $\delta_{1/\tau(t)}(\varphi(\mathcal{X}_j(t,x))^{-1}\varphi(\mathcal{X}_j(\tau(t),x))\to 0$. Thus the last convergence, together with \eqref{eqn:Final0}, \eqref{eqn:Final}, and the fact that $|t/\tau(t)|$ can be taken to be bounded from above and below as $t\to 0$, we get the sought claim \eqref{eqn:SoughtConvergence}.
 
\paragraph{Step 3.} Let us fix $\varepsilon>0$ for this step.  By Lusin's theorem and Severini--Egoroff's theorem, we claim that we can find a compact subset $\tilde{U}$  of the compact set   $U$ such that 
\begin{itemize}
    \item[(i)] $\mu(U\setminus \tilde{U})\leq \varepsilon\mu(U)<+\infty$;
    \item[(ii)] for every $j=1,\ldots,n_1$ the restriction of the maps $\zeta_j$ and $\partial_jf$ are uniformly continuous on $\tilde{U}$;
    \item[(iii)] for every $j=1,\ldots,n_1$, as $t$ goes to $0$, the quantities
    \begin{equation}
    \begin{split}
          \delta_{1/t}\big(\varphi(x)^{-1}\varphi(\mathcal{X}_j(t,x))\big)\\
        \delta_{1/t}\big(f(x)^{-1}f(\mathcal{X}_j(t,x))\big)
        \label{apprxderivate}
    \end{split}
    \end{equation}
    converge uniformly on $\tilde{U}$ to $\zeta_j(x)$ and $\partial_jf(x)$, respectively;
    \item[(\hypertarget{iv}{iv})] for every $j=1,\ldots,n_1(\mathbb{G})$, the following limits
    \begin{equation}\lim_{t\to 0}\frac{\tau(t,x)}{t}=1,\qquad\text{and}\qquad
    \lim_{t\to 0\\
    } \frac{d(\mathcal{X}_j(t,x),\chi_{[j,x]}(  \tau_x+ \tau(t,x)))}{t}=0,
    \label{eq:unifdelicati}
\end{equation}
where $\tau(t,x):=\Pi_{\mathrm{Dom}(\chi_{[j,x]})}(t)$, hold uniformly in $\tilde{U}$. Recall that with $\Pi_{\mathrm{Dom}(\chi_{[j,x]})}(t)$ we mean a measurable choice of a minimal-distance projection of $t$ onto $\mathrm{Dom}(\chi_{[j,x]})$; see \cref{lemmadominimobili}.
\end{itemize}
The existence of a $\tilde{U}$ satisfying items (i)-(iii) above is a consequence of the measurability of the maps $x\mapsto \zeta_j(x)$, $x\mapsto \partial_j f(x)$, and of the maps in \eqref{apprxderivate} for every $t$, which has been proved in Step 2. Let us check that $\tilde U$ can be chosen in such a way that also (iv) holds. 

{Let us fix $j=1,\dots,n_1$. Let us first check that the map $\tau:(t,x)\mapsto \Pi_{\mathrm{Dom}(\chi_{[j,x]})}(t)$ is $\mathcal{L}^1\otimes\mu$-measurable. Indeed, this comes from \cref{lemmadominimobili}, in which proof we showed that there exists a Borel selection $\mathbb R\times \mathfrak{K}(\mathbb R)\ni (t,K)\to \Pi_K(t)$. Hence, since $(t,x)\mapsto (t,\mathrm{Dom}(\chi_{[j,x]}))\subset \mathbb R\times\mathfrak{K}(\mathbb R)$ is $(\mathcal{L}^1\otimes\mu,\mathfrak{Q}(\mathbb R\times\mathfrak{K}(\mathbb R)))$-measurable, we get the conclusion by composing the previous two maps.
This implies by Severini-Egoroff's theorem and the fact that by construction $\tau(0,x)$ is a density point for $\mathrm{Dom}(\chi_{[j,x]})$, that $\tau(t,x)/t$ can be chosen to converge uniformly to $1$ on $\tilde{U}$.}

The latter measurability of $\tau(t,x)$, together with the fact that $x\mapsto \tau_x$ is $\mu$-measurable - see the proof of \cref{prop:misurbillity}, in particular \eqref{eqn:MapMeasurable} - and the measurability of $\mathcal{X}_j$, see \cref{prop:misurbillity}, readily imply {that the map $(t,x)\mapsto d(\mathcal{X}_j(t,x),\chi_{[j,x]}(  \tau_x+ \tau(t,x)))$ is $\mathcal{L}^1\otimes\mu$-measurable}. 
Hence, it is thus immediate to see that the map
$$(t,x)\mapsto \frac{d(\mathcal{X}_j(t,x),\chi_{[j,x]}(  \tau_x+ \tau(t,x)))}{t},$$
is $\mathcal{L}^1\otimes\mu$-measurable. The estimate in \eqref{eqn:3}, together with Proposition~\ref{prop:diff_impl_direct}, readily implies by Severini-Egoroff's theorem that item (iv) can be satisfied too. 

\medskip

In this third step, we show that for every $\alpha>0$ there exists a compact set $E$ of $\tilde{U}$ with  $\mu(\tilde{U}\setminus E)\leq \alpha\mu(\tilde{U})$ which satisfies the following properties. If we choose $v\in\mathbb{G}$ and we write it as in \eqref{espressioneperv}, we have that for every 
$x\in E$ and every $t\in[-\min_i 1/|\lambda_i|,\min_i 1/|\lambda_i|]$ we construct a Borel curve $\gamma_{\delta_t(v)}:[0,M]\to U$ starting from $x$, for which
\begin{itemize}
    \item[(a)] For $\Delta_{M,x}(t)$ such that $\varphi(x)^{-1}\varphi(\gamma_{\delta_t(v)}(M))=\delta_t(v)\Delta_{M,x}(t)$, we have $\lVert\Delta_{M,x}(t)\rVert/t$ converges uniformly to $0$ on $x\in E$ as $t$ goes to $0$;
    \item[(b)] There is a function $\mathcal{R}_{1,x}$ of $t$ such that $d(x,\gamma_{\delta_t(v)}(M))\leq  2(\sum_{i=1}^M|\lambda_i|)|t|+\mathcal{R}_{1,x}(t)$, and $\mathcal{R}_{1,x}(t)/t$ converges uniformly to $0$ on $x\in E$ as $t$ goes to $0$;
    \item[(c)] $\lim_{t\to 0}t^{-1}\mathrm{dist}(\gamma_{\delta_t(v)}(M), \tilde{U})=0$ uniformly on $x\in E$.
\end{itemize}

\smallskip

Notice that by a countable covering argument, we may restrict to prove the property for $\|v\|$ uniformly bounded, for which we also get a uniform upper bound on $\|\lambda_i\|$ from \cref{propdecomposizione}. Let us deal with the base step. 
If $M=1$, we let, for  $t\in[-\min_i 1/|\lambda_i|,\min_i 1/|\lambda_i|]$ and $s\in [0,1]$, $x\in\tilde U$,
$$
\gamma_{\delta_t(v)}(s):=\mathcal{X}_{i_1}(\lambda_1ts,x).
$$
Let us see that with this definition, and by how $\tilde{U}$ is defined, the curve $\gamma_{\delta_t(v)}$ satisfies the items (a) and (b) with $\tilde U$ instead of $E$ in them. 
Thanks to the work done above, cf. \eqref{eqn:SoughtConvergence} and \eqref{apprxderivate}, which comes from the choice of $\tilde{U}$, we know that, whenever $s\in[0,1]$,  $t\in[-\min_i 1/|\lambda_i|,\min_i 1/|\lambda_i|]$, $x\in \tilde U$, $i\in \{1,\dots,n_1\}$, and $\lambda\in[-\max_i|\lambda_i|,\max_i|\lambda_i|]$, we have that, for $\Delta_{\lambda,i}(x,t,s)$ defined as follows
\begin{equation}\label{eqn:SomeUniformity}
\delta_{1/s}\big(\varphi(x)^{-1}\varphi(\mathcal{X}_{i}(\lambda ts,x))\big)=\delta_{\lambda t}(\zeta_{i}(x))\Delta_{\lambda,i}(x,t,s),
\end{equation}
$\lVert\Delta_{\lambda,i}(x,t,1)\rVert/t$ converges to $0$ uniformly on $\tilde{U}$ as $t$ goes to $0$. This directly implies item (a) with $\tilde U$ instead of $E$ in the base step, simply taking $\lambda=\lambda_1$, $i=i_1$, and $s_1$ in \eqref{eqn:SomeUniformity}. 

Let us now check item (b). For every $i\in\{1,\dots,n_1\}$, every $\lambda\in[\max_i|\lambda_i|,\max_i|\lambda_i|]$, every $x\in\tilde U$, and every $t\in[-\min_i 1/|\lambda_i|,\min_i 1/|\lambda_i|]$, we have
\begin{equation}
\begin{split}
    & d(x,\mathcal{X}_{i}(\lambda t,x))\\
     &\leq  d(x,\chi_{[i,x]}(  \tau_x+ \lambda\tau(t,x)))+d(\chi_{[i,x]}(  \tau_x+ \lambda\tau(t,x)),\mathcal{X}_{i}(\lambda t,x))\\
     &\leq 2|\lambda\tau(t,x)|+d(\chi_{[i,x]}(  \tau_x+ \lambda\tau(t,x)),\mathcal{X}_{i}(\lambda t,x)),
     \nonumber
\end{split}
\end{equation}
where $\tau(t,x):=\Pi_{\mathrm{Dom}(\chi_{[i,x]})}(t)$ is a measurable selection of the metric projection of $t$ as defined above. Dividing both sides by $|t|$, 
and recalling item (\hyperlink{iv}{iv}), we infer thanks to the following inequality,
 \begin{equation}\label{eqn:SomeUniformity2}
 \begin{split}
      \frac{d(x,\mathcal{X}_{i}(\lambda t,x))}{|t|}&\leq 2|\lambda|+ \frac{\lvert t-\tau(t,x)\rvert}{|t|}+\frac{d(\chi_{[i,x]}(  \tau_x+ \lambda\tau(t,x)),\mathcal{X}_{i}(\lambda t,x))}{|t|}\\
      &=:2|\lambda|+\tilde{\mathcal{R}}_{1,x}^{\lambda,i}(t),
 \end{split}
 \end{equation}
that, for $\lambda,i$ fixed in the ranges as above, the remainder $\tilde{\mathcal{R}}_{1,x}^{\lambda,i}(t)$ goes uniformly to zero, in $x\in \tilde U$, as $t\to 0$. Hence, taking $\lambda=\lambda_1,i=i_1$ in \eqref{eqn:SomeUniformity2}, we get that also item (b) above is verified in the base step $M=1$ with $\tilde U$ instead of $E$. 

Let us now prove property (c) in the base step with $E$ a proper subset of $\tilde U$, which will be called $E_1$. In the following, denote for every $j\in\{1,\dots,n_1\}$, by $\tilde{\mathcal{X}}_j:[-1,1]\times  \tilde{U}\to \tilde{U}$ the map given by \cref{prop:misurbillity} applied to $\tilde U$. Therefore, since
\begin{equation}
    \begin{split}
        &\qquad\qquad\frac{\mathrm{dist}(\tilde{U},\mathcal{X}_j(t,x))}{\lvert t\rvert}\leq \frac{d(\tilde{\mathcal{X}}_j(t,x),\mathcal{X}_j(t,x))}{\lvert t\rvert}\\
        \leq& \frac{d(\tilde{\mathcal{X}}_j(t,x),\chi_{[j,x]}(\tau_x+\tau(t,x)))}{\lvert t\rvert}+\frac{d(\mathcal{X}_j(t,x),\chi_{[j,x]}(\tau_x+\tau(t,x)))}{\lvert t\rvert},
    \end{split}
\end{equation}
then Step 2, and particularly the computations in \eqref{eqn:3}, implies $t^{-1}\mathrm{dist}(\tilde{U},\mathcal{X}_j(t,x))$ converges pointwise to $0$ at $\mu$-almost every point on $\tilde{U}$. Denote by $\Pi_{\tilde{U}}$ be the map given by Proposition~\ref{lemmamisurabilitaproiezione} and note that $\mathrm{dist}(\tilde{U},\mathcal{X}_j(t,x))=d(\mathcal{X}_j(t,x),\Pi_{\tilde{U}}(\mathcal{X}_j(t,x)))$. Since the map $(t,x)\mapsto \Pi_{\tilde{U}}(\mathcal{X}_j(t,x))$ is $(\mathcal{L}^1\otimes\mu,\mathfrak{Q}(X))$-measurable, thanks to Severini-Egoroff's theorem there exists a set $E_1\subseteq \tilde{U}$ such that $\mu(\tilde{U}\setminus E_1)\leq \alpha\mu(\tilde{U})/2$ and the convergence to $0$ of $t^{-1}\mathrm{dist}(\tilde{U},\mathcal{X}_j(t,x))$ is uniform on $E_1$.
\vspace{0.3cm}

Let us now suppose that we have constructed the curves $\gamma_{\delta_t(v)}$ where $M=k$. Let us construct $\gamma_{\delta_t(v)}$ when $M=k+1$. We let $\tilde{v}:=\delta_{\lambda_1}(\zeta_{i_1}(x))\ldots\delta_{\lambda_k}(\zeta_{i_k}(x))$. We are now assuming, as we indeed proved in the base step $M=1$, that the inductive process yields $\mu$-measurable subsets $E_k\subset E_{k-1}\subset \tilde U$, and for every  $t\in[-\min_i 1/|\lambda_i|,\min_i 1/|\lambda_i|]$, and $x\in E_k$, a Borel curve $\gamma_{\delta_t(\tilde{v})}:[0,k]\to U$ starting from $x$ such that
\begin{itemize}
    \item[(a')] For $\Delta_{k,x}(t)$ such that $\varphi(x)^{-1}\varphi(\gamma_{\delta_t(\tilde v)}(k))=\delta_t(\tilde{v})\Delta_{k,x}(t)$, we have $\lVert\Delta_{k,x}(t)\rVert/t$ converges uniformly to $0$ on {$x\in E_k$}, as $t$ goes to $0$;
    \item[(b')] There exists a function $\mathcal{R}_{k,x}$ of $t$ such that $d(x,\gamma_{\delta_t(\tilde{v})}(k))\leq 2(\sum_{i=1}^k|\lambda_i|)|t|+\mathcal{R}_{k,x}(t)$ where $\mathcal{R}_{k,x}(t)/t$ converges uniformly to $0$ on $x\in E_k$ as $t$ goes to $0$;
    \item[(c')] We have $\mu(E_{k-1}\setminus E_{k})\leq \alpha2^{-k}\mu(E_{k-1})$ and 
   $\lim_{t\to 0}t^{-1}\mathrm{dist}(\gamma_{\delta_t(\tilde v)}(k), \tilde{U})=0$ uniformly on $x\in E_{k}$. 
\end{itemize}
So, let us denote, for every $s\in [0,k+1]$, every $x\in E_{k}$, and every $t\in[-\min_i 1/|\lambda_i|,\min_i 1/|\lambda_i|]$,  
\begin{equation}
    \gamma_{\delta_t(v)}(s):=\begin{cases}
    \gamma_{\delta_t(\tilde{v})}(s) & \text{if $s\in[0,k]$}\\
   \mathcal{X}_{i_{k+1}}(\lambda_{k+1}t(s-k),\Pi_{{\tilde{U}}}(\gamma_{\delta_t(\tilde{v})}(k)))&\text{if $s\in[k,k+1]$,}
    \end{cases}
    \label{definizionegammatv}
\end{equation}
where $\Pi_{{\tilde{U}}}$ is the Borel metric projection onto ${\tilde{U}}$ given by Proposition~\ref{lemmamisurabilitaproiezione}. Note that the map $\gamma_{\delta_t(v)}$ is Borel since it is Borel on a Borel partition of $[0,k+1]$, and it starts from $x$.
 We now prove that the conditions (a') and (b') are satisfied for $\gamma_{\delta_t(v)}$ and with $k+1$ instead of $k$.
 
 Let us first verify (a'). Thanks to the definition of the curve, to the inductive hypothesis, and to \eqref{eqn:SomeUniformity}, we have that for all $x\in E_k$, and $s,t$ in the ranges above, the following equality holds
 \begin{equation}
 \begin{split}
      &\qquad\qquad\varphi(x)^{-1}\varphi(\gamma_{\delta_t(v)}(k+1))=\\
      &=\varphi(x)^{-1}\varphi(\gamma_{\delta_t(\tilde{v})}(k))\cdot\underbrace{\varphi(\gamma_{\delta_t(\tilde{v})}(k))^{-1} \varphi(\Pi_{{\tilde{U}}}(\gamma_{\delta_t(\tilde{v})}(k)))}_{\Theta_{k,x}(t)}
      \cdot\varphi(\Pi_{{\tilde{U}}}(\gamma_{\delta_t(\tilde{v})}(k)))^{-1}\varphi(\gamma_{\delta_t(v)}(k+1))\\
      &=\delta_t(\tilde{v})\Delta_{k,x}(t)\cdot \Theta_{k,x}(t)\cdot
      \delta_{\lambda_{k+1}t}(\zeta_{i_{k+1}}(\Pi_{{\tilde{U}}}(\gamma_{\delta_t(\tilde{v})}(k))))
      \Delta_{\lambda_{k+1},i_{k+1}}(\Pi_{{\tilde{U}}}(\gamma_{\delta_t(\tilde v)}(k)),t,1)\\
      &=\delta_t(\tilde{v})\delta_{\lambda_{k+1}t}(\zeta_{i_{k+1}}(x))\cdot\underbrace{  (\delta_{\lambda_{k+1}t}(\zeta_{i_{k+1}}(x)))^{-1}\delta_{\lambda_{k+1}t}(\zeta_{i_{k+1}}(\Pi_{{\tilde{U}}}(\gamma_{\delta_t(\tilde{v})}(k))))}_{\Delta_{k+1,x}^a(t)}\\
      &\cdot\underbrace{  (\delta_{\lambda_{k+1}t}(\zeta_{i_{k+1}}(\Pi_{{\tilde{U}}}(\gamma_{\delta_t(\tilde{v})}(k)))))^{-1}\Delta_{k,x}(t)\Theta_{k,x}(t)\delta_{\lambda_{k+1}t}(\zeta_{i_{k+1}}(\Pi_{{\tilde{U}}}(\gamma_{\delta_t(\tilde{v})}(k))))}_{\Delta_{k+1,x}^b(t)}\\
      &\cdot\underbrace{\Delta_{\lambda_{k+1},i_{k+1}}(\Pi_{{\tilde{U}}}(\gamma_{\delta_t(\tilde v)}(k)),t,1)}_{\Delta_{k+1,x}^c(t)}=\delta_t(v)\cdot\Delta_{k+1,x}^a(t)\Delta_{k+1,x}^b(t)\Delta_{k+1,x}^c(t).
      \nonumber
 \end{split}
 \end{equation}
 
 {Let us define $\Delta_{k+1,x}(t):=\Delta_{k+1,x}^a(t)\Delta_{k+1,x}^b(t)\Delta_{k+1,x}^c(t)$. The item (a') with $k+1$ will be verified if we show that $\|\Delta_{k+1,x}(t)\|/t\to 0$ as $t\to 0$ uniformly in $x\in E_k$. The latter uniform convergence holds true because 
 \begin{itemize}
     \item $\|\Delta_{k+1,x}^a(t)\|/t\to 0$ as $t\to 0$ uniformly in $x\in E_k$ because $\zeta_{i_{k+1}}$ is uniformly continuous on $\tilde U$ and $\Pi_{{\tilde{U}}}(\gamma_{\delta_t(\tilde v)}(k))\to x$ as $t\to 0$ uniformly on $x\in E_k$ as a consequence of the inductive step in item (c') and the estimate in item (b');
     \item $\|\Delta_{k+1,x}^b(t)\|/t\to 0$ as $t\to 0$ uniformly in $x\in E_k$ because $\|\Delta_{k,x}(t)\|/t\to 0$ as $t\to 0$ uniformly in $x\in E_k$ due to the inductive step in item (a'); $\|\Theta_{k,x}(t)\|/t\to 0$ as $t\to 0$ uniformly in $x\in E_k$ because $\varphi$ is Lipschitz and $t^{-1}\mathrm{dist}(\gamma_{\delta_t(\tilde v)}(k),{\tilde{U}})\to 0$ as $t\to 0$ uniformly in $x\in E_k$ due to the inductive step in item (c'); and because $\zeta_{i_{k+1}}$ is uniformly bounded on $\tilde U$;
     \item $\|\Delta_{k+1,x}^c(t)\|/t\to 0$ as $t\to 0$ uniformly in $x\in E_k$ as a consequence of \eqref{eqn:SomeUniformity}.
 \end{itemize}
 }

Hence item (a') is verified with $k+1$ instead of $k$, recalling that we will construct $E_{k+1}$ in such a way that $E_{k+1}\subset E_k$. Finally, let us check item (b') with $k+1$ instead of $k$. Thanks to the inductive step in item (b'), and \eqref{eqn:SomeUniformity2} we know that 
\begin{equation}
\begin{split}
        d(x_0,&\gamma_{\delta_t(v)}(k+1))\leq   d(x_0,\gamma_{\delta_t(\tilde{v})}(k))+d(\gamma_{\delta_t(\tilde{v})}(k),\gamma_{\delta_t(v)}(k+1))\\
    &\leq  2(\sum_{i=1}^{k}|\lambda_i|)|t|+\mathcal{R}_{k,x}(t)+\dist(\gamma_{\delta_t(\tilde v)}(k),{\tilde{U}})+2|\lambda_{k+1}t|+|t|\tilde{\mathcal{R}}_{1,\Pi_{{\tilde{U}}}(\gamma_{\delta_t(\tilde{v})}(k))}^{\lambda_{k+1},i_{k+1}}(t).
\end{split}
\end{equation}
Defined $$\mathcal{R}_{k+1,x}(t):=\mathcal{R}_{k,x}(t)+\dist(\gamma_{\delta_t(\tilde v)}(k),{\tilde{U}})+|t|\tilde{\mathcal{R}}_{1,\Pi_{{\tilde{U}}}(\gamma_{\delta_t(\tilde{v})}(k))}^{\lambda_{k+1},i_{k+1}}(t),$$ we see that $\mathcal{R}_{k+1,x}(t)/t\to 0$ as $t\to 0$ uniformly as $x\in E_k$, by also using the inductive step in item (c'). Hence item (b') is verified with $k+1$ instead of $k$, recalling that we will construct $E_{k+1}$ in such a way that $E_{k+1}\subset E_k$.

Finally, let us check item (c') with $k+1$ instead of $k$. For every $j\in\{1,\dots,n_1\}$, denote by $\tilde{\mathcal{X}}_j:[-1,1]\times  {\tilde{U}}\to {\tilde{U}}$ the map given by Proposition~\ref{prop:diff_impl_direct} applied to $\tilde U$. Therefore, for every $(t,x)\in [-1,1]\times\tilde U$,
\begin{equation}\label{eqn:DAJE}
    \begin{split}
        &\qquad\qquad\frac{\mathrm{dist}({\tilde{U}},\mathcal{X}_j(t,x))}{\lvert t\rvert}\leq \frac{d(\tilde{\mathcal{X}}_j(t,x),\mathcal{X}_j(t,x))}{\lvert t\rvert}\\
        &\leq\frac{d(\tilde{\mathcal{X}}_j(t,x),\chi_{[j,x]}(\tau_x+\tau(t,x)))}{\lvert t\rvert}+\frac{d(\mathcal{X}_j(t,x),\chi_{[j,x]}(\tau_x+\tau(t,x)))}{\lvert t\rvert}.
    \end{split}
\end{equation}
Now, recall that a function $f:X\to Y$ between metric spaces is called \textit{universally measurable}, if the preimage of every Borel set in $Y$ is measurable with respect to every Radon measure on $X$. Recall that the composition of universally measurable functions is universally measurable; see \cite[434D(f)]{FremlinVol4}.
Let us notice that $[-1,1]\times U\ni (t,x)\mapsto f(t,x):=\Pi_{{\tilde{U}}}(\gamma_{\delta_t(\tilde{v})}(k))$ is universally measurable on $P(U)$.

Indeed, by the inductive definition of the curve $\delta_t(\tilde v)$, $f(t,x)$ is the composition of several $\Pi_{\tilde U}$ and several $\mathcal{X}_j$. Then, notice that $\Pi_{\tilde U}$ is universally measurable since it is Borel, while $\mathcal{X}_j$ is universally measurable on $P(U)$. The last assertion comes by inspecting the proof of \cref{prop:curvefullmeas} (see \cite[Proposition~2.9]{BateJAMS}) from which we infer that $x\mapsto \gamma^x$ is universally measurable on $P(U)$, and by how we construct $\mathcal{X}_j$ in \cref{prop:misurbillity}. 

Reasoning analogously, see also the few lines after item (iv) above, we can also ensure that $(t,x)\mapsto \tau(t,f(t,x))$ is universally measurable. We can thus ensure that on a set $\tilde E\subset\tilde U$ with sufficiently big $\mu$-measure the pointwise convergences in \eqref{eq:unifdelicati} hold for every $x\in\tilde E$, where $x$ there is substituted with $f(t,x)$. Moreover the same estimate in \eqref{eq:unifdelicati}, \eqref{apprxderivate} holds with $\tilde{\mathcal{X}}_j$ instead of $\mathcal{X}_j$. We skip the details since they are similar to what we have done in other parts of the paper.

Thus, by using the latter information and \eqref{eqn:DAJE}, we have that $t^{-1}\mathrm{dist}({\tilde{U}},\mathcal{X}_j(t,\Pi_{{\tilde{U}}}(\gamma_{\delta_t(\tilde{v})}(k))))$ converges pointwise to $0$ at $\mu$-almost every point $x\in \tilde E$, where $x$ is the starting point of the curve $\gamma_{\delta_t(\tilde{v})}(k)$. 

Since the maps considered are universally measurable, as noticed above,  thanks to Severini-Egoroff's theorem, there exists a set $E_{k+1}\subseteq E_{k}\cap\tilde E$ such that $\mu(E_{k}\setminus E_{k+1})\leq \alpha\mu(E_k)/2^{k+1}$, and 
 $t^{-1}\mathrm{dist}({\tilde{U}},\mathcal{X}_j(t,\Pi_{{\tilde{U}}}(\gamma_{\delta_t(\tilde{v})}(k))))$ converges to $0$ uniformly on $x\in E_{k+1}$, where $x$ is the starting point of the curve $\gamma_{\delta_t(\tilde v)}$ as constructed in the inductive step. Thus, item (c') with $k+1$ instead of $k$ is verified by the very definition of the curve $\gamma_{\delta_t}(v)$.

\vspace{0.5cm}

Let us now conclude the proof of step 3, i.e., let us verify how items (a), (b), (c) are satisfied after we have proved the items (a'), (b'), and (c') for all $k=1,\dots,M$. Indeed, let us define $\tilde E:=\cap_{k=1}^M E_k=E_M$, which is $\mu$-measurable. As a consequence of the items (c') for every $k=1,\dots, M$ we conclude that 
\[
\mu(\tilde U\setminus \tilde E)\leq \sum_{k=1}^M \mu(E_{k-1}\setminus E_k) \leq \sum_{k=1}^M\alpha 2^{-k} \mu(E_{k-1})\leq \alpha\mu(\tilde U),
\]
and moreover items (a), (b), (c) with $\tilde E$ instead of $E$ directly follow from items (a'), (b'), (c') at $k=M$. Finally, the Borel regularity of $\mu$ immediately yields the desired compact set $E$.

\paragraph{Step 4.} Let us work with the notation in the previous steps. We have proved that, fixed $\alpha,\varepsilon>0$, there exist $E\subset U$ such that 
\[
\mu( U\setminus E)\leq (\alpha+\varepsilon)\mu(U),
\]
and all the assertions in items (ii), (iii), (iv), (a), (b), (c) in the previous step hold uniformly on $x\in E$.
In this step, renamed $\gamma_t:=\gamma_{\delta_t(v)}:[0,M]\to U$, we claim that for $\Sigma_{M,x}(t)$ defined by
\begin{equation}
    f(x)^{-1}f(\gamma_t(M))=\delta_t  (\delta_{\lambda_{1}}(\partial_{j_1}f(x))\cdot\ldots \cdot\delta_{\lambda_{M}}(\partial_{j_M}f(x)))\Sigma_{M,x}(t),
    \label{derivatadirezionale1}
\end{equation}
we have $\|\Sigma_{M,x}(t)\|/t\to 0$ as $t\to 0$ uniformly on $x\in E$. Let us prove it by induction on $M$. If $M=1$, there is nothing to prove since the claim directly follows from the second of \eqref{apprxderivate}. Let us now assume \eqref{derivatadirezionale1} holds for $M=k$, and let us prove it for $M=k+1$. 
{
Denoting $\tilde{v}:=\delta_{\lambda_1}(\zeta_{i_1}(x))\ldots\delta_{\lambda_k}(\zeta_{i_k}(x))$ as in Step 3, we
have, by using the inductive hypothesis, the very definition of $\gamma_{\delta_t(\tilde v)}$, and \eqref{apprxderivate} when we introduce $\tilde \Delta_{k,x}(t)$ in the following computations,
\begin{equation*}
    \begin{split}
        f(x)^{-1}&f(\gamma_t(k+1)) = f(x)^{-1}f(\gamma_t(k))\cdot f(\gamma_t(k))^{-1}f(\gamma_t(k+1))\\
        &=\delta_t  (\delta_{\lambda_{1}}(\partial_{j_1}f(x))\cdot\ldots \cdot\delta_{\lambda_{k}}(\partial_{j_k}f(x)))\Sigma_{k,x}(t)\cdot f(\gamma_{\delta_t(\tilde v)}(k))^{-1}f(\mathcal{X}_{i_{k+1}}(\lambda_{k+1}t,\Pi_{\tilde U}(\gamma_{\delta_t(\tilde v)}(k))))\\
        &=\delta_t  (\delta_{\lambda_{1}}(\partial_{j_1}f(x))\cdot\ldots \cdot\delta_{\lambda_{k}}(\partial_{j_k}f(x))) \delta_t(\delta_{\lambda_{k+1}}(\partial_{j_{k+1}}f(x))) \\
        &\qquad\qquad  \cdot (\delta_{t\lambda_{k+1}}(\partial_{j_{k+1}}f(x)))^{-1}\Sigma_{k,x}(t)\cdot \underbrace{f(\gamma_{\delta_t(\tilde v)}(k))^{-1} f(\Pi_{\tilde U}(\gamma_{\delta_t(\tilde v)}(k)))}_{\Theta_{k+1,x}^a(t)}\\
        &\qquad \qquad \qquad\qquad \cdot\underbrace{f(\Pi_{\tilde U}(\gamma_{\delta_t(\tilde v)}(k)))^{-1}f(\mathcal{X}_{i_{k+1}}(\lambda_{k+1}t,\Pi_{\tilde U}(\gamma_{\delta_t(\tilde v)}(k))))}_{\delta_{t\lambda_{k+1}}(\partial_{j_{k+1}}f(\Pi_{\tilde U}(\gamma_{\delta_t(\tilde v)}(k))))\tilde\Delta_{k,x}(t)}\\
        &=\delta_t  (\delta_{\lambda_{1}}(\partial_{j_1}f(x))\cdot\ldots \cdot\delta_{\lambda_{k}}(\partial_{j_k}f(x))\delta_t(\delta_{\lambda_{k+1}}(\partial_{j_{k+1}}f(x))))\\
        &\qquad\cdot \underbrace{  (\delta_{t\lambda_{k+1}}(\partial_{j_{k+1}}f(x)))^{-1}\delta_{t\lambda_{k+1}}(\partial_{j_{k+1}}f(\Pi_{\tilde U}(\gamma_{\delta_t(\tilde v)}(k))))}_{\Sigma_{k+1,x}^a(t)}\\
        &\,\qquad\cdot \underbrace{\delta_{t\lambda_{k+1}}(\partial_{j_{k+1}}f(\Pi_{\tilde U}(\gamma_{\delta_t(\tilde v)}(k))))^{-1}\Sigma_{k,x}(t)\Theta_{k+1,x}^a(t)\delta_{t\lambda_{k+1}}(\partial_{j_{k+1}}f(\Pi_{\tilde U}(\gamma_{\delta_t(\tilde v)}(k))))}_{\Sigma_{k+1,x}^b(t)}\cdot \tilde\Delta_{k,x}(t),
    \end{split}
\end{equation*}
and thus the inductive process is concluded if we show that $\|\Sigma_{k+1,x}^a(t)\Sigma_{k+1,x}^b(t)\tilde\Delta_{k,x}(t)\|/t\to 0$ as $t\to 0$ uniformly on $x\in E$. The latter assertion is true because:
\begin{itemize}
    \item $\|\Sigma_{k+1,x}^a(t)\|/t\to 0$ as $t\to 0$ uniformly on $x\in E$, due to the fact that $x\mapsto \partial_{j_{k+1}}f(x)$ is uniformly continuous on $\tilde U\supset E$, and $\Pi_{\tilde U}(\gamma_{\delta_t(\tilde v)}(k))\to x$ as $t\to 0$ uniformly on $x\in E$, as a consequence of the estimate in item (b') of the induction in Step 3, and of the fact that $t^{-1}\mathrm{dist}(\gamma_{\delta_t(\tilde v)}(k),\tilde U)\to 0$ as $t\to 0$ uniformly on $x\in E_k\supset E$, due to item (c') of the induction in Step 3.
    \item $\|\Sigma_{k+1,x}^b(t)\|/t\to 0$ as $t\to 0$ due to the following three facts. The map $x\mapsto \partial_{j_{k+1}}f(x)$ is uniformly bounded on $x\in \tilde U\supset E$; $\|\Sigma_{k,x}(t)\|/t\to 0$ as $t\to 0$ uniformly on $x\in E$ by the inductive hypothesis; $\|\Theta_{k+1,x}^a(t)\|/t\to 0$ as $t\to 0$ uniformly on $x\in E$ because $f$ is Lipschitz and $t^{-1}\mathrm{dist}(\gamma_{\delta_t(\tilde v)}(k),\tilde U)\to 0$ as $t\to 0$ uniformly on $x\in E_k\supset E$, due to item (c') of the induction in Step 3.
    \item $\|\tilde\Delta_{k,x}(t)\|/t\to 0$ as $t\to 0$ uniformly on $x\in E$ due to the uniform convergence in the second line of \eqref{apprxderivate}.
\end{itemize}
}
For some benefit towards subsequent computations, let us register a consequence of \eqref{derivatadirezionale1}. Thanks to item (a) proved in step 3 and the Lipschitzianity of $\varphi$, we infer  that
\begin{equation}
    t\lVert v\rVert/2\leq t\lVert v\rVert-\lVert\Delta_{M,x}(t)\rVert\leq \lVert \varphi(x)^{-1}\varphi(\gamma_t(M))\rVert\leq \mathrm{Lip}(\varphi)d(x,\gamma_t(M)), 
    \label{costruzionefattabene}
\end{equation}
where the first inequality follows provided $t$ is small enough and note that such smallness can be chosen to be uniformly on $E$ thanks to item (a) in Step 3 Thus, as a direct consequence of \eqref{costruzionefattabene} and \eqref{derivatadirezionale1} we get that
\begin{equation}
    \lim_{t\to 0}\delta_{1/d(x,\gamma_t(M))}
    \left(\delta_t(\delta_{\lambda_{1}}(\partial_{j_1}f(x))\ldots \delta_{\lambda_{M}}(\partial_{j_M}f(x)))^{-1}f(x)^{-1}f(\gamma_t(M))\right)=0,
    \label{derivatadirezionale}
\end{equation}
 uniformly in $x\in E$. The previous equality will be exploited in the forthcoming steps.

\paragraph{Step 5.} In the computations of this step, we already assume that $Df(x_0)$ is a homogeneous homomorphism and that it is well-defined. This will be proved in the last step of the proof.

Let $v\in \mathbb{G}$. For $t>0$ let $\gamma_t:=\gamma_{\delta_t(v)}$ be the curve constructed in 
Step 3. Then, as a consequence of \eqref{derivatadirezionale}, we have that, uniformly in $x_0\in E$,
\begin{equation}
    \begin{split}
        \lim_{t\to 0}\delta_{1/d(x_0,\gamma_t(M))}  ((Df(x_0)(\varphi(x_0)^{-1}\varphi(\gamma_t(M))))^{-1}\cdot f(x_0)^{-1}f(\gamma_t(M)))=0.
        \label{difflungov}
    \end{split}
\end{equation}
 Indeed, by using item (a) of Step 3, we have that, uniformly in $x_0\in E$,
\begin{equation*}
    \begin{split}
         &\lim_{t\to 0}\delta_{1/d(x_0,\gamma_t(M))}  ((Df(x_0)(\varphi(x_0)^{-1}\varphi(\gamma_t(M))))^{-1}\cdot f(x_0)^{-1}f(\gamma_t(M)))\\
         &= \lim_{t\to 0}\delta_{1/d(x_0,\gamma_t(M))}(Df(x_0)(\Delta_{M,x_0}(t)))^{-1}\lim_{t\to 0}\delta_{1/d(x_0,\gamma_t(M))}(\delta_t(Df(x_0)(v))^{-1}f(x_0)^{-1}f(\gamma_t(M)))\\
         &=\lim_{t\to 0}\delta_{t/d(x_0,\gamma_t(M))}(Df(x_0)(v)^{-1}\delta_{1/t}(f(x_0)^{-1}f(\gamma_t(M))))=0,
    \end{split}
\end{equation*}
where in the second equality above we used that $\|\Delta_{M,x_0}(t)\|/t\to 0$ as $t\to 0$ uniformly in $x\in E$ by item (a) of Step 3, and \eqref{costruzionefattabene}; and in the last equality we used \eqref{derivatadirezionale}.

\paragraph{Step 6.} Let us prove that if $Df(x_0)$ is well defined for $\mu$-almost every $x_0\in E$, then \eqref{eqn:PANSUDIFF} holds. Hence, taking $\alpha,\varepsilon\to 0$ in Step 3, see the beginning of Step 4, we conclude that \eqref{eqn:PANSUDIFF} holds at $\mu$-almost every $x_0\in U$. 

Let us assume by contradiction that there exists a $\delta>0$ such that on a compact subset $K\ni x_0$ of $\mu$-positive measure of $E$ we have 
\begin{equation}
\limsup_{X\ni x\to x_0}\frac{\Big\|\big(Df(x_0)(\varphi(x_0)^{-1}\cdot\varphi(x))\big)^{-1}\cdot f(x_0)^{-1}\cdot f(x)\Big\|_{\mathbb H}}{d(x_0,x)}>\delta,
\label{equationabsurd}
\end{equation}

We thus can pick a sequence $\{y_k\}_{k\in\N}\subseteq X$ such that $\lim_{k\to\infty} d(y_k,x_0)=0$ for which 
\begin{equation}
    \lim_{k\to \infty}\delta_{1/d(y_k,x_0)}(\varphi(x_0)^{-1}\varphi(y_k))=v\in\mathbb{G},
    \label{eq:vsceltay_k}
\end{equation}
and 
\begin{equation}
\lim_{k\to\infty}\frac{\Big\|\big(Df(x_0)(\varphi(x_0)^{-1}\cdot\varphi(y_k))\big)^{-1}\cdot f(x_0)^{-1}\cdot f(y_k)\Big\|_{\mathbb H}}{d(x_0,y_k)}>\delta.
\label{equationabsurd2}
\end{equation}
Thus, for every $k\in\N$ sufficiently large, let $\gamma_k:[0,M]\to U$ be the curve relative to $\delta_{d(x_0,y_k)}(v)$ constructed in Step 3, starting from $x_0$. This implies, by using \eqref{eq:vsceltay_k}, that
\begin{equation}\label{eqn:EquazioneMaledetta}
\begin{split}
     &Df(x_0)(v)^{-1}\lim_{k\to\infty}\delta_{1/d(x_0,y_k)}(f(x_0)^{-1} f(y_k))\\
     =&\lim_{k\to\infty}\delta_{1/d(x_0,y_k)}(\big(Df(x_0)(\varphi(x_0)^{-1}\varphi(y_k))\big)^{-1} f(x_0)^{-1} f(y_k))\\
    =&Df(x_0)\Big(\lim_{k\to\infty}\delta_{1/d(x_0,y_k)}(\varphi(\gamma_k(M))^{-1}\varphi(y_k))\Big)^{-1}\\
    &\qquad\qquad\qquad\lim_{k\to\infty}\delta_{1/d(x_0,y_k)}(\big(Df(x_0)(\varphi(x_0)^{-1}\varphi(\gamma_k(M)))\big)^{-1} f(x_0)^{-1} f(\gamma_k(M)))       \\
    &\qquad\qquad\qquad\qquad\qquad\qquad\lim_{k\to\infty}\delta_{1/d(x_0,y_k)}(f(\gamma_k(M))^{-1}f(y_k)).
\end{split}
\end{equation}
Notice that by means of item (b) of Step 3 and \eqref{costruzionefattabene}, we get that there exists a constant $C>1$ such that  for $k$ large enough we have
\begin{equation}\label{eqn:OrderMagnitude}
C^{-1}\leq d(x_0,\gamma_k(M))/d(x_0,y_k)\leq C.
\end{equation}
Thanks to item (a) in step 3 and to \eqref{eq:vsceltay_k}, we infer
\begin{equation}
    \begin{split}
    &\lim_{k\to\infty}\delta_{1/d(x_0,y_k)}(\varphi(\gamma_k(M))^{-1}\varphi(y_k))\\
    &=\lim_{k\to\infty} \delta_{1/d(x_0,y_k)}(\varphi(x_0)^{-1}\varphi(\gamma_k(M)))^{-1}\delta_{1/d(x_0,y_k)}(\varphi(x_0)^{-1}\varphi(y_k))\\
    &=\lim_{k\to\infty} (v\cdot\delta_{1/d(x_0,y_k)}(\Delta_{M,x_0}(d(x_0,y_k))))^{-1}\cdot v=0.
    \label{equationsmallphi}
    \end{split}
\end{equation}

On the other hand, let us estimate
\begin{equation}\label{eqn:DiamoleUnNome}
\begin{split}
      \lim_{k\to \infty}\Big\lVert \delta_{1/d(x_0,y_k)}(f(\gamma_k(M))^{-1}f(y_k))\Big\rVert&= \lim_{k\to \infty}\frac{d(\gamma_k(M),y_k)}{d(x_0,y_k)}\Big\lVert \delta_{1/d(\gamma_k(M),y_k)}(f(\gamma_k(M))^{-1}f(y_k))\Big\rVert\\
      \leq &(1+C)\lim_{k\to \infty}\Big\lVert \delta_{1/d(\gamma_k(M),y_k)}(f(\gamma_k(M))^{-1}f(y_k))\Big\rVert.
\end{split}
\end{equation}
Our aim is now to show that the limit in \eqref{eqn:EquazioneMaledetta} is equal to $0$. In order to do so, we distinguish two cases. First, assume that $\lim_{k\to\infty}d(y_k,\gamma_k(M))/d(x_0,y_k)=0$. If this is the case, this would imply that 
\begin{equation}
\begin{split}
     &\lim_{k\to\infty}\delta_{1/d(x_0,y_k)}(\big(Df(x_0)(\varphi(x_0)^{-1}\varphi(y_k))\big)^{-1} f(x_0)^{-1} f(y_k))\\
     &=\lim_{k\to\infty}\delta_{1/d(x_0,y_k)}(Df(x_0)(\varphi(\gamma_k(M))^{-1}\varphi(y_k)))^{-1}\\
     &\qquad\qquad\cdot\delta_{1/d(x_0,y_k)}(\big(Df(x_0)(\varphi(x_0)^{-1}\varphi(\gamma_k(M)))\big)^{-1} f(x_0)^{-1} f(\gamma_k(M)))\\
     &\qquad\qquad\qquad\qquad\cdot\delta_{1/d(x_0,y_k)}(f(\gamma_k(M))^{-1}f(y_k))\\
     &\leq \lVert Df(x_0)\rVert\mathrm{Lip}(\varphi)\lim_{k\to\infty}d(y_k,\gamma_k(M))/d(x_0,y_k)\\
      &\qquad\qquad\cdot\lim_{k\to\infty}\delta_{1/d(x_0,y_k)}(\big(Df(x_0)(\varphi(x_0)^{-1}\varphi(\gamma_k(M)))\big)^{-1} f(x_0)^{-1} f(\gamma_k(M)))\\
      &\qquad\qquad\qquad\qquad\cdot\mathrm{Lip}(f)\lim_{k\to\infty}d(y_k,\gamma_k(M))/d(x_0,y_k)=0,
\end{split}
\nonumber
\end{equation}
where the last identity follows from \eqref{difflungov}, \eqref{eqn:OrderMagnitude}, our hypothesis in this first case \[\lim_{k\to\infty}d(y_k,\gamma_k(M))/d(x_0,y_k)=0,\] and the Lipschitzianity of $\varphi$ and $f$.
On the other hand, in the second case, along every subsequence $y_k$ for which $\lim_{k\to\infty}d(y_k,\gamma_k(M))/d(x_0,y_k)>\delta>0$, we have
\begin{equation*}
    \begin{split}
        \lim_{k\to\infty}\|\delta_{1/d(\gamma_k(M),y_k)}(\varphi(\gamma_k(M))^{-1}\varphi(y_k))\|&=\lim_{k\to\infty}\frac{d(x_0,y_k)}{d(\gamma_k(M),y_k)}\|\delta_{1/d(x_0,y_k)}(\varphi(\gamma_k(M))^{-1}\varphi(y_k))\|\\
        &\leq \delta^{-1}\lim_{k\to+\infty}\|\delta_{1/d(x_0,y_k)}(\varphi(\gamma_k(M))^{-1}\varphi(y_k))\|=0,
    \end{split}
\end{equation*}
where in the last equality we are using \eqref{equationsmallphi}. Hence, in this second case, thanks to the completeness of the chart $(U,\varphi)$ see \cref{def:GHcomplete},  and to the previous equality we infer that also
\begin{equation*}
    \lim_{k\to \infty}\Big\lVert \delta_{1/d(\gamma_k(M),y_k)}(f(\gamma_k(M))^{-1}f(y_k))\Big\rVert=0,
    \nonumber
\end{equation*}
and then, by \eqref{eqn:DiamoleUnNome}, we conclude that
\begin{equation}\label{eqn:PrevPrev}
    \lim_{k\to \infty}\Big\lVert \delta_{1/d(x_0,y_k)}(f(\gamma_k(M))^{-1}f(y_k))\Big\rVert=0.
\end{equation}
From \eqref{eqn:PrevPrev}, together with \eqref{difflungov} in step 5, \eqref{equationsmallphi}, \eqref{eqn:OrderMagnitude}, and \eqref{eqn:EquazioneMaledetta}, we finally conclude that also in this case
\begin{equation}
    \lim_{k\to\infty}\delta_{1/d(x_0,y_k)}(\big(Df(x_0)(\varphi(x_0)^{-1}\varphi(y_k))\big)^{-1} f(x_0)^{-1} f(y_k))=0,
\end{equation}
and this comes in contradiction with \eqref{equationabsurd2}. Notice that the uniqueness of $Df(x_0)$ follows from the forthcoming \cref{rk_identity}. This concludes the proof.

\paragraph{Step 7.} We are left to show that $Df(x_0)$ is well defined. In order to prove this, we write $v$ in two different ways
$$\delta_{\eta_1}(\zeta_{j_1}(x_0))\cdot\ldots\cdot\delta_{\eta_N}(\zeta_{j_N}(x_0))=v=\delta_{\lambda_1}(\zeta_{i_1}(x_0))\cdot\ldots\cdot\delta_{\lambda_M}(\zeta_{i_M}(x_0)),$$
and we let 
\begin{equation}
    \begin{split}
        \partial_1&:=\delta_{\lambda_1}(\partial_{i_1} f(x_0))\cdot\ldots\cdot\delta_{\lambda_M}(\partial_{i_M} f(x_0)).\\
        \partial_2&:=\delta_{\eta_1}(\partial_{j_1} f(x_0))\cdot\ldots\cdot\delta_{\eta_N}(\partial_{j_N} f(x_0)).
        \nonumber
    \end{split}
\end{equation}

We associate to each expression for $v$ the Lipschitz curves $\gamma_{t}^1$ and $\gamma_{t}^2$ yielded by Step 3. {We have that
\begin{equation}\label{eqn:Lunga}
    \begin{split}
       &\qquad\qquad \qquad 0= \limsup_{t\to 0}\delta_{1/d(x_0,\gamma_t^1(M))}(\delta_t(\partial_1)^{-1}f(x_0)^{-1}f(\gamma_t^1(M)))=\\
        =&\limsup_{t\to 0}\delta_{t/d(x_0,\gamma_t^1(M))}(\partial_1^{-1}\partial_2)\delta_{d(x_0,\gamma_t^2(M))/d(x_0,\gamma_t^1(M))}  (\delta_{1/d(x_0,\gamma_t^2(M))}(\delta_t(\partial_2)^{-1}f(x_0)^{-1}f(\gamma_t^2(M))))\\
         &\qquad\qquad\qquad\qquad\cdot \delta_{d(x_0,\gamma_t^2(M))/d(x_0,\gamma_t^1(M))}  (\delta_{1/d(x_0,\gamma_t^2(M))}(f(\gamma_t^2(M))^{-1}f(\gamma_t^1(M))))\\
         &=\limsup_{t\to 0}\delta_{t/d(x_0,\gamma_t^1(M))}(\partial_1^{-1}\partial_2)\cdot \\
         &\qquad\qquad\qquad\cdot \delta_{d(x_0,\gamma_t^2(M))/d(x_0,\gamma_t^1(M))}  (\delta_{1/d(x_0,\gamma_t^2(M))}(f(\gamma_t^2(M))^{-1}f(\gamma_t^1(M)))).
    \end{split}
\end{equation}
where in the first equality we are using \eqref{derivatadirezionale} for $\gamma_t^1$, the second equality is just an algebraic re-writing of the terms, and in the third equality we are using \eqref{derivatadirezionale} for $\gamma_t^2$ together with the fact that, as a consequence of the analog of \eqref{eqn:OrderMagnitude} applied to $\gamma_t^1,\gamma_t^2$, for $t$ small enough, there exists a constant $\tilde C>1$ such that 
\begin{equation}\label{eqn:OrderMagnitude2}
\tilde C^{-1}\leq\frac{d(x_0,\gamma_t^2(M))}{d(x_0,\gamma_t^1(M))}\leq \tilde C.
\end{equation}
Our aim is now to show that 
\begin{equation}\label{eqn:TS}
\limsup_{t\to 0} \delta_{1/d(x_0,\gamma_t^1(M))}(f(\gamma_t^2(M))^{-1}f(\gamma_t^1(M)))=0,
\end{equation}
and, in order to do so, we will distinguish two cases as above.
First, let us assume that $\lim_{t\to 0} d(\gamma_t^1(M),\gamma_t^2(M))/d(x_0,\gamma_t^1(M))=0$. In this case
\[
\limsup_{t\to 0}\|\delta_{1/d(x_0,\gamma_t^1(M))}(f(\gamma_t^2(M))^{-1}f(\gamma_t^1(M)))\|\leq \mathrm{Lip}(f)\lim_{t\to 0}\frac{d(\gamma_t^1(M),\gamma_t^2(M))}{d(x_0,\gamma_t^1(M))}=0,
\]
and then, from the previous inequality we get the sought \eqref{eqn:TS} in the first case. In the second case, we have, along a subsequence $t_k\to 0$, that 
\begin{equation}\label{eqn:AssuMpt}
\lim_{k\to +\infty}\frac{d(\gamma_{t_k}^1(M),\gamma_{t_k}^2(M))}{d(x_0,\gamma_{t_k}^1(M))}>\delta>0.
\end{equation}
First, let us notice that, as a consequence of item (a) in Step 3, we have 
\begin{equation}\label{eqn:Deduci}
\begin{split}
\varphi(\gamma^2_{t_k}(M))^{-1}\varphi(\gamma_{t_k}^1(M)) &=   (\varphi(x_0)^{-1}\varphi(\gamma^2_{t_k}(M)))^{-1}  (\varphi(x_0)^{-1}\varphi(\gamma_{t_k}^1(M)))\\
&=  (\delta_{t_k}(v)\Delta^2_{M,x_0}(t_k))^{-1}  (\delta_{t_k}(v)\Delta^1_{M,x_0}(t_k))=\Delta^2_{M,x_0}(t_k)^{-1}\Delta^1_{M,x_0}(t_k).
\end{split}
\end{equation}
Moreover, since from \eqref{costruzionefattabene} we have that $t_k/d(x_0,\gamma_{t_k}^1(M))$ is uniformly bounded above for $k$ big enough, and since we also have \eqref{eqn:AssuMpt} and that $\|\Delta^i_{M,x_0}(t_k)\|/t_k\to 0$ for $i=1,2$ as $k\to +\infty$, from the previous \eqref{eqn:Deduci} we get that
\begin{equation}\label{eqn:Deduci2}
\begin{split}
\lim_{k\to +\infty}\delta_{1/d(\gamma_{t_k}^2(M),\gamma_{t_k}^1(M))}  (\varphi(\gamma^2_{t_k}(M))^{-1}\varphi(\gamma_{t_k}^1(M))) = 0.
\end{split}
\end{equation}
From \eqref{eqn:Deduci2} and the fact that the chart $(U,\varphi)$ is complete, see \cref{def:GHcomplete}, we deduce that 
\begin{equation}\label{eqn:Deduci3}
\begin{split}
\lim_{k\to +\infty}\delta_{1/d(\gamma_{t_k}^2(M),\gamma_{t_k}^1(M))}  (f(\gamma^2_{t_k}(M))^{-1}f(\gamma_{t_k}^1(M))) = 0,
\end{split}
\end{equation}
and finally, since as a consequence of \eqref{eqn:OrderMagnitude2} and the triangle inequality we have that \[d(\gamma_{t_k}^1(M),\gamma_{t_k}^2(M))/d(x_0,\gamma_{t_k}^1(M))\] is uniformly bounded above for $k$ big enough, we finally deduce also in this case, from \eqref{eqn:Deduci3}, the sought \eqref{eqn:TS}.

From \eqref{eqn:Lunga} and \eqref{eqn:TS} we finally get
\begin{equation}
    0=\lim_{t\to 0}\delta_{t/d(x_0,\gamma_t^1(M))}(\partial_1^{-1}\partial_2),
\end{equation}
which concludes that $\partial_1^{-1}\partial_2=0$ since $t/d(x_0,\gamma_t^1(M))$ is uniformly bounded below as $t\to 0$ as a consequence of item (b) in Step 3. This finally proves the well-posedness of the definition of $Df(x_0)$.}
 \end{proof}

{
\begin{osservazione}\label{rk_identity}
In this remark, we prove, as anticipated in Step 6 of \cref{thm:Fondamentale2}, that in the hypotheses of \cref{thm:Fondamentale2} the homogeneous homomorphism $Df(x_0)$ for which \eqref{eqn:PANSUDIFF} holds is unique.

By using the Step 3 in \cref{thm:Fondamentale2}, see in particular item (a) and item (b) in there, we have that, for every $\varepsilon>0$, there exists $\tilde U\subset U$ with $\mu(U\setminus\tilde U)\leq \varepsilon\mu (U)$ such that the following holds. For every $x\in \tilde U$, every $v\in\mathbb G$, and every $t>0$ small enough there is $x_t\in X$ such that $x_t\to x$ as $t\to 0$, and, for every $t>0$ small enough, being $\Delta_x(t)$ defined by
\begin{equation}
\varphi(x)^{-1}\varphi(x_t)=\delta_t(v)\Delta_x(t),
\end{equation}
we have $\|\Delta_x(t)\|/t\to 0$ as $t\to 0$.

 Let now $L_1,L_2$ be two homogeneous homomorphisms satisfying \eqref{eqn:PANSUDIFF} at $x\in \tilde U$. Let us fix $v>0$.
Then 
\begin{equation}
    \begin{split}
        \lVert L_1(v)^{-1}&L_2(v)\rVert=\lim_{t\to 0}\frac{ \lVert L_1(\Delta_x(t))\cdot L_1(\delta_{t}(v)\Delta_x(t))^{-1}\cdot L_2(\delta_{t}(v)\Delta_x(t))\cdot L_2(\Delta_x(t))^{-1}\rVert}{t}\\
        &\qquad\qquad\leq\lim_{t\to 0}\frac{\lVert L_1(\varphi(x)^{-1}\varphi(x_t))^{-1}L_2(\varphi(x)^{-1}\varphi(x_t))\rVert}{t}\\
        \leq &\lim_{t\to 0}\frac{\lVert L_1(\varphi(x)^{-1}\varphi(x_t))^{-1}f(x)^{-1}f(x_t)\rVert+\lVert L_2(\varphi(x)^{-1}\varphi(x_t))^{-1}f(x)^{-1}f(x_t)\rVert}{d(x,x_t)}\frac{d(x,x_t)}{t}=0,
    \end{split}
\end{equation}
where in the last equality we exploited \eqref{eqn:PANSUDIFF} and the estimate in item (b) of Step 3 of \cref{thm:Fondamentale2}. Hence $L_1(v)=L_2(v)$. Since this holds for every $v\in\mathbb G$ we get that at $\mu$-almost every $x\in\tilde U$ the homogeneous differential is unique. Taking $\varepsilon\to 0$, we get that the differential is unique for $\mu$-almost every $x\in U$.
\end{osservazione}

}

\section{Pansu differentiability spaces iff horizontally independent Alberti representations}\label{sec4}

We are now ready to state the following equivalence theorem, which is the analog of \cite[Theorem~7.8]{BateJAMS} in the setting of Pansu differentiability spaces. We stress that as soon as we drop the third bullet in the second item of \cref{chardiffspaces}, the result is no longer true, see \cref{rem:NotDropComplete}. Notice that in the case of Lipschitz differentiability spaces, the third bullet in item 2 below is not needed as a consequence of the results of Bate \cite{BateJAMS}. The condition on the completeness of the chart is controlling, in some sense, the $f$-speed of the Alberti representations in the vertical directions. Up to restriction, the horizontal $f$-speed is controlled by the universality assumption, but this gives no control on the total speed, a priori. See also the discussion in the Introduction right after \cref{chardiffspacesINTRO}.

\begin{teorema}\label{chardiffspaces}
Let $(X,d,\mu)$ be a metric measure space. Let $\mathbb G$ and $\mathbb H$ be Carnot groups, and let us denote with $n_1$ the dimension of the first stratum of $\mathbb G$.  
Then the following are equivalent:
\begin{enumerate}
    \item $(X,d,\mu)$ is a $(\mathbb G,\mathbb H)$-differentiability space, see \cref{def:PansuDifferentiabilitySpace}.
    \item There exist countably many Lipschitz functions $\varphi_i:X\to\mathbb G$, and countably many Borel subsets $U_i$ of $X$ such that 
    \begin{itemize}
        \item $\mu(U\setminus \cup_{i\in\mathbb N} U_i)=0$;
        \item for every $i\in\mathbb N$ there exists $\rho_i>0$ and $\mathcal{A}_1,\dots,\mathcal{A}_{n_1}$ $\varphi_i$-independent, horizontally $\rho_i$-universal Alberti representations of $\mu\llcorner U_i$;
        \item for every $i\in\mathbb N$ and every Lipschitz function $f:X\to\mathbb H$, there exist Borel sets $U_i^j\subset U_i$ with $\mu(U_i\setminus \cup_{j\in\mathbb N}U_i^j)=0$ such that, for every $j\in\mathbb N$, the pair $(U_i^j,\varphi_i)$ is $(\mathbb G,\mathbb H)$-complete with respect to $f$, see \cref{def:GHcomplete}.
    \end{itemize}
\end{enumerate}
    Moreover, if $(X,d,\mu)$ is a $(\mathbb G,\mathbb H)$-differentiability space, then one can choose a system of charts $\{(U_i,\varphi_i)\}_{i\in\mathbb N}$ for $(X,d,\mu)$ such that the three bullets of item 2 hold.
\end{teorema} 

\begin{proof}
Let us first prove at once that 1 $\Rightarrow$ 2 and that the last part of the statement holds. Recall that $(X,d,\mu)$ is an $n_1$-Lipschitz differentiability space by considering the projections of the charts on the first stratum, see \cref{rem:Importante}. Then, let us take first a system of charts $\{(U_i,\pi_{\mathbb G}\circ\varphi_i)\}_{i\in\mathbb N}$ as in \cref{propcartestrutturate}, which means that $\pi_{\mathbb G}\circ\varphi_i$ is a structured chart in the terminology of \cite{BateJAMS}, when $(X,d,\mu)$ is seen an $n_1$-Lipschitz differentiability space. Now we can argue as in the second part of \cite[Theorem~7.8]{BateJAMS} to obtain that, up to possibly refining the covering, there exist charts $\{(U_i,\pi_{\mathbb G}\circ\varphi_i)\}_{i\in\mathbb N}$ such that, for every $i\in\mathbb N$, there exist $\mathcal{A}_1,\dots,\mathcal{A}_{n_1}$ $(\pi_{\mathbb G}\circ\varphi_i)$-independent Alberti representations of $\mu\llcorner U_i$ that are $\rho_i$-universal when $(X,d,\mu)$ is seen an $n_1$-Lipschitz differentiability space.

Now, recall that if $f:X\to\mathbb H$ is Lipschitz, hence $\pi_{\mathbb H}\circ f:X\to \mathbb V_1(\mathbb H)$ is Lipschitz, and if $\gamma\in \Gamma(X)$ is such that $D(f\circ\gamma)(t_0)$ exists for $t_0\in\mathrm{Dom}(\gamma)$, hence $D(f\circ\gamma)(t_0)=(\pi_{\mathbb H}\circ f\circ \gamma)'(t_0)$. This means, by exploiting the very definitions, that $\{(U_i,\varphi_i)\}_{i\in\mathbb N}$ is a system of charts such that for every $i\in\mathbb N$ there exists $\rho_i>0$ and $\mathcal{A}_1,\dots,\mathcal{A}_{n_1}$ $\varphi_i$-independent, horizontally $\rho_i$-universal Alberti representations of $\mu\llcorner U_i$. 

Finally, since the charts are chosen as in \cref{propcartestrutturate}, a direct application of \cref{prop:GoingTo0InChart}, see \cref{rem:GoingTo0}, gives the last bullet of item 2.
\medskip

Let us finally prove that 2 $\Rightarrow$ 1. Let us show that $\{(U_i,\varphi_i)\}$ is a system of charts that makes $(X,d,\mu)$ a $(\mathbb G,\mathbb H)$-differentiability space. If it is not so, there exist $f$ and $i\in\mathbb N$ such that item 3 of \cref{def:PansuDifferentiabilitySpace} does not hold. Then, let us take $U_i^j\subset U_i$ as in the assumptions such that $\{(U_i^j,\varphi_i)\}_{j\in\mathbb N}$ is a collection of $(\mathbb G,\mathbb H)$-complete pairs with respect to $f$. We have that for every $j\in\mathbb N$, by restriction, $\mathcal{A}_1,\dots,\mathcal{A}_{n_1}$ give raise to $\varphi_i$-independent, horizontally $\rho_i$-universal Alberti representations of $\mu\llcorner U_i^{j}$. Then by applying \cref{thm:Fondamentale2} we deduce that \eqref{eqn:PANSUDIFF} holds $\mu$-almost everywhere on $U_i^j$ for every $j$, and hence $\mu$-almost everywhere on $U_i$, which is a contradiction.
\end{proof}

\section{Cheeger property on Pansu differentiability spaces}\label{sec5}

This section is devoted to proving the following result, which is a more refined version of Cheeger's conjecture \cite[Conjecture 4.63]{CheegerGAFA} whose proof was obtained in \cite{DePhilippisMarcheseRindler}.

\begin{teorema}\label{cheegerpansu}
Let $(X,d,\mu)$ be a Pansu-differentiability space with respect to $\mathbb H$. Let $U\subset X$ such that $(U,\varphi)$ is a Lipschitz chart with target $\mathbb G$ for $\mathbb H$-valued maps, where $\mathbb G$ is a Carnot group. Hence $\varphi_\sharp(\mu\llcorner U)$ is absolutely continuous with respect to every Haar measure of $\mathbb{G}$.
\end{teorema}

The proof of Theorem~\ref{cheegerpansu} is mainly a consequence of the results of \cite{ReversePansu} with the additional ingredient given by the following lemma.

\begin{lemma} \label{lem:Arepr_push-forward}
Let $(X,\rho,\mu)$ be a Pansu differentiability space with respect to $\mathbb H$, let $\varphi:X\to \mathbb{G}$ be a Lipschitz map, and $U$ a Borel set such that $(\varphi,U)$ is a chart. Then, there are countably many disjoint Borel sets $U_i$ such that $\mu(U\setminus \cup_{i\in\N}U_i)=0$ and 
$\varphi_\# (\mu\llcorner U_i)$ has $n_1$-independent Alberti representations in $\mathbb{G}$.
\end{lemma}

\begin{proof}
\spacedlowsmallcaps{Step 1.} Let us prove that if there exists a representation of the form   $\mu\llcorner U = \int \mu_\gamma {d} \mathbb{P}(\gamma)$ where $\mu_\gamma\ll \mathcal{H}^1\llcorner \mathrm{im}\gamma$ and $\varphi\circ\gamma$ is a $C(e,\sigma)$-curve in the sense of \cref{C-curves} for $\mathbb{P}$-almost every $\gamma\in\Gamma(X)$, then we can build an Alberti representation for $\varphi_\#(\mu\llcorner U)$ as
\[
\varphi_\#(\mu\llcorner U)=\int \nu_{\bar \gamma}{d}\bar{ \mathbb{P}}(\bar \gamma)\qquad
\text{where $\bar{\mathbb{P}}$ is a probability measure on $\Gamma(\mathbb{G})$,}
\]
where $\nu_{\bar \gamma}\ll \mathcal{H}^1\llcorner \mathrm{im}\bar\gamma$ and $\bar\gamma$ is a $C(e,\sigma)$-curve for $\bar{\mathbb{P}}$-almost every $\bar\gamma\in\Gamma(\mathbb{G})$. 

In order to do this, let us define the map $\Phi \colon \Gamma(X)\to \Gamma (\mathbb{G})$ given by $\Phi(\gamma) := \varphi \circ \gamma$ and let 
$$\bar{\mathbb{P}}:=\Phi_\#\mathbb{P},$$
which is a probability measure on $\Gamma(\mathbb{G})$. 
Note that by definition, for $\mathbb{P}$-almost every \(\bar \gamma\),  it holds that \(\bar \gamma=\varphi\circ \gamma\) for some \(\gamma\in \Gamma (X)\).  

By considering \(\mathbb{P}\) as a probability measure defined on the complete metric space \(H(X)\) defined in Definition~\ref{lipcurvez}, and noting that $\mathbb{P}$ is concentrated on \(\Gamma(X)\), we can  apply the disintegration theorem for measures ~\cite[Theorem~5.3.1]{MR2129498} to show that for \(\bar{\mathbb P}\)-almost every \(\bar \gamma\), there exists a Borel probability measure   \(\eta_{\bar \gamma}\)  concentrated on 
\(\Phi^{-1}(\bar\gamma)\) and such that 
\[
\mathbb P(A) =\int  \eta_{\bar \gamma}(A) {d}\bar{\mathbb{P}}(\bar \gamma)\qquad\textrm{for all Borel sets \(A\subset\Gamma(X)\).}
\]
Note also that, by the disintegration theorem, the map \(\bar \gamma \mapsto \eta_{\bar \gamma}\) is Borel measurable. Let us now set 
\[
\nu_{\bar \gamma} := \int_{ \Phi^{-1}(\bar\gamma)} \varphi_\#(\mu_\gamma) {d} \eta_{\bar \gamma} (\gamma).
\]
Clearly, we have the representation
\[
\varphi_\#(\mu\llcorner U)=\int \nu_{\bar \gamma}{d}\bar{\mathbb{P}}(\bar \gamma)
\]
and \(\bar \gamma=\varphi\circ \gamma \) is by construction a $C(e,\sigma)$-curve. Hence, to conclude, we only have to show that 
\[
\nu_{\bar \gamma}\ll \mathcal{H}^1\llcorner \mathrm{im} \,\bar \gamma\qquad \textrm{for \(\bar{\mathbb{P}}\)-almost every  \(\bar\gamma\).}
\]
Let \(E\) be a set with \(\mathcal{H}^1(E\cap  \mathrm{im}\,\bar \gamma)=0\). Since \(\bar \gamma'(t)\ne 0\) for almost every \(t\in \mathrm{dom} \gamma\), the area formula implies that \( \mathcal{L}^1(\bar \gamma^{-1}(E))=0\). If \(\gamma\in \Phi^{-1}(\bar \gamma)\), then \(\bar \gamma=\varphi\circ \gamma\) and hence
\[
\mathcal{H}^1(\varphi^{-1}(E)\cap \mathrm{im}\,\gamma)\le \mathcal{H}^1(\gamma(\bar \gamma^{-1}(E)))=0\qquad \textrm{for all \(\gamma\in \Phi^{-1}(\bar \gamma)\).}
\]
Hence, \(\mu_\gamma (\varphi^{-1}(E))=0\) for all \(\gamma\in \Phi^{-1}(\bar \gamma)\), which immediately gives
\[
\nu_{\bar \gamma}(E)=\int_{\Phi^{-1}(\bar \gamma)} \mu_\gamma (\varphi^{-1}(E)) {d} \eta_{\bar \gamma} (\gamma)=0 \; .
\]
This concludes the proof of the first step. 

\spacedlowsmallcaps{Step 2.} We now show that there are countably many Borel sets $\{U_i\}_{i\in\N}$ such that $\mu(U\setminus \bigcup_{i\in\N}U_i)=0$ and $\mu\llcorner U_i$ has $n_1$ indipendent Alberti representations with respect to $\varphi$. This, however, can be obtained as follows. First, let us note that thanks to Remark~\ref{rem:Importante}, the metric space $(X,d,\mu)$ is a Lipschitz differentiability space and that $\pi_1\circ \varphi$ is an $n_1$-dimensional chart. This implies by \cite[Theorem~6.6]{BateJAMS} that there are countably many disjoint Borel sets $\{U_i\}_{i\in\N}$ such that $\mu(U\setminus \bigcup_{i\in\N}U_i)=0$ and $\mu\llcorner U_i$ has $n_1$ indipendent Alberti representations with respect to $\pi_1\circ\varphi$. However, by definition, this means that these representations are also $\varphi$-independent.
\end{proof}

We are ready to prove the generalisation of De Philippis--Marchese--Rindler's result on Cheeger's conjecture.

\begin{proof}[Proof of Theorem~\ref{cheegerpansu}]
    Thanks to Lemma~\ref{lem:Arepr_push-forward}, we know that there are countably many Borel sets $\{U_i\}_{i\in\N}$ such that $\mu(U\setminus \cup_{i\in\N}U_i)=0$ and $\varphi_\#(\mu\llcorner U_i)$ has $n_1$-independent representations with respect to $\varphi$. It is obvious that the measure $\varphi_\#(\mu\llcorner U)$ is absolutely continuous with respect to $\sum_{i\in\N}\varphi_\#(\mu\llcorner U_i)$ and thus it is sufficient to prove that $\varphi_\#(\mu\llcorner U_i)\ll \mathcal{H}^Q$. 

    The above argument shows that without loss of generality, we can assume that $\varphi_\#(\mu\llcorner U)$ has $n_1$-independent Alberti representations in $\mathbb G$. This easily implies by \cite[Propositions 2.4, 3.6]{ReversePansu}
and \cite[Remark 3.1]{ReversePansu} that the decomposability bundle $V(\varphi_\#(\mu\llcorner U),x)$ of $\varphi_\#(\mu\llcorner U)$ coincides with $\mathbb{G}$ for $\varphi_\#(\mu\llcorner U)$-almost every $x\in\mathbb{G}$. For a definition of the decomposability bundle for measures in Carnot groups, we refer to \cite[\S 3]{ReversePansu}. However, this implies that Lipschitz functions from $\mathbb{G}$ to $\R$ are Pansu differentiable $\varphi_\#(\mu\llcorner U)$-almost everywhere thanks to \cite[Theorem~6.5]{ReversePansu}, or in the language of \cite{ReversePansu}, the measure $\varphi_\#(\mu\llcorner U)$ has the \emph{Pansu property}. This, together with \cite[Theorem~7.5]{ReversePansu}, immediately concludes the proof.    
\end{proof}

\section{Equivalent criteria for $\mathbb G$-rectifiability under density assumptions}\label{sec6}

This section aims to prove the following characterization, which generalizes to arbitrary Carnot groups the one in \cite[Theorem~1.2]{bateli}. We stress that, in the generality of Carnot groups, \cref{thm:SeanLi} does not hold if, instead of item (i), we only require that $(X,d)$ is Lipschitz rectifiable, see \cref{rem:PerForzaBilip}.

\begin{definizione}[(bi)Lipschitz rectifiability]\label{def:GbiLipRect}
Let $(X,d)$ be a metric space. Let $\mathbb G$ be a Carnot group of homogeneous dimension $Q$. We say that $X$ is {\em (bi)Lipschitz $\mathbb G$-rectifiable} if, for every $i\in\mathbb N$, there exists a (bi)Lipschitz map $f_i:U_i\subset \mathbb G\to X$, where $U_i$ is Borel, such that
\[
\mathcal{H}^Q  \left(X\setminus \cup_{i=1}^{+\infty}f_i(U_i)\right)=0.
\]
\end{definizione}

\begin{definizione}[Densitites]\label{def:density}
    Let $(X,d,\mu)$ be a measure metric space. For every $\alpha\geq 0$ we will denote with
    $$\Theta^\alpha_*(\mu,x):=\liminf_{r\to 0}\frac{\mu(B(x,r))}{r^\alpha}\qquad\text{and}\qquad\Theta^{\alpha,*}(\mu,x):=\limsup_{r\to 0}\frac{\mu(B(x,r))}{r^\alpha},$$
    the {\em lower} and {\em upper $\alpha$-dimensional density} of $\mu$ at $x$, rispectively.
\end{definizione}

{
\begin{teorema}\label{thm:SeanLi}
Let $\mathbb G$ be a Carnot group of homogeneous dimension $Q$. Let $(X,d,\mu)$ be a metric measure space. The following facts are equivalent:
\begin{itemize}
    \item[(i)] $(X,d)$ is biLipschitz $\mathbb{G}$-rectifiable and $\mu$ is mutually absolutely continuous with respect to $\mathcal{H}^Q$;
    \item[(ii)]
    For every Carnot group $\mathbb H$, the following holds. There is a countable collection $\{U_i\}_{i\in\mathbb N}$ of Borel sets of $X$ with $\mu  (X\setminus\cup_{i\in\mathbb N}U_i)=0$,
    such that for every $i\in\mathbb N$,
    \begin{equation}
        0<\Theta^{Q}_*(\mu\llcorner U_i,x)\leq \Theta^{Q,*}(\mu\llcorner U_i,x)<\infty, \text{holds for $\mu$-almost every $x\in U_i$}\label{densitàipotez}
    \end{equation}
 and each $(U_i,d,\mu)$ $X$ is a $(\mathbb{G},\mathbb H)$-differentiability space;
    \item[(iii)] There exists a Carnot group $\mathbb H$ for which item (ii) holds;
    \item[(iv)] There exists a countable collection $\{U_i\}_{i\in\mathbb N}$ of Borel sets of $X$ with $
    \mu  (X\setminus\cup_{i\in\mathbb N}U_i)=0$ satisfying \eqref{densitàipotez}, and there are countably many Lipschitz maps $\varphi_i:U_i\to\mathbb G$ such that each $\mu\llcorner U_i$ has $n_1$ $\varphi_i$-independent Alberti representations, where $n_1$ is the rank of $\mathbb G$.
    \item[(v)] There exists a countable collection $\{U_i\}_{i\in\mathbb N}$ of Borel sets of $X$ with $
    \mu  (X\setminus\cup_{i\in\mathbb N}U_i)=0$ 
    satisfying \eqref{densitàipotez},  and   each $U_i$ satisfies David Condition with respect to a Lipschitz map $\varphi_i:U_i\to\mathbb G$. For the definition of David Condition, we refer to \cref{def:DavidsCondition}.
\end{itemize}
\end{teorema}
}

\subsection{Alberti representations imply big projections}\label{sec61}

Before announcing the main result of this subsection, let us introduce some notation.

\begin{definizione}[The sets $DC$]\label{def:DC}
Suppose $(X,d,\mu)$ is a metric measure space and let $\varphi:X\to \mathbb{G}$ be a Lipschitz map. For every $\beta,\varepsilon,R>0$ and every Borel set $U$ of $X$, we denote with $DC_{U,\varphi}(\beta,\varepsilon,R)$, or simply $DC(\beta,\varepsilon,R)$, the set of those $x\in U$ for which 
\begin{align*}
 \mathcal{H}^Q(B(\varphi(x),\beta r) \cap \varphi(B(x,r) \cap U)) \geq (1-\varepsilon) \mathcal{H}^Q(B(\varphi(x),\beta r))\qquad\text{for every $0<r<R$.}
\end{align*}

\end{definizione}

\begin{proposizione}
For every compact set $K\subseteq X$ and every Lipschitz map $\varphi:X\mapsto \mathbb{G}$ we have 
that 
\begin{itemize}
    \item[(i)] $DC_{K,\varphi}(\beta,\varepsilon,R) \subseteq DC_{K,\varphi}(\beta',\varepsilon,R)$ and $DC_{K,\varphi}(\beta,\varepsilon,R) \subseteq DC_{K,\varphi}(\beta,\varepsilon,R')$ whenever $\beta' \leq \beta$ and $R' \leq R$;
    \item[(ii)] $DC_{K,\varphi}(\beta,\varepsilon,R)$ is a Borel set for every $\beta,\varepsilon,R>0$.
\end{itemize}
\end{proposizione}

{
\begin{proof}
The item (i) follows directly from the definition. For item (ii), we argue as in the beginning of \cite[Proposition~3.1]{bateli}, and we see that $DC_{K,\varphi}(\beta,\varepsilon, R)$ is closed for every $\beta, \varepsilon, R>0$ and hence Borel.
\end{proof}
}

{
\begin{definizione}[David Condition]\label{def:DavidsCondition}
Let $U\subset X$ be a Borel subset of the metric measure space $(X,d,\mu)$, and let $\mathbb G$ be a Carnot group. We say that \textit{$U$ satisfies David Condition with respect to the Lipschitz map $\varphi:X\to\mathbb G$} if, for every $\varepsilon>0$, we have that 
\[
\mu\left(U\setminus \bigcup_{\beta,R\in\mathbb Q^+} DC_{U,\varphi}(\beta,\varepsilon,R)\right)=0.
\]
\end{definizione}
}

This section is devoted to the proof of the following 

{
\begin{teorema}\label{th:rep->davidcondpimr}
Let $\mathbb G$ be a Carnot group of homogeneous dimension $Q$. Let $(X,d,\mu)$ be a metric measure space such that
\[
    0<\Theta^{Q}_*(\mu,x)\leq \Theta^{Q,*}(\mu,x)<\infty,
    \]
for $\mu$-almost every $x\in X$.
Let $\varphi:X\to \mathbb{G}$ be a Lipschitz map 
and let us assume that $\mu$ has $n_1(\mathbb{G})$ $\varphi$-independent Alberti representation.
Then there exists a collection $\{U_i\}_{i\in\mathbb N}$ of Borel subsets of $X$ such that 
\[
\mu(X\setminus \cup_{i\in\mathbb N} U_i)=0,
\]
and each $U_i$ satisfies David Condition with respect to $\varphi:X\to\mathbb G$, see \cref{def:DavidsCondition}.
\end{teorema}
}

\begin{osservazione}\label{restrizionewlog}
    Let us note that in order to prove \cref{th:rep->davidcondpimr}, it suffices to show that $(X,d,\mu)$ contains a $\mu$-positive Borel subset $E$ that satisfies David Condition with respect to $\varphi$. This is due to the fact that the hypotheses of the theorem are stable by restrictions to a subset thanks, for instance, to Lebesgue differentiation theorem \cite[p. 77]{HeinonenKoskelaShanmugalingam}, and to a classical exhaustion procedure. 
\end{osservazione}

Throughout this section, if not otherwise stated, we let $(X,d,\mu)$ be a fixed metric measure space such that
\begin{itemize}
    \item[(i)] there exists $Q\in \N$
$$
0<\Theta^Q_*(\mu,x)\leq \Theta^{Q,*}(\mu,x)<\infty,
$$
for $\mu$-almost every $x\in U$, where $U\subset X$ is Borel;
\item[(ii)] there exist a Carnot group $\mathbb{G}$ of homogeneous dimension $Q$ and a Lipschitz map $\varphi:X\to \mathbb{G}$ such that $\mu\llcorner U$ has $n_1(\mathbb{G})$ $\varphi$-independent Alberti representations.  
\end{itemize}

Defined, for every $l\in\mathbb N$ and $R>0$,
\begin{equation}\label{eqn:ControlU}
  E(l,R) = \Big\{x \in U : l^{-1}r^Q \leq \mu(B(x,r)) \leq lr^Q, \text{ for all } r < R\Big\},
\end{equation}
it is easy to see that each $E(l,R)$ is a Borel set and that $ \mu( U \backslash \bigcup_{l=1}^\infty \bigcup_{k=1}^\infty E(l,k^{-1})) = 0$. A proof of this can be achieved with a straightforward adaptation of \cite[Propositions 1.14, 1.16]{MarstrandMattila20}. Moreover, we fix $l\in\N$ and $R>0$ in such a way that there is a compact set $K\subseteq E(l,R)$ such that $\mu(K)>0$. Note that by \cite[\S 2.10.17,\S 2.10.18]{Federer1996GeometricTheory}, we have $(2^Ql)^{-1}\mathcal{H}^Q\leq \mu\llcorner K\leq 2^Ql\mathcal{H}^Q$.
\medskip

Since the measure $\mu\llcorner U$ has $n_1$ $\varphi$-independent representations, it is immediate to see that $\mu\llcorner K$ has $n_1$ $\varphi$-independent representations as well. By definition this means that there is a basis $\mathcal{B}:=\{\bar e_1,\ldots,\bar e_{n_1}\}\in(\mathbb S^{n_1-1})^{n_1}$ of $V_1(\mathbb G)$, representations $\mathcal{A}_1,\ldots, \mathcal{A}_{n_1}$, a $\zeta>0$, and independent cones $C(\bar e_1,\sigma_1),\ldots,C(\bar e_{n_1},\sigma_{n_1})$ such that the cones $C(\bar e_i,\sigma_i)$ are $8\zeta$-separated and such that almost every curve of $\mathcal{A}_i$ is a $C(\bar e_i,\sigma_i)$-curve. 

Let $\eta_1>0$ be a small number for which the cones $C(\bar e_i,\sigma_i+\eta_1)$ are $4\zeta$-separated.
Thanks to \cite[Corollary 5.9]{BateJAMS} we can find countably many compact subsets $\{A_j\}_{j\in\N}$ of $K$ such that 
\begin{equation}
    \mu(K\setminus \bigcup_{j\in\N} A_j)=0,
    \label{A_j}
\end{equation}
and $\mu\llcorner A_j$ has an Alberti representation $\tilde{\mathcal{A}}_i^j$ in the $\varphi$-direction of $C(\bar e_i,\sigma_i+\eta_1)$, that thanks to the choice of $\eta_1$ are still $4\zeta$-separated cones. In addition the representations $\tilde{\mathcal{A}}_i^j$ have $\varphi$-horizontal speed $1/j$. 

\medskip

Thanks to Proposition~\ref{cubofigo} and to \cite[Corollary 5.9]{BateJAMS} for every $j\in\N$ we can find countably many compact $B_{\iota,j}$ of $A_j$ such that $\mu(A_j\setminus \cup_{\iota\in\N}B_{\iota,j})=0$ and cones $C(w_{\iota,j}^i,\alpha_{\iota,j}^i)$ such that 
\begin{itemize}
\item[(i)] the cones $C(w_{\iota,j}^i,\alpha_{\iota,j}^i)$ are $2\zeta$-separated,
    \item[(ii)] $\mu\llcorner B_{\iota,j}$ has $\mathcal{B}_1,\ldots,\mathcal{B}_{n_1}$ Alberti representations with  $\varphi$-horizontal speed $1/j$, and such that each $\mathcal{B}_i$ goes in the $\varphi$-direction of the independent cones $C(w_{\iota,j}^i,\alpha_{\iota,j}^i)$,
    \item[(\hypertarget{iiicubox}{iii})] for every $\iota,j\in \N$ there are constants $0<\lambda,\xi,\newC\label{C:scalasurj}\leq 1/10$, $\Lambda>0$, $M\in\N$\footnote{The constant $M$ can be chosen to be $2n$, where $n$ is the topological dimension of $\mathbb G$.}, $i_1,\ldots,i_M\in\{1,\ldots,n_1\}$, a compact cube $T$, and points $p_1,\ldots,p_{2^Q}\in T$ satisfying the first four items of \cref{cubofigo} and such that the following holds.
    For every basis $\{v_1,\ldots,v_{n_1}\}\in(\mathbb S^{n_1-1})^{n_1}$ of $V_1(\mathbb G)$ for which $v_i\in C(w_{\iota,j}^i,\alpha_{\iota,j}^i)$, every $y\in \mathbb{G}$ with $\lVert y\rVert\leq \oldC{C:scalasurj}$, and every $k=1,\dots,2^Q$ we can write $y^{-1}p_k$ as
$$
y^{-1}p_k=\delta_{s_1}(v_{i_1})\ldots\delta_{s_{M}}(v_{i_{M}}).
$$
and there holds $\min\{\lvert s_i\rvert :i=1,\ldots,M\}>\xi$ and $\max\{\lvert s_i\rvert :i=1,\ldots,M\}\leq \frac{n\Lambda}{ \mathrm{diam}(T)}\lVert y^{-1}p_k\rVert$. Recall that $n$ is the topological dimension of the group $\mathbb G$.
\end{itemize}

\textbf{We now assume without loss of generality, thanks to \cref{restrizionewlog}, and to what said above, that there are $\iota,j\in\N$ such that $\mu(K\setminus B_{\iota,j})=0$ and let us define} $\newep\label{ep:1}:= \min\{\lambda,\xi\}$ and
    \begin{equation}
 c_0:=\frac{n\Lambda}{ \mathrm{diam}(T)},\qquad \theta:=1/j, \qquad C_i:=C(w_{\iota,j}^i,\alpha_{\iota,j}^i).  
 \label{costantis}
    \end{equation}
In the following, under this assumption, we will drop every dependence on the indexes $j$ and $\iota$.

\begin{osservazione}[Constants $\oldC{C:1}$, $K_1$]\label{ConstantsCK}
Let $\oldC{C:1}$ be the constant defined in \cref{lemma:drift}, where we assume the bounds $\mathrm{diam}(C)\leq { 80(\mathrm{diam}(T)+1)c_0}$\footnote{Here the set $C$ is the compact set on which the curve $\gamma$ is defined in the statement of \cref{lemma:drift}.} and $\mathrm{Lip}(\gamma)\leq 2$ in there. Notice that the constant $\oldC{C:1}$ only depends on $\mathrm{diam}(T)$ and $c_0$. We recall that $c_0$ was introduced in \eqref{costantis}.

Let $K_1$ be a fixed constant such that
\begin{equation}\label{eqn:K1def}
K_1<\min\left\{\frac{\oldep{ep:1}}{32(\mathrm{diam}(T)+1)c_0[\oldC{C:1}+2\mathrm{Lip}(\varphi)]},1/100\right\}.
\end{equation}
Notice that, by how $\oldep{ep:1}$ and $\oldC{C:1}$ are defined, $K_1$ can be chosen to depend only on the cube $T$, and on $c_0, \mathrm{Lip}(\varphi)$. 
\end{osservazione}

\begin{definizione}[$\mathfrak{G}_\mathfrak{E}(v,R)$]\label{def:GvR}
Let $T$ be the cube introduced in (\hyperlink{iiicubox}{iii}) and let $\varphi:K\to\mathbb G$ be the Lipschitz function in \cref{th:rep->davidcondpimr}. Let us choose $0<v<\theta/100$, where we recall that $\theta$ was introduced in \eqref{costantis}. Let $\mathfrak{E}:=\{e_1,\ldots,e_{n_1}\}$ be such that $e_i\in C_i\cap \mathbb{S}^{n_1-1}$ for every $i=1,\ldots,n_1$, and let us define the set $\mathfrak{G}_\mathfrak{E}(v, R)$ to be the set of those $y \in K$ for which the following holds. 
For each $1 \leq j \leq n_1(\mathbb{G})$, there exists a $\gamma \in \Gamma(K)$ and a $t$ in the domain of $\gamma$ such that
\begin{enumerate}
\item[(i)] $\gamma(t)= y$;
\item[(ii)] for every $0< r < { 40(\mathrm{diam}(T)+1)c_0R}$ we have $$\mathcal{L}^1(B(t,r) \cap \gamma^{-1}(E)) > (1-(K_1v)^{s^{2M}})\mathcal{L}^1(B(t,r));$$
\item[(iii)] for every $s>s' \in \mathrm{Dom}\gamma$, $\pi_{\mathbb G}(\varphi(\gamma(s)))-\pi_{\mathbb G}(\varphi(\gamma(s'))) \in C(e_j, (K_1v)^{s^{2M}})$ and
\begin{align}
  d(\varphi(\gamma(s)),\varphi(\gamma(s'))) \geq \|\pi_{\mathbb G}(\varphi(\gamma(s)))-\pi_{\mathbb G}(\varphi(\gamma(s')))\| > v d(\gamma(s), \gamma(s')). \label{e:phi-speed}
\end{align}
\end{enumerate}
\end{definizione}

Let us fix $v\in (0,\theta/100)\cap \mathbb Q$, where $\theta$ was introduced in \eqref{costantis}. Thanks to \cite[Corollary 5.9]{BateJAMS}, we know that we can split  $K$ into countably many compact sets $ C_\kappa$ such that 
$\mu\llcorner  C_\kappa$ has $n_1(\mathbb{G})$ $\varphi$-independent Alberti representation going in the direction of $\zeta$-separated cones $C(e_i^\kappa,(K_1 v)^{s^{2M}})$ with $e_i^\kappa\in C_i\cap \mathbb{S}^{n_1-1}$ and $\varphi$-speed $\theta$ respectively.
Note that the choice of the unitary vectors $e_i^\kappa$s and of the compact sets $C_\kappa$s depends on $v$.

\begin{definizione}\label{def:Dstorto}
For $v,\varepsilon,R>0$ and   basis $\mathfrak{E}:=\{e_1,\ldots,e_{n_1}\}$   such that $e_i\in C_i\cap \mathbb{S}^{n_1-1}$, for every $i=1,\ldots,n_1$, we define
\begin{align*}
  \mathfrak{D}_{\mathfrak{E}}(v,\varepsilon,R) := \{x \in K : \mu(\mathfrak{G}_\mathfrak{E}(v,R) \cap B(x,r) )\geq (1-\varepsilon) \mu(B(x,r)), \text{ for every } r < R \}.
\end{align*}
\end{definizione}

{
\begin{proposizione}\label{prop:Sopra}
Let $\mu$ and $K$ as above. Let further $\mathscr{D}$ be a dense countable set in $\mathbb{S}^{n_1-1}$ and denote 
$$\mathscr{E}:=\{(e_1,\ldots,e_{n_1})\in \mathscr{D}^{n_1}:e_i\in C_i\}.$$
Then, we have 
 \begin{equation}
      \mu\left(K\setminus \bigcup_{v,R\in \Q^+}\bigcup_{\mathfrak{E}\in \mathscr{E}}\mathfrak{G}_\mathfrak{E}(v,R)\right)=0.
      \label{eq:identitàcurveprecise}
 \end{equation}
\end{proposizione}

\begin{proof}
The first step is to show that $\mathfrak{G}_\mathfrak{E}(v,R)$ is $\mu$-measurable for every $v,R>0$ and every basis $\mathfrak{E}\in \mathscr{E}$. 
This is true since $K$ is complete and separable and $\mu$ is Borel regular, and thus thanks to \cref{prop:curvefullmeas} we know that $\mathfrak{G}_\mathfrak{E}(v,R)$ is $\mu$-measurable for every $v, R>0$.
The fact that 
\eqref{eq:identitàcurveprecise} holds is an immediate consequence of our reductions.
\end{proof}
}

\textbf{From now on, let us fix $v,R$ and $\mathfrak{E}$ as in \cref{def:GvR}. The aim is to prove \cref{l:david-inclusion} and then exploit \cref{prop:Sopra} to conclude the proof of \cref{th:rep->davidcondpimr}}. 

Since for every $v$ as above we have $\mu(K\setminus \bigcup_{R\in \mathbb Q^+,\mathfrak{E}\in \mathscr{E}} \mathfrak{G}_\mathfrak{E}(v,R))=0$, and since we are fixing $v,R$ and $\mathfrak{E}$, in the following we assume that $\mu(\mathfrak{G}_\mathfrak{E}(v,R))>0$ and we will drop from the notation the dependence on $\mathfrak{E}=\{e_1,\ldots,e_{n_1}\}$.

In the following, in order to prove \cref{th:rep->davidcondpimr}, we will need a technical lemma. Since there will be a lot of constants involved in the proof, we will list all of them here for ease of reference. 
Let us recall that $c_0,M,\lambda,\oldep{ep:1}$ and the cube $T$ are all defined in the discussion above in between \cref{restrizionewlog} and \cref{ConstantsCK}. 
The Lipschitz function $\varphi$ is the one in the statement of \cref{th:rep->davidcondpimr}. The constant $K_1$ is defined in \eqref{eqn:K1def}, also in terms of $\oldC{C:1}$ that can be chosen to depend only on $\mathrm{diam}(T)$ and $c_0$.
{ 
 Let $C$ be a constant such that for all $x,y\in B(0,\max\{4c_0(\mathrm{diam}(T)+1),2\})$ we have 
\begin{equation}\label{eqn:ConiugoDai}
\|y^{-1}\cdot x \cdot y\|\leq C\|x\|^{1/s}.
\end{equation}
This constant exists and depends only on $\mathbb G,c_0$, and the cube $T$, due to \cite[Lemma~2.13]{FranchiSerapioni16}.
Moreover, let us set

\[
\newC\label{C:7}=4c_0(\mathrm{diam}(T)+1),
\]
\[
N=2\cdot \frac{8Mc_0(\mathrm{diam}(T)+1)+1}{\lambda}
\]
\[
\newC\label{C:6}<\min\left\{\frac{\oldep{ep:1}}{4\mathrm{Lip}(\varphi)},\frac{1}{100+N}, \oldC{C:7}/100\right\},
\]
\[
\newC\label{C:5}=4\cdot 2M(2+64(\mathrm{diam}(T)+1)c_0),
\]
\[
\newC\label{C:11}= C\cdot \oldC{C:6}^{1/s}\mathrm{Lip}(\varphi)^{1/s}+4K_1[2\oldC{C:1}+4\mathrm{Lip}(\varphi)](\mathrm{diam}(T)+1)c_0,
\]
\[
\newC\label{C:4}= { \oldC{C:11}^{\frac{1}{s^{M}}} (2\max\{C,1\})^{M}},
\]
\[
\newC\label{C:2}=10\oldC{C:4}.
\]
Notice that by taking $K_1$ and $\oldC{C:6}$ verifying the inequalities above, and sufficiently small with respect to $\mathrm{Lip}(\varphi),\mathrm{diam}(T),c_0, \oldC{C:scalasurj}$ we can further ensure $\oldC{C:11}<1$ and
\[
\oldC{C:2}<\min\{1/100,\oldC{C:7}/10, \oldC{C:scalasurj}/10\},
\]
and thus, in particular 
\begin{equation}\label{eqn:StimaNDai}
N>\frac{8Mc_0(\mathrm{diam}(T)+1)+\oldC{C:2}}{\lambda}.
\end{equation}
We are now ready for the lemma.

\begin{lemma} \label{l:induction-alternative}
Let $x \in K$, and let $\mathfrak{G}(v,R)$ be as above. Let $\varphi:X\to\mathbb G$ be the Lipschitz function in \cref{th:rep->davidcondpimr}.
 Let $K_1, \oldC{C:1}$ be the constants defined in \cref{ConstantsCK}. Let $c_0,M,\lambda,\oldep{ep:1},C, \oldC{C:7}, N, \oldC{C:6}, \oldC{C:5}, \oldC{C:11}, \oldC{C:4}, \oldC{C:2}$ be the constants defined above and let $T$ be the cube introduced in the above reductions, see the discussion from \cref{restrizionewlog} to \cref{ConstantsCK}. Let $r<R$ and let us define 
$$
\mathfrak Q:=c_{\mathfrak Q}\cdot\delta_{vr}(T),
$$
where $c_{\mathfrak Q}$ satisfies $d(\varphi(x),c_{\mathfrak Q}) <{ \oldC{C:2}vr}$.  Let $\{p_i\}_{i=1,\ldots, 2^Q}$ denote the centres of the $2^Q$ subcubes of $T$ introduced in (\hyperlink{iiicubox}{iii}).  Then at least one of the following must be true.
  \begin{enumerate}
    \item[(\hypertarget{alternativei}{i})] There exist points $\{q_i\}_{i=1,\ldots,2^Q}$ in $B(x,\oldC{C:5}Mr) \cap K$ so that 
    $$
    d(\varphi(q_i) ,c_{\mathfrak{Q}}\cdot\delta_{vr}( p_i)) <{ \oldC{C:4}}vr;
    $$
    \item[(ii)] there exists some $y \in B(x,\oldC{C:5}Mr) \cap K$ so that $B(y, \oldC{C:6}{v^{s^{2M}}}r) \cap \mathfrak{G}(v, R) = \emptyset$ and $$
    \dist(\varphi(y),(\delta_{N}\mathfrak{Q})^c) \geq \oldC{C:7}vr,
    $$
    where $\delta_{N}\mathfrak{Q}:=c_\mathfrak{Q}\cdot \delta_{N vr}(T)$,
where $\lambda>0$ is the constant introduced in Proposition~\ref{prop:tilings}(iv).
  \end{enumerate}
\end{lemma}

\begin{proof}
Assume without loss of generality that $c_\mathfrak{Q}=0$.
Suppose that (ii) is false, i.e., for every $y\in B(x,\oldC{C:5}Mr)\cap K$ we have 
\begin{equation}
    \text{$B(y,\oldC{C:6}v^{s^{2M}}r)\cap \mathfrak{G}(v,R)\neq \emptyset\qquad$or$\qquad\mathrm{dist}(\varphi(y),(\delta_{N}\mathfrak{Q})^c)< \oldC{C:7}vr$.}
    \label{condizioneiinegata1}
\end{equation}
 However, by assumption we know that  $d(\varphi(x),0)< \oldC{C:2}vr$. Since $B(0,\lambda)\subseteq T$, we have that $B(0,N\lambda vr)\subseteq \delta_N\mathfrak{Q}$ and thus 
 \begin{equation}
 \begin{split}
      \mathrm{dist}(\varphi(x),(\delta_{N}\mathfrak{Q})^c)&\geq \mathrm{dist}(0,(\delta_{N}\mathfrak{Q})^c)-d(\varphi(x),0)\\
     &\geq N\lambda vr - \oldC{C:2}vr >2M\oldC{C:7}vr.
     \label{stimabordox}
 \end{split}
 \end{equation}
where the last inequality follows from \eqref{eqn:StimaNDai} and the definition of \oldC{C:7}.
 This therefore shows that there exists, by \eqref{condizioneiinegata1}, an 
$$
\bar{x}\in B(x,\oldC{C:6}v^{s^{2M}}r)\cap \mathfrak{G}(v,R).
$$

Let us prove that item (\hyperlink{alternativei}{i}) holds. In the following, we show that the conclusion holds just in the case of $p=p_1$. However, the argument for all the other points follows verbatim.

\textbf{Claim I.} As a first step, let us prove the following. Let us fix an arbitrary $r/2<r'\leq Mr$. We aim to prove that for every $\kappa>M$, we have the following. For every $y\in B(x,\oldC{C:5}r')$ such that
\begin{equation}\label{eqn:InductiveDistance}
\mathrm{dist}(\varphi(y),(\delta_{N}\mathfrak{Q})^c)>\kappa \oldC{C:7}vr,
\end{equation}
then, for every $e\in \{e_1,\ldots,e_{n_1(\mathbb{G})}\}$, and for every  $\oldep{ep:1}\leq \lvert \eta\rvert\leq c_0(\mathrm{diam}(T)+1)$, there exists a $\bar y \in \mathfrak{G}(v,R)\cap B(y,\oldC{C:6}v^{s^{2M}}r)$, a curve $\gamma\in \Gamma(X)$ starting from $\bar y$, and a time $\bar t$ such that 
\begin{equation}\label{eqn:STIMaDAI}
    \gamma(\bar t)\in B(y,\oldC{C:5}r'/M),
\end{equation}
and 
\begin{equation}\label{eqn:STIMAINTER}
d(\varphi(\gamma(\bar t)),\varphi(y)\delta_{\eta vr}(e)))\le \oldC{C:11}v^{s^{2M-1}}r,
\end{equation}
and moreover 
\begin{equation}\label{eqn:DistanzaBordoVerificata}
    \mathrm{dist}(\varphi(\gamma(\bar t)),(\delta_{N}\mathfrak{Q})^c)>(\kappa-1)\oldC{C:7}vr.
\end{equation}

Indeed, thanks to \eqref{eqn:InductiveDistance}, the fact that $\kappa>M>1$, and \eqref{condizioneiinegata1}, we know that there exists a $\bar{y}\in \mathfrak{G}(v,R)$ such that 
\begin{equation}\label{eqn:yday}
d(y,\bar y)\leq \oldC{C:6}v^{s^{2M}}r.
\end{equation}
This implies, by definition of $\mathfrak{G}(v,R)$, that there exists a $\gamma\in \Gamma(X)$ and a $t_0\in \mathrm{Dom}(\gamma)$ such that
\begin{enumerate}
\item[($\alpha$)] $\gamma(t_0)=\bar y$;
\item[($\beta$)] for every $0< \rho <40(\mathrm{diam}(T)+1)c_0R$ we have 
$$
\mathcal{L}^1(I(t_0,\rho) \cap \gamma^{-1}(K)) > (1-(K_1v)^{s^{2M}})\mathcal{L}^1(I(t_0,\rho)),
$$
where $I(t_0,\rho)$ is the neighbourhood of centre $t_0$ and radius $\rho$ in $\mathbb R$;
\item[($\gamma$)] for every $t>t' \in \mathrm{Dom}\gamma$, $\pi_1(\varphi(\gamma(t)))-\pi_1(\varphi(\gamma(t'))) \in C(e, (K_1v)^{s^{2M}})$ and
\begin{equation}\label{eqn:AboveBound}
  d(\varphi(\gamma(t)),\varphi(\gamma(t'))) > v d(\gamma(t), \gamma(t')). \nonumber
\end{equation}
\end{enumerate}
Thanks to ($\beta$) for every $0<\rho<{40(\mathrm{diam}(T)+1)c_0R}$ we can find $t_{+}:=t_{+}(\rho)$ and $t_{-}:=t_{-}(\rho)$ in $I(t_0,\rho) \cap \gamma^{-1}(K)$ such that 
\begin{equation}\label{eqn:Rho1}
\lvert  \rho\mp( t_{\pm}(\rho)-t_0)\rvert<2(K_1v)^{s^{2M}}\rho< 2K_1\rho,
\end{equation}
where the last inequality is true since $K_1v<v<1$. Let us recall, for the ease of notation, $C:=\gamma^{-1}(K)$. Thus by Lemma~\ref{lemma:drift}, where $\sigma=(K_1v)^{s^{2M}}$ here, that
\begin{equation}
 \begin{split}
      d\Big(\varphi\circ\gamma(t_{\pm}(\rho)),\,&\varphi(\bar{y})\cdot\Big[\left(\int_{t_0}^{t_{\pm}(\rho)}\lvert D(\varphi\circ\gamma)(s)\rvert d\mathcal{L}^1\llcorner C(s)\right)\,e\Big]\Big)\leq \oldC{C:1}\sigma^{1/s}\lvert t_{\pm}(\rho)-t_0\rvert\\
      &\leq\oldC{C:1}\sigma^{1/s} \Big(1+2K_1\Big)\rho\leq \oldC{C:1}\Big(1+2K_1\Big)(K_1v)^{s^{2M-1}}\rho
 \end{split}
   \label{stimadriftydrifty1}
\end{equation}
In addition, since
$$2\mathrm{Lip}(\varphi)\geq \lvert D(\varphi\circ \gamma)(s)\rvert\geq v/2\qquad\text{for $\mathcal{L}^1$-almost every $t\in\mathrm{Dom}(\gamma)$},$$
thanks to the fact that by construction the curves $\gamma$ are taken $2$-biLipschitz and $(\gamma)$, we have that, thanks to the above bound in \eqref{eqn:AboveBound} and the fact that $K_1<1/100$,
\begin{equation}
    \begin{split}
    v\rho/4 \leq v\mathcal{L}^1([t_0,t_0\pm \rho]\cap C)/2\leq \pm\int_{t_0}^{t_{0}\pm\rho}\lvert D(\varphi\circ\gamma)(s)\rvert d\mathcal{L}^1\llcorner C(s)\leq 4\mathrm{Lip}(\varphi)\rho,
     \label{eq:stimaderivate1}
    \end{split}
\end{equation}
Let us now prove that there exists a $\frac{vr\oldep{ep:1}}{4\mathrm{Lip}(\varphi)}\leq \bar\rho\leq 4(\mathrm{diam}(T)+1)c_0r$ such that
\begin{equation}
    \Big\lvert\pm\int_{t_0}^{t_{\pm}(\bar\rho)}\lvert D(\varphi\circ\gamma)(s)\rvert d\mathcal{L}^1\llcorner C(s)\mp\lvert \eta\rvert  vr\Big\rvert\leq 4\mathrm{Lip}(\varphi)(K_1v)^{s^{2M}}\bar\rho.
    \label{numberointegralestima1}
\end{equation}

For every $0<{\rho<40(\mathrm{diam}(T)+1)c_0R}$, the above \eqref{eqn:Rho1}, together with the estimates on $\lvert D(\varphi\circ \gamma)\rvert$ and the choice of $t_{\pm}(\rho)$ implies
\begin{equation}
 \Big\lvert \int_{t_0\pm\rho}^{t_{\pm}(\rho)}\lvert D(\varphi\circ\gamma_j)(s)\rvert d\mathcal{L}^1\llcorner C(s)\Big\rvert\leq 2(K_1v)^{s^{2M}}\rho\cdot 2\mathrm{Lip}(\varphi) = 4\mathrm{Lip}(\varphi)(K_1v)^{s^{2M}}\rho.
    \label{stimaerrorepip1}
\end{equation}
Since the function $\rho\mapsto\int_{t_0}^{t_{0}\pm\rho}\lvert D(\varphi\circ\gamma_j)(s)\rvert d\mathcal{L}^1\llcorner C(s)$ is continuous, thanks to \eqref{eq:stimaderivate1} and to \eqref{stimaerrorepip1}, and to the fact that $\oldep{ep:1}\leq |\eta|\leq (\mathrm{diam}(T)+1)c_0$ it is immediate to see, by the triangle inequality and the intermediate value theorem, that there exists a $\bar\rho\in (\frac{vr\oldep{ep:1}}{4\mathrm{Lip}(\varphi)},4(\mathrm{diam}(T)+1)c_0r)$ such that \eqref{numberointegralestima1} holds. 
Let us further note that, exploiting \eqref{eqn:yday},
\begin{equation}
    \begin{split}
        d(\varphi(y),\varphi(\gamma(t_{\pm}(\bar\rho))))&\leq d(\varphi(\bar y),\varphi(y))+d(\varphi(\bar y), \varphi(\gamma(t_{\pm}(\bar\rho)))) \\
        &\leq \mathrm{Lip}(\varphi)\oldC{C:6}v^{s^{2M}}r+d(\varphi(\gamma(t_{0})), \varphi(\gamma(t_{\pm}(\bar{\rho})))\\
        &\leq \mathrm{Lip}(\varphi)\oldC{C:6}v^{s^{2M}}r\\
        &\qquad+d(\varphi(\gamma(t_{\pm}(\bar\rho))),\varphi(\gamma(t_{0}))\cdot\Big[\left(\int_{t_0}^{t_{\pm}(\bar\rho)}\lvert D(\varphi\circ\gamma)(s)\rvert d\mathcal{L}^1\llcorner C(s)\right)\,e\Big]\Big))\\
        &\qquad+d(\varphi(\gamma(t_{0})),\varphi(\gamma(t_{0}))\cdot\Big[\left(\int_{t_0}^{t_{\pm}(\bar\rho)}\lvert D(\varphi\circ\gamma)(s)\rvert d\mathcal{L}^1\llcorner C(s)\right)\,e\Big]\Big))\\
       \overset{\eqref{stimadriftydrifty1},\eqref{numberointegralestima1}}{ \leq} &\mathrm{Lip}(\varphi)\oldC{C:6}v^{s^{2M}}r+\oldC{C:1}\Big(1+2K_1\Big)(K_1v)^{s^{2M-1}}\bar\rho\\
       &\qquad\qquad\qquad+\lvert \eta\rvert  vr+4\mathrm{Lip}(\varphi)(K_1v)^{s^{2M}}\bar \rho\\
       &\leq \mathrm{Lip}(\varphi)\oldC{C:6}v^{s^{2M}}r+\oldC{C:1}\Big(1+2K_1\Big)(K_1v)^{s^{2M-1}}\cdot 4(\mathrm{diam}(T)+1)c_0r\\
       &\qquad\qquad\qquad+\lvert \eta\rvert  vr+4\mathrm{Lip}(\varphi)(K_1v)^{s^{2M}}\cdot4(\mathrm{diam}(T)+1)c_0r\\
       &\leq 2\rvert\eta\rvert vr,
        \label{eq:stima drift}
    \end{split}
\end{equation}
where the last inequality comes from the fact that $K_1<1/2$, $\oldep{ep:1}\leq |\eta|$, and the inequality satisfied by $\oldC{C:6}, K_1$, in particular the fact that 
\[
\oldC{C:6}<\frac{\oldep{ep:1}}{4\mathrm{Lip}(\varphi)}, \qquad K_1 < \frac{\oldep{ep:1}}{32(\mathrm{diam}(T)+1)c_0[\oldC{C:1}+2\mathrm{Lip}(\varphi)]}.
\]
This implies in particular that, taking into account \eqref{eq:stima drift}, and \eqref{eqn:InductiveDistance},
\begin{equation}\label{eqn:DistanzaDalBordo}
\begin{split}
     &\mathrm{dist}(\delta_{N}\mathfrak{Q}^c,\varphi\circ\gamma(t_{\pm}(\bar\rho))))
     \geq \mathrm{dist}(\delta_{N}\mathfrak{Q}^c,\varphi(y))-d(\varphi( y),\varphi\circ\gamma(t_{\pm}(\bar\rho)))
     \\&\geq \kappa\oldC{C:7}vr -2|\eta|vr> (\kappa-1)\oldC{C:7}vr,
      \end{split}
\end{equation}
{where the last inequality comes from the fact that $2|\eta|\leq 2c_0(\mathrm{diam}(T)+1)\leq\oldC{C:7} $}.
Moreover we have that, thanks to \eqref{eqn:Rho1} and \eqref{eqn:yday},
\begin{equation}\label{eqn:StimaydaCurva}
    \begin{split}
        d(y,\gamma(t_{\pm}(\bar\rho)))&\leq d(y,\bar y)+d(\bar y,\gamma(t_{\pm}(\bar\rho)))\leq \oldC{C:6}v^{s^{2M}}r+2\lvert t_{\pm}(\bar\rho)-t_0\rvert\\
        &\leq \oldC{C:6}v^{s^{2M}}r+2(2(K_1v)^{s^{2M}}+1)\bar\rho\\
        &\leq \oldC{C:6}v^{s^{2M}}r+2(2(K_1v)^{s^{2M}}+1)\cdot 4(\mathrm{diam}(T)+1)c_0r<\frac{\oldC{C:5}r'}{2M}.
    \end{split}
\end{equation}
{where the last inequality comes from the choice of the constant $\oldC{C:5}$, after having noticed that $\oldC{C:6}<1$,$K_1v<1$}.
Let us assume without loss of generality that $\eta>0$. Let us denote $\bar t:=t_+(\bar\rho)$ as defined above. Hence \eqref{eqn:StimaydaCurva} gives precisely \eqref{eqn:STIMaDAI}, while \eqref{eqn:DistanzaDalBordo} gives \eqref{eqn:DistanzaBordoVerificata}. Let us now verify \eqref{eqn:STIMAINTER}. 
Indeed, we have, by also using \eqref{eqn:ConiugoDai},

\begin{equation}
    \begin{split}
        d(\varphi(\gamma(\bar t)),\varphi(y)\cdot\delta_{\eta vr}(e)))&\leq d(\varphi(\bar y)\cdot\delta_{\eta vr}(e),\varphi(y)\cdot\delta_{\eta vr}(e)) \\
&+ d\left(\varphi(\gamma(\bar t)),\varphi(\bar{y})\cdot\Big[\left(\int_{t_0}^{t_{\pm}(\bar\rho)}\lvert D(\varphi\circ\gamma)(s)\rvert d\mathcal{L}^1\llcorner C(s)\right)\,e\Big]\right)\\
&+d\left(\varphi(\bar{y})\cdot\Big[\left(\int_{t_0}^{t_{\pm}(\bar\rho)}\lvert D(\varphi\circ\gamma)(s)\rvert d\mathcal{L}^1\llcorner C(s)\right)\,e\Big],\varphi(\bar y)\cdot\delta_{\eta vr}(e)\right)\\
\overset{\substack{\eqref{eqn:yday},\eqref{stimadriftydrifty1},\\\eqref{numberointegralestima1}}}{\leq} C(r^{-1}d(\varphi(\bar y),\varphi(y)))^{1/s}r+&\oldC{C:1}\Big(1+2K_1\Big)(K_1v)^{s^{2M-1}}\bar\rho+4\mathrm{Lip}(\varphi)(K_1v)^{s^{2M}}\bar\rho\\
&\leq C(\oldC{C:6}\mathrm{Lip}(\varphi)v^{s^{2M}})^{1/s}r\\&+\left(\oldC{C:1}\Big(1+2K_1\Big)(K_1v)^{s^{2M-1}}+4\mathrm{Lip}(\varphi)(K_1v)^{s^{2M}}\right)4(\mathrm{diam}(T)+1)c_0r\\
&\leq \oldC{C:11}v^{s^{2M-1}}r.
    \end{split}
\end{equation}
{ where the last inequality is true by how we chose $\oldC{C:11}$ an by simply noticing that $K_1<1/2$ and $v<1$}. Thus the Claim I is proved.

Let us now continue with the proof. Thanks to the discussion from \cref{restrizionewlog} to \cref{ConstantsCK}, and thanks to the fact that $\oldC{C:2}<\min\{\oldC{C:scalasurj}/10,1/100\}$, we can find $\lambda_1,\ldots,\lambda_M$ such that
    \begin{equation}
        \delta_{(rv)^{-1}}(\varphi(x))^{-1}   p=\delta_{\lambda_1}(v_1)\cdot\ldots\cdot\delta_{\lambda_M}(v_M),
           \label{eq:espanzionedirezioni}
    \end{equation}
    with $v_1,\ldots,v_M\in\{e_1,\ldots,e_{n_1(\mathbb{G})}\}$ and 
    $$
    \oldep{ep:1}\leq | \lambda_{\ell}|\leq c_0\lVert \delta_{(rv)^{-1}}(\varphi(x))^{-1}   p\rVert\leq c_0(\mathrm{diam}(T)+\oldC{C:2})\leq c_0(\mathrm{diam}(T)+1).
    $$
    
\textbf{Claim II.} Let us prove inductively that for every $1\leq \kappa\leq M$ there is $\bar x_\kappa\in B(x,\oldC{C:5}r/2\cdot k)$ such that
\begin{equation}\label{eqn:Inductive}
    \begin{split}
         d(\varphi(\bar x_{\kappa}),\varphi(x)\delta_{ vr}(\delta_{\lambda_1}v_1\cdot\ldots\cdot\delta_{\lambda_\kappa}v_\kappa))&\leq { \oldC{C:11}^{\frac{1}{s^{k-1}}} (2\max\{C,1\})^{k-1}}  v^{s^{2M-\kappa}}r\\
         \mathrm{dist}(\varphi(\bar x_k),(\delta_{N}\mathfrak{Q})^c)&>(2M-\kappa)\oldC{C:7}vr.
    \end{split}
\end{equation}
Let us prove the claim by induction. First, let us take care of the basis step $\kappa=1$. Thanks to \eqref{stimabordox} and Claim I applied with $y=x$, $\eta=\lambda_1$, and $r'=r/2$, we know that there is an $\bar x_1\in B(x,\oldC{C:5}r/(2M))$, which is the $\gamma(\bar t)$ in the Claim I, such that 
\begin{align}
    d(\varphi(\bar x_1),\varphi(x)\delta_{\lambda_1 vr}(v_1)))&\le \oldC{C:11}v^{s^{2M-1}}r,\label{eqn:STIMAINTER2}\\
    \mathrm{dist}(\varphi(\bar x_1),(\delta_{N}\mathfrak{Q})^c)&>(2M-1)\oldC{C:7}vr.\label{eqn:DistanzaBordoVerificata2}
\end{align}
In addition since $\bar x_1\in B(x, \oldC{C:5}r/2M)$, this 
ends the proof of the basis step. Let us now deal with the inductive step. 

Assume we proved the claim in \eqref{eqn:Inductive} for a $1\leq\kappa<M$. Let us apply Claim I to $\bar x_k$, with $r'=r/2\cdot k$, where we choose $\eta=\lambda_{k+1}$.
Notice that the Claim I can be applied due to the second of \eqref{eqn:Inductive}. Thus we get an $\bar x_{k+1}$, which is the $\gamma(\bar t)$ in the Claim I, such that the second of \eqref{eqn:Inductive} holds with $\kappa+1$ instead of $\kappa$, and moreover
\begin{equation}\label{eqn:AncoraInduco}
    \begin{split}
        d(\bar x_{k+1},\bar x_k)&\leq {\oldC{C:5}r/(2M) \cdot k}\\
        d(\varphi(\bar x_{\kappa+1}),\varphi(\bar x_\kappa)\delta_{vr}(\delta_{\lambda_{\kappa+1}}(v_{\kappa+1})))&\leq \oldC{C:11}v^{s^{2M-1}}r.
    \end{split}
\end{equation}
From the first of \eqref{eqn:AncoraInduco}, the triangle inequality, and the induction hypothesis, we get that 
\begin{equation}\label{eqn:Indu1}
{d(\bar x_{k+1},x)\leq d(\bar x_{k+1},\bar x_{k})+d(\bar x_{k},x)}\leq \oldC{C:5}r/2\cdot (k+1/M)\leq \oldC{C:5}r/2\cdot (k+1).
\end{equation}
Finally, from the second of \eqref{eqn:AncoraInduco}, the triangle inequality, the inductive hypothesis, and again \eqref{eqn:ConiugoDai}, we get
\begin{equation}\label{eqn:Indu2}
    \begin{split}
        d(\varphi(\bar x_{\kappa+1}),&\varphi(x)\delta_{ vr}(\delta_{\lambda_1}v_1\cdot\ldots\cdot\delta_{\lambda_{\kappa+1}}v_{\kappa+1})))\leq d(\varphi(\bar x_{\kappa+1}),\varphi(\bar x_\kappa)\delta_{vr}(\delta_{\lambda_{\kappa+1}}(v_{\kappa+1}))) \\
        &+ d(\varphi(\bar x_\kappa)\delta_{vr}(\delta_{\lambda_{\kappa+1}}(v_{\kappa+1})),\varphi(x)\delta_{ vr}(\delta_{\lambda_1}v_1\cdot\ldots\cdot\delta_{\lambda_{\kappa+1}}v_{\kappa+1})) \\
        &\leq \oldC{C:11}v^{s^{2M-1}}r+C\left(r^{-1}d(\varphi(\bar x_{\kappa}),\varphi(x)\delta_{ vr}(\delta_{\lambda_1}v_1\cdot\ldots\cdot\delta_{\lambda_\kappa}v_\kappa))\right)^{1/s}\cdot r\\
        &\leq \oldC{C:11}v^{s^{2M-1}}r+ C\left({\oldC{C:11}^{\frac{1}{s^{k-1}}} (2\max\{C,1\})^{k-1}} v^{s^{2M-\kappa}}\right)^{1/s}\cdot r \leq { \oldC{C:11}^{\frac{1}{s^{k}}} (2\max\{C,1\})^{k}}  v^{s^{2M-\kappa-1}}r.\\
    \end{split}
\end{equation}
{
Thus the induction is concluded since we have \eqref{eqn:Indu1}, and \eqref{eqn:Indu2}.

\textbf{Conclusion}. Let us now take $\kappa=M$ in Claim II. We get $\bar x_M=q_1\in B(x,\oldC{C:5}Mr/2)\subseteq B(x,\oldC{C:5}Mr)$ such that 
\[
d(\varphi(q_1),\varphi(x)\delta_{vr}(\delta_{(vr)^{-1}}(\varphi(x)^{-1})p_1))=d(\varphi(q_1),\delta_{vr}(p_1))\leq { \oldC{C:11}^{\frac{1}{s^{M}}} (2\max\{C,1\})^{M}} vr,
\]
which is precisely the conclusion of item (i).

}
\end{proof}

{
\begin{lemma} \label{l:david-inclusion}
  Let us assume we are in the hypotheses of \cref{th:rep->davidcondpimr}, and let $\varphi:X\to\mathbb G$ be the Lipschitz function in the assumptions.
  
  Let $x \in K$, and let $\mathfrak{G}(\cdot,\cdot), \mathfrak{D}(\cdot,\cdot,\cdot), DC(\cdot,\cdot,\cdot)$ be defined as in \cref{def:GvR}, \cref{def:DC}, and \cref{def:Dstorto}.
 Let $K_1, \oldC{C:1}$ be the constants defined in \cref{def:GvR}. Let $c_0,M,\lambda,\oldep{ep:1},C, \oldC{C:7}, N, \oldC{C:6}, \oldC{C:5}, \oldC{C:11}, \oldC{C:4}, \oldC{C:2}$ be the constants defined above, right before the statement of \cref{l:induction-alternative} and let $T$ be the compact cube introduced at the initial reduction argument at the beginning of the section, see the discussion from \cref{restrizionewlog} to \cref{ConstantsCK}.
  
  {
  There exists some constant $\alpha > 0$ depending only on $K$ and the constants above
  so that for every $\varepsilon>0$ and every $v,\rho$ sufficiently small with respect to the constants defined above and every basis $\mathfrak{E}$ as chosen above, we have
  {
  \begin{equation}\label{eqn:FinalCube2}
    \mathfrak{D}_{\mathfrak{E}}(v,\varepsilon v^{Qs^{2M}}, 2\oldC{C:5}M\rho) \subseteq DC  \left( \frac{\lambda v}{2\oldC{C:5}M},\alpha \varepsilon,2\oldC{C:5}M\rho \right).
  \end{equation}
  }
  }
\end{lemma}

\begin{proof}
{
   Let $x \in \mathfrak{D}_{\mathfrak{E}}(v,\varepsilon v^{Qs^{2M}}, 2\oldC{C:5}M\rho)$ and let $r < \rho$. Let $\mathfrak Q$ be the cube $\mathfrak{Q}:=\varphi(x)\cdot\delta_{vr}(T)$.  We will show that there exists a constant $C$ depending only on $K$ and the constants we already chose above such that
  \begin{align}
    \mathcal{H}^Q(\mathfrak{Q} \backslash \varphi(B(x,{2\oldC{C:5}Mr}) \cap K)) \leq C\varepsilon v^Q \mu(B(x,{2\oldC{C:5}Mr})).
    \label{stimacolcubo2}
  \end{align}
 Let us first show that \eqref{stimacolcubo2} is sufficient to conclude the proof. Thanks to Proposition~\ref{prop:tilings}(iv), there exists a $\lambda>0$ such that $B\Big(\varphi(x),v\lambda r\Big)\subseteq \mathfrak{Q}$ and hence
 \begin{equation}
 \begin{split}
     & \mathcal{H}^Q\Big(B\Big(\varphi(x),v\lambda r\Big)\cap \varphi(B(x,{2\oldC{C:5}Mr})\cap K)\Big)\\
     &\geq\mathcal{H}^Q\Big(B\Big(\varphi(x),v\lambda r\Big)\Big)-\mathcal{H}^Q\Big(\mathfrak{Q}\setminus \varphi(B(x,{2\oldC{C:5}Mr})\cap K)\Big) \\
    &\geq \mathcal{H}^Q\Big(B\Big(\varphi(x),v\lambda r\Big)\Big)-C\varepsilon v^Q\mu(B(x,{2\oldC{C:5}Mr}))\\
    &\geq\mathcal{H}^Q\Big(B\Big(\varphi(x),v\lambda r\Big)\Big)-Cj\varepsilon (2\oldC{C:5}Mrv)^Q\\
    &\geq  \Big({1-\frac{2^Q\oldC{C:5}^QM^QCj}{C'^Q\lambda^Q}}\varepsilon\Big)\mathcal{H}^Q\Big(B\Big(\varphi(x),v\lambda r\Big)\Big),
 \end{split}
 \end{equation}
 {where $C'$ is such that $\mathcal{H}^Q(B(0,r))=C'r^Q$, with $B(0,r)\subset\mathbb G$}. Hence, the previous computation shows that, calling $\alpha=\frac{2^Q\oldC{C:5}^QM^QCj}{C'^Q\lambda^Q}$, we have that 
 \[
 \mathcal{H}^Q\left(B\left(\varphi(x),\frac{\lambda v}{2\oldC{C:5}M}r'\right)\cap\varphi(B(x,r')\cap K)\right)\geq (1-\alpha\varepsilon)\mathcal{H}^Q\left(B\left(\varphi(x),\frac{\lambda v}{2\oldC{C:5}M}r'\right)\right),
 \]
 for every $x\in \mathfrak{D}_{\mathfrak{E}}(v,\varepsilon v^{Qs^{2M}}, 2\oldC{C:5}M\rho)$ and $r'<2\oldC{C:5}M\rho $. This is precisely \eqref{eqn:FinalCube2}.}
 
 We are thus reduced to prove \eqref{stimacolcubo2}.
 We now construct a tree $\mathfrak{T}$ of cubes of $\mathbb{G}$.
 The first layer of $\mathfrak{T}$ is the cube $\mathfrak{Q}$, whose centre is $\varphi(x)$. Let us now consider the $p_i$'s as defined in \cref{prop:tilings}. We will now apply \cref{l:induction-alternative} with $x\in K$, $x_{\mathfrak{Q}}:=\varphi(x)$, $0<r<2\oldC{C:5}M\rho$, and $0<v$, where $\rho,v$ are as in the statement of this Lemma. 
  Concerning the second layer of elements of $\mathfrak{T}$, if item (i) of \cref{l:induction-alternative} with the previously explained choices of variables \emph{does not hold}, then the construction of the tree stops. On the other hand, if the first alternative of Lemma~\ref{l:induction-alternative} is satisfied, we can find points $\{q_i\}_{i=1}^{2^Q} \subset K$, such that $\varphi(q_i)$ are close to $\varphi(x)\cdot \delta_{vr}(p_i)$, the centres of the quadrant subcubes $\{\mathfrak{Q}_i\}_{i=1}^{2^Q}$ of $\mathfrak{Q}$, where $\mathfrak{Q}_i:=\varphi(x)\cdot \delta_{vr} (p_i)\cdot\delta_{1/2}(\varphi(x)^{-1}\cdot\mathfrak{Q})$. Notice that $\varphi(x)\cdot \delta_{vr}(p_i)$ is the centre of the new cube $\mathfrak{Q}_i$.
 
As a second step, assuming the process continues, for every $j=1,\dots,2^Q$, we again consider \cref{l:induction-alternative} with $q_j\in K$, $0<r/2<\rho$, $x_{\mathfrak{Q}}=\varphi(x)\cdot \delta_{vr}(p_i)$. Notice that at this stage the hypotheses of \cref{l:induction-alternative} are met because of the estimate coming from item (i) of \cref{l:induction-alternative} at the previous stage and the fact that $\oldC{C:4}=\oldC{C:2}/10<\oldC{C:2}/2$.

If for some $\bar j$, the first alternative of Lemma~\ref{l:induction-alternative} holds as intended above, we add the subcubes of $\mathfrak{Q}_{\bar j}$ as leaves of the cube $\mathfrak{Q}_{\bar j}$ in $\mathfrak{T}$. If, on the other hand, the first alternative fails, we do not add leaves to $\mathfrak{Q}_{\bar j}$.
Inductively, we stop the process at each subcube where the first alternative fails and continue in the cubes where it doesn't, dividing $r$ by a further factor of $2$ at each stage, as depicted above.

Clearly, the cubes $\{\mathfrak{B}_i\}_{i\in\N}$ of $\mathfrak{T}$ without leaves are disjoint dyadic subcubes of $\mathfrak{Q}$. Let us pick $\mathfrak{B}_i$ without leaves, and let us assume $\mathfrak{B}_i$ belongs to the $\ell_i$-th stratum of the tree $\mathfrak{T}$. Let {$\alpha_i\in K\cap B(x,2\oldC{C:5}Mr)$} be the point yielded by the inductive construction for which $\varphi(\alpha_i)$ is close the centre of the cube $\mathfrak{B}_i$.

Thanks to Lemma~\ref{l:induction-alternative}, we know that 
\begin{itemize}
    \item[] there exists some $y_i \in B(\alpha_i,\oldC{C:5}Mr/2^{\ell_i}) \cap K$ so that $B(y_i, \oldC{C:6}{v^{s^{2M}}}r/2^{\ell_i}) \cap \mathfrak{G}(v, 2\oldC{C:5}M\rho) = \emptyset$ and $$
    \dist(\varphi(y_i),(\delta_{N}\mathfrak{B}_i)^c) \geq \oldC{C:7}vr/2^{\ell_i}.
    $$
\end{itemize}
Notice that by how we chose the constants $\oldC{C:6}v^{s^{2M}}r/2^{\ell_i}<\oldC{C:7}vr/2^{\ell_i}$. This implies that 
$$
B(\varphi(y_i),\oldC{C:6}v^{s^{2M}}r/2^{\ell_i})\subseteq B(\varphi(y_i),\oldC{C:7}vr/2^{\ell_i})\subseteq \delta_N\mathfrak{B}_i
$$
and thus
\begin{equation}\label{eqn:ContainBall}
\varphi(B(y_i, \mathrm{Lip}(\varphi)^{-1}\oldC{C:6}v^{s^{2M}}r/2^{\ell_i}))\subseteq B(\varphi(y_i), \oldC{C:6}v^{s^{2M}}r/2^{\ell_i}){\subseteq \delta_N\mathfrak{B}_i}.
\end{equation}
Let us define $B_i:=B(y_i,\mathrm{Lip}(\varphi)^{-1}\oldC{C:6}v^{s^{2M}}r/2^{\ell_i})$. From how the construction is defined and by homogeneity, we have that
\[
\mathcal{H}^Q(\delta_N\mathfrak{B}_i)=(N2^{-\ell_i}vr)^Q\mathcal{H}^Q(T),
\]
and, since $K\subset U(j,\overline R)$ for some $j\in\mathbb N$, see \eqref{eqn:ControlU},the assumptions of \cref{th:rep->davidcondpimr} and the discussion from \cref{restrizionewlog} to \cref{ConstantsCK}, we thus have \[
\mu(B_i)\geq j^{-1}(\mathrm{Lip}(\varphi)^{-1}\oldC{C:6}v^{s^{2M}}r/2^{\ell_i})^Q.
\]
Then, from the previous two equations, we infer that there exists a constant $\mathfrak{C}$ only depending on $j,N,\mathrm{Lip}(\varphi),\oldC{C:6},M,Q,s$ and the cube $T$ for which 
\begin{equation}\label{eqn:EstimateCube}
    \mathcal{H}^Q(\delta_N\mathfrak{B}_i)\leq \mathfrak{C}v^{Q(-s^{2M}+1)}\mu(B_i),
\end{equation}
for every cube $\mathfrak{B}_i$ that has no leaves.
 
{

Let us now notice that the balls $ B_i$ and $B_j$ relative to two disjoint cubes $\delta_N\mathfrak{B}_i$ and $\delta_N\mathfrak{B}_j$ must be disjoint. Indeed, if this were not the case, we would have that also $\varphi(B_i)\cap \varphi(B_j)\neq \emptyset$ and this would imply, by \eqref{eqn:ContainBall}, that $\delta_N\mathfrak{B}_i\cap \delta_N\mathfrak{B}_j\neq \emptyset$, which is in contradiction with the choice of the cubes.

\textbf{Claim.} There exists a constant $\mathfrak{c}>1$, and a subfamily $\{\bar{\mathfrak{B}}_\kappa\}_{\kappa\in\N}$ of $\{\mathfrak{B}_j\}_{j\in\N}$ such that
$$\bigcup_{i\in\N}\mathfrak B_i\subseteq \bigcup_{j\in\N} \delta_{\mathfrak{c}N}\bar{\mathfrak B}_j,$$
and the $\delta_N\bar{\mathfrak{B}}_{j}$'s
are pairwise disjoint. 

\smallskip

It is clear that $\bigcup_{i\in\N}\mathfrak{B}_i\subseteq \bigcup_{i\in\N}\delta_N\mathfrak{B}_i$ and, denoted by $\mathfrak{c}_i=\mathfrak{c}(\mathfrak{B}_i)$ the centre of the cube $\mathfrak{B}_i$ since $\lambda\mathrm{diam}(T)<4$, see Proposition~\ref{prop:tilings}, we obviously have
$$
\mathfrak{B}_i\subseteq B(\mathfrak{c}(\mathfrak{B}_i),4\lambda^{-1}\mathrm{diam}(T)^{-1}\mathrm{diam}(\mathfrak{B}_i)).
$$
Hence
$$\bigcup_{i\in\N}\mathfrak{B}_i\subseteq \bigcup_{i\in\N}B(\mathfrak{c}_i,4N\lambda^{-1}\mathrm{diam}(T)^{-1}\mathrm{diam}(\mathfrak{B}_i))\subseteq \bigcup_{j\in\N}B(\mathfrak{c}_{i_j},20N\lambda^{-1}\mathrm{diam}(T)^{-1}\mathrm{diam}(\mathfrak{B}_{i_j})),$$
where in the last inclusion we used the Vitali's covering lemma and the balls $$B(\mathfrak{c}_{i_j},4N\lambda^{-1}\mathrm{diam}(T)^{-1}\mathrm{diam}(\mathfrak{B}_{i_j})),$$
are pairwise disjoint. Hence, we infer that 
$$
\bigcup_{i\in\N}\mathfrak{B}_i\subseteq \bigcup_{j\in\N}\delta_{80N\lambda^{-2}\mathrm{diam}(T)^{-1}}(\mathfrak{B}_{i_j}),
$$
and thus the the result follows with $\mathfrak{c}=80\lambda^{-2}\mathrm{diam}(T)^{-1}$.

Taking into account the claim above, this implies that
  \begin{align}
    \sum_{i=1}^\infty \mathcal{H}^Q(\mathfrak{B}_i)\leq&\sum_{j=1}^\infty \mathcal{H}^Q(\delta_{\mathfrak{c}N}(\mathfrak{B}_{i_j}))=\mathfrak{c}^Q\sum_{j\in\N}\mathcal{H}^Q(\delta_{N}(\mathfrak{B}_{i_j}))\\
    &\leq \mathfrak{C}\mathfrak{c}^Qv^{Q(-s^{2M}+1)}\sum_{j=1}^\infty \mu(B_{i_j}) \leq \mathfrak{C}\mathfrak{c}^Qv^{Q(-s^{2M}+1)}\mu(B(x,{2\oldC{C:5}Mr}) \setminus \mathfrak{G}(v,2\oldC{C:5}M\rho))\\
    &\leq \mathfrak{C}\mathfrak{c}^Qv^{Q(-s^{2M}+1)} \varepsilon v^{Qs^{2M}} \mu(B(x,{2\oldC{C:5}Mr}))=\mathfrak{C}\mathfrak{c}^Q\varepsilon v^{Q} \mu(B(x,{2\oldC{C:5}Mr})). \label{e:holes-bound2}
  \end{align}
  where the second last inequality comes from the fact that $B_{i_j} \cap \mathfrak{G}(v, 2\oldC{C:5}M\rho) = \emptyset$, the fact that $B_{i_j}\subset B(x,2\oldC{C:5}Mr)$ by the triangle inequality, and the fact that $B_{i_j}$ are pairwise disjoint according to what we said above.

}

  { Let us now conclude the proof. For $\mathcal{H}^Q$-almost every $p \in \mathfrak{Q} \setminus \bigcup_{i=1}^\infty \mathfrak{B}_i$, we infer that 
  $$
  \mathcal{H}^Q \Big(B(p,s) \cap (\mathfrak{Q} \backslash \bigcup_{i=1}^\infty \mathfrak{B}_i)\Big) > 0, \qquad \text{for every $s > 0$}.
  $$
  }
Let us show that, for every $s>0$, there must exist a point $y\in B(x,2 \oldC{C:5}Mr)\cap K$ that is mapped by $\varphi$ into $B(p,s)$. Let $\{\mathfrak{Q}\}_{j\in\N}$ be a sequence of cubes of $\mathfrak{T}$ containing $p$. Our construction implies that for every $j\in\N$ there is a $\beta_j\in K\cap B(x,{2\oldC{C:5}M r})$
such that $d(\varphi(\beta_j),p)\leq C 2^{-j} rv$, where $C$ is a constant depending on the relevant constants in the Lemma. It thus suffices to choose $j$ sufficiently large and pass to the limit to prove the claim.  In particular $p$ is a limit point of $\varphi(B(x,{ 2\oldC{C:5}Mr}) \cap K)$ and since $\varphi(B(x,{ 2\oldC{C:5}Mr}) \cap K)$ is compact, we get that almost every point of $\mathfrak{Q} \setminus \bigcup_{i=1}^\N \mathfrak{B}_i$ is contained in $\varphi(B(x,{ 2\oldC{C:5}Mr}) \cap K)$.  Hence
  \begin{align*}
    \mathcal{H}^Q(\mathfrak{Q} \backslash \varphi(B(x,{2\oldC{C:5}Mr}) \cap K))\leq \mathcal{H}^Q\left( \bigcup_{i=1}^\infty\mathfrak{B}_i\right) \overset{\eqref{e:holes-bound2}}{\leq} \mathfrak{C}\mathfrak{c}^Q \varepsilon v^Q \mu(B(x,{ 2\oldC{C:5}Mr})),
  \end{align*}
  which is what we wanted.
\end{proof}
}

}

\begin{proof}[Proof of \cref{th:rep->davidcondpimr}]
By \cref{prop:Sopra} we know that $\mu  (K\setminus \bigcup_{j\in\mathbb N} \bigcup_{k\in\mathbb N} \bigcup_{\mathfrak{E}\in\mathscr{E}}\mathfrak{G}_\mathfrak{E}(1/j,1/k))=0$, where $\mathfrak{G}_{\mathfrak{E}}(\cdot,\cdot)$ is defined in \cref{def:GvR}. By Lebesgue  density theorem, for every $\varepsilon >0$, and every $j,k\in\mathbb N$ for which $\mathfrak{G}_\mathfrak{E}(1/j,1/k)$ is defined, we have 
\[
\mu\left(\mathfrak{G}_{\mathfrak{E}}(1/j,1/k)\setminus \bigcup_{{\mathfrak{E}\in \mathscr{E}}}\bigcup_{\ell\geq k} \mathfrak{D}_\mathfrak{E}(1/j,\varepsilon/j^{Qs^{2M}},1/\ell)\right)=0.
\]
Putting everything together, we conclude that for every $\varepsilon>0$ the following holds 
\[
\mu\left(K\setminus \bigcup_{{\mathfrak{E}\in \mathscr{E}}}\bigcup_{j\in\mathbb N}\bigcup_{k\in\mathbb N} \mathfrak{D}_\mathfrak{E}(1/j,\varepsilon/j^{Qs^{2M}},1/k)\right)=0,
\]
and thus, by using \cref{l:david-inclusion} we conclude that for every $\varepsilon>0$ we have 
\[
\mu\left(K\setminus \bigcup_{j\in\mathbb N}\bigcup_{k\in\mathbb N} DC(1/j,\varepsilon,1/k)\right)=0,
\]
which is what we wanted to prove.
\end{proof}

Let us now recall the following, which is an immediate corollary of \cite[Theorem~5.3]{bateli}, taking into account the fact that for every Carnot group $\mathbb G$ of homogeneous dimension $Q$ endowed with a homogeneous left-invariant distance $d$, $(\mathbb G,d,\mathcal{H}^Q)$ is $Q$-Ahlfors regular.
\begin{teorema}\label{th:DavidImplyRect}
Let $(X,d,\mu)$ be a metric measure space such that 
\[
0<\Theta^{Q}_*(\mu,x)\leq \Theta^{Q,*}(\mu,x)<+\infty,
\]
for $\mu$-almost every $x\in X$. Let $\varphi:(X,d)\to (\mathbb G,d)$ be a Lipschitz map. If $X$ satisfies David Condition with respect to $\varphi$, see \cref{def:DavidsCondition}, then there exists a countable number of Borel sets $\{U_i\}_{i\in\mathbb N}$ such that $\varphi|_{U_i}:U_i\to  \varphi(U_i)\subset\mathbb G$ is biLipschitz and $\varphi(X\setminus \cup_{i\in\mathbb N} U_i)=0$.
\end{teorema}

We are now ready to prove the main theorem of this section, i.e., \cref{thm:SeanLi}.

\begin{proof}[Proof of \cref{thm:SeanLi}]
{The fact that (i)$\Rightarrow$(ii) immediately follows by working in a biLipschitz chart, and using Pansu Theorem for Lipschitz functions from Borel subsets of $\mathbb G$ into $\mathbb H$. 
The fact that (ii)$\Rightarrow$(iii) is trivial. The fact that (iii)$\Rightarrow$(iv) is a direct consequence of 1$\Rightarrow$2 in \cref{chardiffspaces}.
The fact that (iv)$\Rightarrow$(v) is a direct consequence of \cref{th:rep->davidcondpimr}, and finally (v)$\Rightarrow$(i) comes from \cref{th:DavidImplyRect} and the density assumption.}
\end{proof}

\section{Examples, sharpness of the results, and future directions}

In this section, we discuss the optimality of the hypotheses in some of our results. Moreover, we present some examples and some possible future directions.

First, we remark here that one cannot take the atlas of a $(\mathbb G,\mathbb H)$-differentiable space to be made of charts that are complete for \textit{every} Lipschitz function.
{
\begin{osservazione}[Non-existence of an atlas that is complete with respect to every Lipschitz function]\label{rem:NonTuttoCompleto}
If $(X,d,\mu)$ is a $(\mathbb G,\mathbb H)$-Pansu differentiability space, in general, it is not possible to provide a system of charts for $X$ that is $(\mathbb{G},\mathbb{H})$-complete for \textit{every} Lipschitz function $f$. Indeed, let $X:=\mathbb{H}^1$ be the first Heisenberg group, $d$ be a homogeneous left-invariant distance, and $\mu=\mathcal{H}^4$. Notice that $\mathbb H^1$ is a $(\mathbb R^2,\mathbb R)$-differentiable space, by taking as a unique global chart the projection on the first stratum $\pi_1:\mathbb H^1\to\mathbb R^2$. Let now $U\subseteq \mathbb H^1$ be a Borel set with $\mathcal{H}^4(U)>0$, and $\varphi:\mathbb H^1\to \R^2$ be a Lipschitz function. 

Arguing as in Proposition~\ref{diffunif} it is possible to prove that there is a compact subset $\tilde{U}\subseteq U$ of positive $\mathcal{H}^4$-measure of $\mathbb H^1$ such that for every $x_0\in\tilde{U}$ the map $\varphi$ has a Pansu differential at $x_0$, and for every $\varepsilon>0$ there exists a $\delta>0$ such that 
$$
\frac{\lvert \varphi(y)-\varphi(x)-D\varphi(x_0)[x^{-1}y]\rvert} {d(x,y)}\leq \varepsilon,
$$
for every $x\in \tilde{U}\cap B(x_0,\delta)$ and every $y\in B(x,\delta)$.
Fix a density point $x_0\in \tilde{U}$ and let us choose $x_i:=x_0$ and $y_i$ as follows
\begin{equation}
   y_i:=x_0\cdot (0,0,2i^{-1}).
\end{equation}
This implies in particular that, since $D\varphi(x_0)$ sends the vertical line to $\{0\}$,
\begin{equation}
    \begin{split}
       0=&\lim_{i\to\infty}\frac{\lvert \varphi(y_i)-\varphi(x_i)-D\varphi(x_0)[x^{-1}_iy_i]\rvert} {d(x_i,y_i)}
       =\lim_{i\to\infty}\frac{\lvert \varphi(x_i)-\varphi(y_i)\rvert} {d(x_i,y_i)}. 
    \end{split}
\end{equation}
On the other hand, if we let $F_{x_0}(z):=d(x_0,z)$ we have
\begin{equation}
    \lim_{i\to\infty}\frac{\lvert F_{x_0}(x_i)-F_{x_0}(y_i)\rvert}{d(x_i,y_i)}=1.
    \nonumber
\end{equation}
This shows that for every Lipschitz function $\varphi: U\to \R^2$ with $\mathcal{H}^4(U)>0$, there are a Lipschitz function $F:\mathbb H^1\to \R$, a point $x_0\in U$, and two sequences $x_i\in U,y_i$ in $\mathbb{H}^1$ such that 
$$
\lim_{i\to\infty}\frac{\lvert \varphi(x_i)-\varphi(y_i)\rvert} {d(x_i,y_i)}=0\qquad\text{however}\qquad  \lim_{i\to\infty}\frac{\lvert F(x_i)-F(y_i)\rvert}{d(x_i,y_i)}=1,
$$
and thus $\varphi$ could not be a $(\mathbb H^1,\mathbb R)$-complete chart for every Lipschitz function from $\mathbb H^1$ to $\mathbb R$ because of the existence of $F$.
\end{osservazione}
}
\medskip

In the following remark, we notice that being an $n$-differentiability space with respect to a chart might not imply being an $(\mathbb R^n,\mathbb H)$-differentiability space, for some non-Abelian Carnot group $\mathbb H$, with respect to the same chart. 

\begin{osservazione}[Lipschitz differentiability space does not imply Pansu differentiability space with the same charts]
Notice that, as remarked in \cref{rem:Importante}, every non-Abelian Carnot group $\mathbb G$ is a $n_1$-Lipschitz differentiability space and the (unique global) chart can be taken to be the projection on the horizontal space $\varphi:(\mathbb G,d_{\mathbb G}) \to (V_1(\mathbb G),d_{\mathrm{eu}})$.

We claim that for every isometric embedding $f:\mathbb G\to\mathbb H$, and for every Borel set $U\subseteq \mathbb G$, the pair $(U,\varphi)$ is not $(\mathbb R^{n_1},\mathbb H)$-complete with respect to $f$ nor is a Lipschitz chart for $\mathbb H$-valued maps. Indeed, let us just consider an isometric embedding $f:(\mathbb G,d_{\mathbb G})\to(\mathbb H,d_{\mathbb H})$ and notice that, for every $x_0\in\mathbb G$ there is a sequence $x_i\to x_0$ made of pairwise distinct elements such that $\varphi(x_i)=\varphi(x_0)$. Hence, for every $i\in\mathbb N$, we have
\[
\frac{d_{\mathrm{eu}}(\varphi(x_0),\varphi(x_i))}{d_{\mathbb G}(x_0,x_i)}=0,\qquad \text{while}\qquad\frac{d_{\mathbb H}(f(x_0),f(x_i))}{d_{\mathbb G}(x_0,x_i)}=1.
\]

Thus, if a metric space $(X,d)$ is an $n$-Lipschitz differentiability space, it might not automatically be a $(\mathbb R^n,\mathbb G)$-Lipschitz differentiability space, for another group $\mathbb G$, with the same charts, because otherwise the previous equalities would be in contrast with \eqref{eqn:PANSUDIFF}.

{Nevertheless, for example, as a consequence of Pansu--Rademacher theorem, or as a particular case of the proof of the \cref{thm:Fondamentale2},  which only uses that Lipschitz curves in Carnot groups are almost everywhere differentiable and not the full Pansu--Rademacher theorem (see \cref{rem:ConfrontoConPansu})  every Carnot group $(\mathbb G,d_{\mathbb G})$ is a $(\mathbb G,\mathbb H)$-Pansu differentiability space for every Carnot group $\mathbb H$, and it suffices to take as a unique global chart the identity map.}
\end{osservazione}

In the following remark, we comment on the relation between \cref{thm:Fondamentale2} and Pansu--Rademacher theorem. 

\begin{osservazione}[Comparison of \cref{thm:Fondamentale2} with Pansu--Rademacher theorem]\label{rem:ConfrontoConPansu}
    By the classical Pansu--Rademacher theorem \cite{Pansu, MagnaniPhD} we know that every Lipschitz function $f:K\subset \mathbb G\to\mathbb H$, where $K$ is a measurable set and $\mathbb G,\mathbb H$ are Carnot groups, is differentiable almost everywhere, i.e., for almost every $x\in\mathbb G$ there is  a homogeneous homomorphism $Df(x):\mathbb G\to\mathbb H$ such that 
    \[
    0=\lim_{K\ni y\to x}\frac{\|(Df(x)(x^{-1}y))^{-1}f(x)^{-1}f(y)\|_{\mathbb H}}{\|x^{-1}y\|_{\mathbb G}}.
    \]
    The result in \cref{thm:Fondamentale2} can be understood as a metric generalization of Pansu--Rademacher theorem: whenever we have a metric space $X$ with $n_1$ $\varphi$-independent (in a Borel set $U$) curves on it - this is encoded in the existence of the Alberti representations - and also $(U,\varphi)$ is $(\mathbb G,\mathbb H)$-complete with respect to a Lipschitz function $f:X\to\mathbb H$, then we can differentiate $f$ in chart and get a homogeneous homomorphism $Df(x):\mathbb G\to\mathbb H$ for almost every $x$. 

    Notice that \cref{thm:Fondamentale2} does not use the full Pansu--Rademacher theorem stated above. Anyway, along the proof, see in particular \cref{prop:curvefullmeas}, and \cref{constructionvectorfields}, we are using that the Lipschitz fragments $\gamma:\mathrm{dom}(\gamma)\subset \mathbb R\to\mathbb L$, where $\mathbb L$ is an arbitrary Carnot group, are almost everywhere differentiable. Thus, taking $X=K\subseteq\mathbb G$ a $\mathcal{H}^Q$-positive measurable set in $\mathbb G$, and $\varphi:=\mathrm{id}_{\mathbb G}$, \cref{thm:Fondamentale2} gives a proof of the classical Pansu--Rademacher theorem, as soon as we know that all Lipschitz curves in Carnot groups are almost everywhere differentiable.
\end{osservazione}

In the following observation, we notice that most of the results in \cref{sec2} and \cref{sec3} have natural reformulations when $\mathbb H$ is a Banach-homogeneous group with the Radon-Nikodym property, see \cite{MagnaniRajala} for the definition.
\begin{osservazione}[Extension of the results when $\mathbb H$ is a Banach-homogeneous group with RNP]\label{rem:MoreGeneraltarget}
    Most of the results and the theory in \cref{sec2} and \cref{sec3} can be readily reformulated, with the obvious variations, in the case $\mathbb H$ is a Banach homogeneous group with the Radon-Nikodym property, see the introduction of \cite{MagnaniRajala}, or even more generally when it is a scalable group with the Radon--Nikodym property \cite{Scalable}. Again, compare \cref{rem:ConfrontoConPansu}, a key property is that all Lipschitz curves in Banach-homogeneous groups with the Radon-Nikodym property are differentiable almost everywhere, see \cite[Theorem~3.1]{MagnaniRajala}. 

    For example, let us discuss what happens in \cref{thm:Fondamentale2} when $\mathbb H$ is a Banach-homogeneous group with the Radon-Nikodym property. First, by projection on the horizontal stratum, we deduce the analog of \cref{prophordiff}, and thus we have the {horizontal differential $D_{\mathrm{H}}f(x):V_1(\mathbb G)\to V_1(\mathbb H)$ for almost every $x$}. Then, using that Lipschitz curves are differentiable, see \cite[Theorem~3.1]{MagnaniRajala}, we first find a good family of curves on our space $X$ as in \cref{constructionvectorfields}; and then we run the argument of the proof of \cref{thm:Fondamentale2} to get the existence of a homogeneous homomorphism $Df(x):\mathbb G\to\mathbb H$ for $\mu$-a.e. $x\in U$. Thus, notice that by using \cite[Theorem~3.1]{MagnaniRajala} and \cref{thm:Fondamentale2} we get as a corollary \cite[Theorem~1.1]{MagnaniRajala}. We notice, anyway, that the idea of the proof of \cite[Theorem~1.1]{MagnaniRajala} is similar to that of our \cref{thm:Fondamentale2}.\\ 
    Similarly, in \cref{chardiffspaces} and \cref{thm:SeanLiINTRO}, we only use that Lipschitz curves into $\mathbb H$ are intrinsically differentiable. 
\end{osservazione}

We discuss a criterion, based on the existence of certain biLipschitz embeddings, to show that a space is not a $(\mathbb G,\mathbb H)$-differentiability space, see \cref{cor:Embeddability}. We start with a remark.
{
\begin{osservazione}[Towards non-embeddability of Pansu differentiable spaces]\label{rem:BiLipschitzEmbedding}
Let $\mathbb G$ be a Carnot group, and let $\mathbb H$ be Carnot groups of homogeneous  dimension $Q'$.
Let us suppose there exists a biLipschitz embedding $f:X\to\mathbb H$ with $f^\sharp\mathcal{H}^{Q'}\ll \mu$. Let $(U,\varphi)$ be a Lipschitz chart with $\varphi:X\to \mathbb G$ as in the hypotheses of \cref{thm:Fondamentale2}, and with $\mathcal{H}^{Q'}(f(U))>0$. We aim at showing that for $\mu$-almost every $x_0\in U$ the homogeneous homomorphism $Df(x_0)$, whose existence is guaranteed by \cref{thm:Fondamentale2}, is injective.

Let us fix $\varepsilon>0$ small enough, and let us take $\tilde U\subset U$ such that $\mu(U\setminus\tilde U)\leq \varepsilon\mu(U)$ and all the properties in Step 3 of \cref{thm:Fondamentale2} hold. 
First, as a consequence of \cref{thm:Fondamentale2}, we have that the map $f:X\to f(X)\subset \mathbb H$ is almost everywhere differentiable with respect to $\varphi$, which means that, for $\mu$-almost every $x_0\in \tilde U$ there is a homogeneous homomorphism $Df(x_0):\mathbb G\to\mathbb H$ such that, as $x\to x_0$,
\begin{equation}\label{eqn:Diff1}
\|Df(x_0)(\varphi(x_0)^{-1}\varphi(x))^{-1}f(x_0)^{-1}f(x)\|_\mathbb H=o(d(x,x_0)).
\end{equation}
Moreover, being $\varphi\circ f^{-1}:=g:f(X)\subset\mathbb H\to\mathbb G$, we have that $g$ is Lipschitz and $g\circ f=\varphi$. 

We claim that, for $\mathcal{H}^{Q'}$-almost every $y\in f(\tilde U)$, we have that $f$ is differentiable at $f^{-1}(y)$ with respect to $\varphi$, and $y$ is a differentiability point of $g$. Indeed, since $g$ is Lipschitz and $\mathcal{H}^{Q'}(f(\tilde U))>0$, we have that, calling $D$ the differentiability points of $g$ in $f(\tilde U)$, we have $\mathcal{H}^{Q'}(D)>0$ and $\mathcal{H}^{Q'}(f(\tilde U)\setminus D)=0$. Moreover, let $A\subset \tilde U$ be the set of points where $f$ is not differentiable with respect to $\varphi$. Hence $\mu(A)=0$, and thus, from the assumption, $\mathcal{H}^{Q'}(f(A))=0$. Hence $f(\tilde U\setminus A)\cap D$ is of $\mathcal{H}^{Q'}$-full measure in $f(\tilde U)$ and the claim is proved. 

Let us take $f(x_0)\in f(\tilde U\setminus A)\cap D$. Hence, for $y\to x_0$,
\begin{equation}
    \|Dg(f(x_0))(f(x_0)^{-1}y)^{-1}g(f(x_0))^{-1}g(y)\|_{\mathbb G} = o(d_{\mathbb H}(y,f(x_0)),
\end{equation}
and thus taking $y=f(x)$ and using that $f$ is biLipschitz we have, as $x\to x_0$,
\begin{equation}\label{eqn:Diff2}
    \|Dg(f(x_0))(f(x_0)^{-1}f(x))^{-1}\varphi(x_0)^{-1}\varphi(x)\|_{\mathbb G} = o(d(x,x_0)),
\end{equation}
{ and combining \eqref{eqn:Diff1} and \eqref{eqn:Diff2} we get that, as $X\ni x\to x_0$,
\begin{equation}\label{eqn:Diff3}
    \left\|  (Dg(f(x_0))\circ Df(x_0))  (\varphi(x_0)^{-1}\varphi(x))\cdot\varphi(x_0)^{-1}\varphi(x)\right\|_{\mathbb G} = o(d(x,x_0)),
\end{equation}

By using the Step 3 in \cref{thm:Fondamentale2}, see in particular item (a) and item (b) in there, we have that the following holds. For every $x_0\in \tilde U$, every $v\in\mathbb G$, and every $t>0$ small enough there is $x_t\in X$ such that $x_t\to x_0$ as $t\to 0$, and, for every $t>0$ small enough,
\[
    \varphi(x_0)^{-1}\varphi(x_t)=\delta_t(v)\Delta_{x_0}(t),
\]
where $\|\Delta_{x_0}(t)\|/t\to 0$ as $t\to 0$. Plugging $x=x_t$ in \eqref{eqn:Diff3}, letting $t\to 0$, using the estimate in item (b) in Step 3 of \cref{thm:Fondamentale2}, and using that $v$ is arbitrary, we deduce that for $\mu$-almost every $x_0\in\tilde U$ we have $Dg(f(x_0))\circ Df(x_0)=\mathrm{id}_{\mathbb G}$. By the arbitrariness of $\varepsilon>0$ we get that the previous equality holds for $\mu$-almost every $x_0\in U$, and thus $Df(x_0)$ is injective for $\mu$-almost every $x_0\in U$ which is the sought claim.
}
\end{osservazione}
}

\begin{corollario}[Non-embeddability of Pansu differentiable spaces]\label{cor:Embeddability}
Let $(X,d,\mu)$ be a metric measure space such that there exists a biLipschitz embedding $f:X\to \mathbb R^n$ with $\mathcal{H}^n(f(X))>0$ and $f^\sharp\mathcal{H}^n\ll \mu$. Then for every non-Abelian Carnot group $\mathbb G$, the metric measure space $(X,d,\mu)$ does not have the structure of a $(\mathbb G,\mathbb R^n)$-differentiability space.
\end{corollario}
\begin{proof}
Let us assume by contradiction that $(X,d,\mu)$ has the structure of a $(\mathbb G,\mathbb R^n)$-differentiability space. Hence by \cref{chardiffspaces} we have that there exist a Lipschitz chart $(U,\varphi)$, with $\varphi:X\to\mathbb G$ as in the hypotheses of \cref{thm:Fondamentale2} with $\mathcal{H}^n(f(U))>0$. Hence, \cref{rem:BiLipschitzEmbedding} tells us that for $\mu$-almost every $x_0\in U$ we have that $Df(x_0):\mathbb G\to\mathbb R^n$ is injective, which is a contradiction since $\mathbb G$ is not Abelian.
\end{proof}

Let us now comment on the sharpness of the complete assumption in \cref{chardiffspaces}.
\begin{osservazione}[Sharpness of the complete assumption in \cref{chardiffspaces}]\label{rem:NotDropComplete}
We stress that \cref{chardiffspaces} is false if we remove the third bullet in the second item.
This is false as soon as the Carnot group is not Abelian. For simplicity, we just present the situation in the Heisenberg group.
Namely,
we shall prove that the first Heisenberg group $(\mathbb{H}^1,\mathcal{H}^4)$
\begin{enumerate}
    \item is not a $(\R^2,\mathbb{H}^1)$-Pansu differentiability space;
    \item satisfies the first and the second bullet in the second item of \cref{chardiffspaces} with $\mathbb G=\mathbb R^2$ and $\mathbb H=\mathbb H^1$.
\end{enumerate}

To prove item 1 above, we claim that no chart $(\varphi,U)$ with $\mathcal{H}^4(U)>0$ and $\varphi:\mathbb{H}^1\to \R^2$ Lipschitz can be $(\R^2,\mathbb{H}^1)$-complete with respect to the identity map $\mathrm{id}:\mathbb{H}^1\to\mathbb{H}^1$. Consequently, this claim will conclude that $(\mathbb H^1,\mathcal{H}^4)$ is not a $(\mathbb R^2,\mathbb H^1)$-differentiability space by means of \cref{chardiffspaces}.

Indeed, by Pansu--Rademacher theorem for every $U,\varphi$ as above, $\varphi$ is differentiable at $\mathcal{H}^4$-almost every point $x_0\in U$. Hence there exists $x_0\in U$ such that $D\varphi(x_0):\mathbb H^1\to\mathbb R^2$ is a homogeneous homomorphism.  Hence, because the Carnot group is not abelian, there exists some $v\neq 0$ for which $D\varphi(x_0)(v)=0$. As a consequence, for a sequence $t_i\to 0$, $\varphi(x_0\cdot \delta_{t_i}v)-\varphi(x_0)=o(t_i)$. If we let $x_i:=x_0\cdot\delta_{t_i}v$ this shows that we found a point $x_0\in U$ and $x_i\to x_0$ with $x_i\neq x_0$, such that 
\[
\lim_{i\to +\infty}\frac{|\varphi(x_i)-\varphi(x_0)|_{\mathbb R^2}}{\|x_0^{-1}x_i\|_{\mathbb H^1}}=0,
\]
however $\lim_{i\to \infty}\lVert \mathrm{id}(x_0)^{-1}\mathrm{id}(x_i)\rVert/d(x_i,x_0)=\lim_{i\to \infty}\lVert x_0^{-1}x_i\rVert/d(x_i,x_0)=1.$

To prove item 2 above, it is readily seen that for the projection on the first stratum $\varphi:\mathbb H^1\to\mathbb R^2$ one can find two horizontally universal (with respect to $\mathbb H^1$) $\varphi$-independent Alberti representations of $\mathcal{H}^4$. On a fixed compact set, it suffices to take the two representations concentrated on a set of horizontal segments associated with two independent horizontal directions.
\end{osservazione}

We next stress that, in the generality of Carnot groups, \cref{thm:SeanLi} does not hold if, instead of item (i), we require that $(X,d)$ is just Lipschitz rectifiable.

\begin{osservazione}[About asking Lipschitz rectifiability in item (i) of \cref{thm:SeanLi}]\label{rem:PerForzaBilip}
\cref{thm:SeanLi} does not hold if, instead of item (i), we require that $(X,d)$ is just Lipschitz rectifiable. Indeed, there exists a distance $d$ on the first Heisenberg group $\mathbb H^1$ such that $d\leq d_{cc}$, $d$ is $4$-Ahlfors regular, and such that $d$ is not biLipschitz equivalent to $d_{cc}$ on every $\mathcal{H}^4_{d_{cc}}$-positive-measured set. Here $d_{cc}$ is a Carnot--Carathéodory distance on $\mathbb H^1$. The example is in \cite[Theorem~1.2]{LeDonneLiRajala}. Thus $(\mathbb H^1,d)$ is Lipschitz rectifiable and $4$-Ahlfors regular, but it is not biLipschitz rectifiable.

In the case $\mathbb G=\mathbb R^n$ for some $n\in\mathbb N$, we have that item (i) in \cref{thm:SeanLi} can be relaxed to $(X,d)$ is $n$-Lipschitz rectifiable as proved in \cite[Theorem~1.2]{bateli}. The latter is a consequence of a fact due to Kirchheim \cite{Kirchheim}.
\end{osservazione}

{
\begin{osservazione}\label{rk:laakso}
    It is known that Laakso spaces \cite{Laakso} are PI spaces, and thus Lipschitz differentiability spaces by \cite{CheegerGAFA}. It would be interesting to construct similar examples of Pansu differentiability spaces with a non-Abelian Carnot model.
\end{osservazione}

\begin{osservazione}[PI spaces, Lip-lip inequality, and differentiability spaces]\label{rem:LipLip}
It is known that on a PI space a Lip-lip inequality holds, see \cite{CheegerGAFA}, and \cite[Section 6]{KleinerMackay} for an account of this fact. Such an inequality is enough to show that the space is a Lipschitz differentiability space. It is also known that a complete metric measure space is an RNP-differentiability space if and only if it is rectifiable with PI spaces and every porous set has measure zero \cite[Theorem~1.1]{ErikssonBiqueGAFA}.\\
Moreover, it is known that a metric measure space is a Lipschitz differentiability space if and only if a Lip-lip inequality holds and it is asymptotically doubling, see \cite[Theorem~10.5]{BateJAMS}.\\
It would be interesting to understand what are the analogous statements in the non-Abelian Carnot setting.
\end{osservazione}

}

{
\begin{osservazione}

We developed the language of $(\mathbb G,\mathbb H)$-Pansu differentiability spaces in order to lay the foundations to approach the following conjecture. For the notions of pointed measure Gromov--Hausdorff tangent (pmGH for short), we refer to \cite{GigliMondinoSavare15}.

\begin{congettura}\label{cong}
Let $(X,d)$ be a complete metric space, and let $\mathbb G$ be a Carnot group of homogeneous dimension $Q$.
Let $E\subseteq X$ be a Borel subset such that $0<\mathcal{H}^Q(E)<+\infty$. Then the following are equivalent
\begin{enumerate}
    \item We have 
    \[
    \Theta^Q_*(\mathcal{H}^Q\llcorner E,x)>0,\qquad \text{for $\mathcal{H}^Q$-almost every $x\in X$,}
    \]
    and, for $\mathcal{H}^Q$-a.e. $x\in X$, there exists a constant $K_x$ such that all the pmGH tangents of $(X,d,\mathcal{H}^Q\llcorner E)$ at $x$ are supported on metric spaces that are $K_x$-biLipschitz equivalent to $\mathbb G$.
    \item $(E,d)$ is biLipschitz $\mathbb G$-rectifiable.
\end{enumerate}
\end{congettura}

The implication 2$\Rightarrow$ 1 is straightforward by blow-up. The previous statement, for $\mathbb G=\mathbb R^n$, is a consequence of \cite[Theorem~1.2]{BateInv}, and it can be understood as a metric version of Marstrand--Mattila's rectifiability criterion, see also \cite[Theorem~1.1]{BateInv}. For a version of Marstrand--Mattila's rectifiability criterion for $\mathscr{P}$-rectifiable measures in Carnot groups, see \cite{antonelli2020rectifiableB,MarstrandMattila20}. It would be interesting to understand whether the metric version of Federer's projection theorem in \cite{BateActa} can be extended to our setting, and this seems to be an important step to prove the previous conjecture.\\
The implication 1$\Rightarrow$2 in \cref{cong} is known to be true for equiregular sub-Riemannian manifolds, as a consequence of the results in \cite{LDY19}. We stress that, since there exists a nilpotent Lie group equipped with a left-invariant sub-Riemannian metric that is not locally quasiconformally equivalent to its tangent cone, see \cite[Theorem 1.1]{LDOW}, the implication 1$\Rightarrow$2 in \cref{cong} might be false if we require the biLipschtiz charts to be defined on open sets.\\
Notice that it might not be true that in the hypotheses of \cref{cong} the tangent is unique almost everywhere, or that the density exists almost everywhere, as it happens in the Euclidean case \cite{BateInv, Kirchheim}. This is due to the following example: there exists a left-invariant distance $d$ on $\mathbb H^1$ that is biLipschitz equivalent to a Carnot--Carathéodory distance $d_{\mathrm{cc}}$, such that no blow-up is self-similar, see \cite[Theorem~1.5]{LeDonneLiRajala}.\\
We also stress that it is important that the model Carnot group $\mathbb G$ is fixed in \cref{cong}. Indeed, there are examples (\cite[Theorem~1.1]{ALD}) of hypersurfaces in Carnot groups such that their tangents are everywhere unique and they are Carnot groups that might depend on the point, but for every Carnot group $\mathbb G$ they are not $\mathbb G$-Lipschitz rectifiable.
\end{osservazione}
}

\section{Appendix}\label{appendixcubes}

Throughout this section, we setwise identify $\mathbb{G}$ with $\R^n$ via exponential coordinates with respect to a basis adapted to the stratification. With the number $ n_1$ we denote the rank of $\mathbb{G}$, which is the dimension of its first stratum  $V_1(\mathbb{G})$. 

\subsection{Endpoint drift for fragments}

We collect here two technical results that have been used in the paper. \cref{lemma:drift} is used in the proof of \cref{l:induction-alternative}, and \cref{cubofigo} is used in the proof of the (iv)$\Rightarrow$(v) of \cref{thm:SeanLiINTRO}.
\begin{lemma}[Endpoint drift]\label{lemma:drift}
Let 
 $C$ be a compact set of the real line and let $\gamma:C\to \mathbb{G}$ be a Lipschitz curve in the Carnot group $\mathbb{G}$. Suppose  
that there exists  $t_0\in C$,  $e\in\mathbb{S}^{n_1-1}$,  $0<\sigma<1$, and  $R>0$, such that
\begin{itemize}
    \item[(\hypertarget{driftendpointi}{i})] for every $0< r < R$ we have $\mathcal{L}^1(B(t_0,r) \cap C) > (1-\sigma)\mathcal{L}^1(B(t_0,r))$;
    \item[(ii)] for every $t>t' \in C$, $\pi_1(\gamma(t))-\pi_1(\gamma(t')) \in C(e, \sigma)$.
\end{itemize}
Then, there exists a constant $\newC\label{C:1}>0$ depending only on $\mathbb G$, $\mathrm{Lip}(\gamma)$ and $\mathrm{diam}(C)$, such that for every $0<\lvert\rho\rvert<R$ we have
\begin{equation}
   d\left(\gamma(t_0+\rho),\gamma(t_0)\left[\left(\int_0^\rho\lvert D\gamma(s)\rvert d\mathcal{L}^1\llcorner C(s)\right)\,e\right]\right)\leq \oldC{C:1}\sigma^{1/s}\rho, 
   \label{eq:stimadriftix}
\end{equation}
for every $\rho$ for which $\rho+t_0\in C$.
\end{lemma}

\begin{proof}
We can assume without loss of generality that $t_0=0$ and $\gamma(t_0)=0$.
For every $-R<r<R$ for which $r\in C\setminus\{0\}$ let us define $\bar C:=C/r\cap [0,1]$
and let us denote $\bar \gamma:\bar C\to \mathbb{G}$ the curve $s\mapsto \delta_{1/r}\left(\gamma(rs)\right)$. Further, we extend $\bar \gamma$ to a curve $\bar{\bar{\gamma}}:[0,1]\to \mathbb{G}$
in such a way that $\bar{\bar{\gamma}}\vert_{\bar C}=\bar \gamma$ on $\bar C$ and on $[0,1]\setminus \bar C=\cup_{i\in\N} (a_i,b_i)$ we have that $\bar{\bar{\gamma}}$ coincides with a  geodesic connecting $\gamma(a_i)$ and $\gamma(b_i)$. We parametrize $\bar{\bar\gamma}$ on each $(a_i,b_i)$ in such a way that $|D\bar{\bar\gamma}|=1$ holds almost everywhere on $(a_i,b_i)$.

Then, seeing $\bar{\bar{ \gamma}}$ as an Euclidean Lipschitz curve, and denoting by $\mathscr{C}(x)$ the differential of the left translation by $x$, by \cite[Proposition~1.3.3]{tesimonti} we have that for every $\mathfrak{t}\in [0,1]$
\begin{equation}
    \begin{split}
        &\bar{\bar{\gamma}}(\mathfrak{t})-\left(\int_0^\mathfrak{t} \lvert D{\bar{\gamma}}(s)\rvert d\mathcal{L}^1\llcorner     \bar C(s)\right) e=\int_0^\mathfrak{t} \left(\mathscr{C}(\bar{\bar{\gamma}}(s))[D\bar{\bar{\gamma}}(s)]-(\chi_{\bar C}(s)\lvert D{\bar{\gamma}}(s)\rvert) e\right)d\mathcal{L}^1(s)\\
        &=\Delta(\mathfrak{t})+\int_0^\mathfrak{t} \left(\mathscr{C}(\bar{\gamma}(s))[D\bar{\gamma}(s)]-\lvert D\bar{\gamma}(s)\rvert e\right)d\mathcal{L}^1\llcorner \bar C(s),
          \label{eq:integrali1}
    \end{split}
\end{equation}
where the last identity in \eqref{eq:integrali1} follows by noting that 
$D\bar\gamma(s)=D\bar{\bar{\gamma}}(s)$ for $\mathcal{L}^1$-almost every $x\in \bar C$, and with the obvious meaning of $\Delta(t)$. 

On the other hand, thanks to item (\hyperlink{driftendpointi}{i}) we know that 
{
$$
\mathcal{L}^1([0,\mathfrak{t}]\setminus \bar C)\leq \mathfrak{t}-(1-\sigma)\mathfrak{t}=\sigma\mathfrak{t},
$$
}
{
and hence 
\begin{equation}\label{eqn:EST13}
\lvert \Delta(\mathfrak{t})\rvert\leq \lVert\mathscr{C}\rVert_{\infty,(\overline{B}(0,10\mathrm{diam}(\mathrm{im}\gamma)))}\sigma\mathfrak{t}.
\end{equation}
}

In addition, we further have
\begin{equation}\label{eqn:EST12}
    \begin{split}
        &\bar{\bar{\gamma}}(\mathfrak{t})-\left(\int_0^\mathfrak{t} \lvert D{\bar{\gamma}}(s)\rvert d\mathcal{L}^1\llcorner \bar C(s)\right) e
        =\Delta(\mathfrak{t})+\int_0^\mathfrak{t} (\mathscr{C}(\bar{\gamma}(s))[D\bar{\gamma}(s)]-\lvert D\bar{\gamma}(s)\rvert e)d\mathcal{L}^1\llcorner \bar C(s)\\
          &=\Delta(\mathfrak{t})+\int_0^\mathfrak{t} \mathscr{C}(\bar \gamma(s))[D\bar \gamma(s)-\lvert D\bar\gamma(s)\rvert e]d\mathcal{L}^1\llcorner \bar C(s)\\
          &\qquad\qquad\qquad\qquad\qquad\qquad+\int_0^\mathfrak{t} \Big(\mathscr{C}(\bar \gamma(s))-\mathscr{C}\Big(\left(\int_0^s \lvert \chi_{\bar C}(\tau)D{\bar\gamma}(\tau)\rvert d\tau\right) e\Big)\Big)[\lvert D\bar\gamma(s)\rvert e]d\mathcal{L}^1\llcorner \bar C(s)
    \end{split}
\end{equation}
where the last identity above follows 
from the fact that $\mathscr{C}(\left(\int_0^s \lvert \chi_{\bar C}(\tau)D{\bar\gamma}(\tau)\rvert d\tau \right)e)[e]=e$. 
{
If we fix $s\in [0,\mathfrak{t}]\cap\bar C$, and we call $\alpha:=D\bar\gamma(s)$, we have by item (ii) that $\langle \alpha,e\rangle\geq (1-\sigma^2)|\alpha|$. Hence, by calling $\alpha^\perp$ the orthogonal projection of $\alpha$ onto $e^\perp$, we have that $|\alpha^\perp|\leq \sqrt{1-(1-\sigma^2)^2}|\alpha|\leq 2\sigma|\alpha|$. Hence 
\[
|\alpha-|\alpha|e|\leq |\alpha^\perp|+|\langle\alpha,e\rangle-|\alpha||\leq (2\sigma+\sigma^2)|\alpha|\leq 3\sigma|\alpha|.
\]
Thus, we have showed that for every $s\in [0,\mathfrak t]\cap\bar C$ we have
\begin{equation}\label{eqn:EST14}
\big\lvert D\bar \gamma(s)-\lvert D\bar\gamma(s)\rvert e\big\rvert\leq 3\sigma |D\bar \gamma(s)|\leq 3\sigma \mathrm{Lip}(\gamma).
\end{equation}
}
Hence, we infer from \eqref{eqn:EST12}, \eqref{eqn:EST13}, and \eqref{eqn:EST14} that, calling $\vartheta:=\lVert\mathscr{C}\rVert_{\infty,(\overline{B}(0,10\mathrm{diam}(\mathrm{im}\gamma)))}$, and $\Theta:=\mathrm{Lip}(\mathscr{C}\lvert_{\bar B(0,10\mathrm{diam}(\mathrm{im}\gamma))})$, for every $\mathfrak{t}\in [0,1]$ we have
\begin{equation}\label{eqn:EVVAI}
    \begin{split}
        \left| \bar{\bar\gamma}(\mathfrak{t})-\left(\int_0^\mathfrak{t}\lvert D{\bar\gamma}(s)\rvert d\mathcal{L}^1\llcorner\bar C(s)\right)e\right|&\leq\vartheta\sigma\mathfrak{t}+ \vartheta\mathrm{Lip}(\gamma)\cdot 3\sigma\mathfrak{t}\\
        &+\Theta\mathrm{Lip}(\gamma)\int_0^\mathfrak{t}\Big\lvert \bar\gamma(s)- \Big(\int_0^s \lvert D{\bar\gamma}(\tau)\rvert d\mathcal{L}^1\llcorner\bar C(\tau) \Big) e\Big\rvert d\mathcal{L}^1\llcorner \bar C(s).
    \end{split}
\end{equation}
Denote $A(t):=(\vartheta\sigma+3\vartheta\mathrm{Lip}(\gamma)\sigma)t$ and $B:=\Theta\mathrm{Lip}(\gamma)$. By using the previous inequality at each value $\mathfrak t\in[0,1]\cap\bar C$, we conclude that, by Gr\"onwall Lemma, for every $\mathfrak t \in [0,1]\cap\bar C$ we have
\begin{equation}\label{eqn:EVVAI2}
    \begin{split}
      &  \left| \bar\gamma(\mathfrak t)-\left(\int_0^{\mathfrak t}\lvert D{\bar\gamma}(s)\rvert d\mathcal{L}^1\llcorner \bar C(s)\right) e\right|\leq A(\mathfrak t)+B\int_0^{\mathfrak t} A(s)e^{B(\mathfrak t-s)}ds
    \end{split}
\end{equation}
In particular, evaluating \eqref{eqn:EVVAI2} at $\mathfrak t=1$, since by construction $1\in\bar C$, we get that there exists a constant $\mathfrak{c}>0$ depending only on $\vartheta,\Theta,\mathrm{Lip}(\gamma)$ such that 
$$
\left| {\bar\gamma}(1)-\left(\int_0^1\lvert D{\bar\gamma}(s)\rvert d\mathcal{L}^1\llcorner\bar C(s)\right)e\right|\leq\mathfrak{c}\sigma.
$$
However, thanks to \cite[Proposition~5.15.1]{MR2363343}, we know that up to a constant $\bar {\mathfrak{c}}>0$ that depends only on $\mathrm{Lip}(\gamma)$ and $\mathrm{diam}(C)$, we have, for every $r\in C$,
\begin{equation}
    \begin{split}
        d\Big(\gamma(r),\left(\int_0^r\lvert D\gamma(s)\rvert d\mathcal{L}^1\llcorner C(s)\right)\, e\Big)&= r d\Big(\delta_{1/r}(\gamma(r)),\left(\frac{1}{r}\int_0^r\lvert D\gamma(s)\rvert d\mathcal{L}^1\llcorner C(s)\right)\, e\Big)\\
        &=rd\Big(\bar\gamma(1),\left(\int_0^1\lvert D\bar\gamma(s)\rvert d\mathcal{L}^1\llcorner \bar C(s)\right)\, e\Big)\\
        &\leq \bar{\mathfrak{c}}r\left| \bar\gamma(1)-\left(\int_0^1\lvert D\bar\gamma(s)\rvert d\mathcal{L}^1\llcorner \bar C(s)\right) e\right|^{1/s}
        \leq \bar{\mathfrak{c}}\mathfrak{c}^{1/s}\sigma^{1/s}r.
        \nonumber
    \end{split}
\end{equation}
This concludes the proof of the proposition.
\end{proof}

\subsection{On dyadic tiles in Carnot groups}\label{sec82}

\begin{proposizione}[{\cite[Theorem~3.1, Proposition~3.4 and Lemma~3.5]{tilings}}]\label{prop:tilings} Let $Q$ be homogeneous dimension of the Carnot group $\mathbb G$. Then, we can find $2^Q$ points $p_1,\ldots,p_{2^Q}$ and a compact set $T\subseteq \mathbb{G}$ such that
\begin{itemize}
    \item[(i)] $T$ is the closure of an open set and thus $\mathcal{H}^Q(T)>0$;
    \item[(ii)] $T=\bigcup_{j=1}^{2^Q}p_j\cdot\delta_{1/2}(T)$;
    \item[(iii)] $\mathcal{H}^Q(p_j\cdot\delta_{1/2}(T)\cap p_k\cdot\delta_{1/2}(T))=0$ for every $j\neq k\in\{1,\ldots,2^Q\}$;
    \item[(iv)] $0$ is in the interior of $T$.
\end{itemize}
\end{proposizione}
Let $T$ be as in the statement of \cref{prop:tilings}. From now on, we  fix
    $0<\lambda<\mathrm{diam}(T)/4$ such that $B(0,\lambda)\subseteq T$. This $\lambda$ exists thanks to item (iv) of \cref{prop:tilings}.

We will sometimes refer to the $T$ in \cref{prop:tilings} as a {\em cube}. We will sometimes denote the sets $p_j\cdot \delta_{1/2}(T)$ as the {\em subcubes} of $T$.

\begin{osservazione}\label{FTV}
Let $\mathfrak{V}=\{v_1,\ldots,v_{n_1}\}$ be a basis of $V_1(\mathbb{G})$ and let $\mathfrak{W}^\kappa=\{w_1^\kappa,\ldots,w^\kappa_{n_1}\}$ be a sequence of bases converging to $\mathfrak{V}$, i.e., for every $\ell=1,\ldots,n_1$ we have
$$\lim_{\kappa\to\infty} w_\ell^\kappa=v_\ell \qquad\text{for every }\ell=1,\ldots,n_1.$$
Let $F_\mathfrak{V}$ be the map constructed in \cref{propdecomposizione} and for every $\kappa\in\N$ let us define
\begin{equation}\label{eqn:Fhat}
\hat F_{\mathfrak{W}^\kappa}(s_1,\dots,s_n):= \delta_{s_1}(w_{i_1}^\kappa)\dots\delta_{s_{n}}(w^\kappa_{i_{n}})\delta_{-\hat s_{n}}(w^\kappa_{i_{n}})\dots\delta_{-\hat s_{1}}(w^\kappa_{i_{1}}).
\end{equation}
Notice that the indices $i_1,\ldots,i_{2n}$ and the numbers $\hat s_i$ are the ones associated to the basis $\mathfrak{V}$, as produced in \cref{propdecomposizione}.
It is elementary to see that $\hat F_{\mathfrak{W}^\kappa}$ converges on compacts to $F_{\mathfrak{V}}$ in the $C^\infty$-metric. This is due to the following observation. Define
$\mathscr{F}:V\times (\mathbb{S}^{n_1-1})^{n_1}\to \R^{n}\equiv \mathbb{G}$ as
$$
\mathscr{F}(u_1,\ldots,u_{n_1},s_1,\ldots,s_{n}):=\delta_{s_1}(u_{i_1})\dots\delta_{s_{n}}(u_{i_{n}})\delta_{-\hat s_{n}}(u_{i_{n}})\dots\delta_{-\hat s_{1}}(u_{i_{1}}),
$$
where $i_1,\ldots,i_n$ are the indexes and $\hat s_1,\ldots,\hat s_{n_1}\in (0,1)$ are the real numbers yielded by \cref{propdecomposizione} {associated to the basis $\mathfrak{V}$}. Thanks to the Baker-Campbell-Hausdorff formula, the map $\mathscr{F}$ is a polynomial in the coordinates of the $u_i$s and in the $s_i$. This means that each of its derivatives in $s_1,\ldots,s_n$ is still a polynomial in the coordinates of the $u_i$s. This means, since $\mathscr{F}(u_1,\ldots,u_{n_1},s_1,\ldots,s_n)=\hat F_{\{u_1,\ldots,u_{n_1}\}}(s_1,\ldots,s_n)$, that the functions $\hat F_{\mathfrak{W}^\kappa}$ converge in the $C^\infty$ metric to $\hat F_{\mathfrak{V}}=F_{\mathfrak{V}}$, {on $V$, where $V$ is the open set in \cref{propdecomposizione}}. 

{This implies that for $\kappa$ sufficiently big, up to taking $V$ slightly smaller, the maps $\hat F_{\mathfrak{W}^\kappa}:V\to \hat F_{\mathfrak{W}^\kappa}(V)\supset B(0,\zeta(\mathfrak{V}))$ are diffeomorphisms, where $\zeta(\mathfrak{V})$ is some fixed number, using \cite[Lemma~B.3]{ALDNG}.}
Let us note that the above discussion implies the following. Let $\mathfrak{V}$ be as above. There exists a sufficiently small neighbourhood $V$ in $\mathbb R^n$, and two numbers $1\geq \sigma_0(\mathfrak{V}),\zeta(\mathfrak{V})>0$  such that if $\mathfrak{W}$ is another basis of $V_1(\mathbb{G})$, with
\begin{equation}
    \label{eq:basivicine}
    \text{$\lvert w_\ell-v_\ell\rvert\leq \sigma_0(\mathfrak{V})$ for every $\ell=1,\ldots,n_1$,}
\end{equation}
 we have 
$B(0,\zeta(\mathfrak{V}))\subseteq \hat F_{\mathfrak{W}}(V)$.
Throughout the rest of the proof, we denote 
\begin{equation}\label{sceltalambda}    \Lambda(\mathfrak{V}):=4\mathrm{diam}(T)/\zeta(\mathfrak{V}),
\end{equation}
and we will denote by $\bar{F}_{\mathfrak{W}}$ the map $\delta_{\Lambda(\mathfrak{V})}\circ \hat F_{\mathfrak{W}}$. In this way we have $B( T,\mathrm{diam}(T))\Subset \bar F_\mathfrak{W}(V)$. In addition, arguing as in the proof of \cite[Lemma~4.5]{ALDPolynomial}, we also infer that
$$\sum_{i=1}^n\lvert s_i\rvert+\lvert \hat s_i\rvert \leq \frac{4n}{\zeta(\mathfrak{V})}\lVert  \hat F_{\mathfrak{W}}(s_1,\ldots,s_{n})\rVert=\frac{n}{\mathrm{diam}(T)}\lVert \bar F_{\mathfrak{W}}(s_1,\ldots,s_{n})\rVert=:c_0\lVert \bar F_{\mathfrak{W}}(s_1,\ldots,s_{n})\rVert,$$
for every $(s_1,\ldots,s_{n})\in V$. Note that the constant $c_0$ does not depend on the basis $\mathfrak{W}$. 
\end{osservazione}

\begin{proposizione}\label{translcool}
     Let $\mathfrak{V}:=\{v_1,\ldots,v_{n_1}\}\subseteq (\mathbb{S}^{n_1-1})^{n_1}$ be a basis of $V_1(\mathbb{G})$. In the notations of \cref{FTV}, for every $\epsilon>0$ there exists a $\tau=\tau(\mathfrak{V})\in B(0,\min\{\lambda,\epsilon\}/4)$ such that $B(0,\lambda/4)\subseteq\tau\cdot T\subseteq \bar F_{\mathfrak{V}}(V)$, $\tau\cdot p_j\cdot \delta_{1/2}(\tau)^{-1}\in  \bar F_{\mathfrak{V}}(V)$ for every $j=1,\ldots,2^Q$, and 
\begin{equation}\label{intersectioncoordinateplanes}
\bar F_{\mathfrak{V}}(e_i^\perp\cap V)\cap\{\tau\cdot p_j\cdot \delta_{1/2}(\tau)^{-1}:j=1,\ldots,2^Q\}=\emptyset, \quad \forall i=1,\ldots,n,
\end{equation}
where as usual $e_1,\ldots,e_n$ denote the standard orthonormal basis of $\R^n$. 
\end{proposizione}

\begin{proof}
Thanks to the choice of $s$ as above, there exists $0<\delta\leq\min\{\lambda,\epsilon\}/4$ such that for every $\tau \in B(0,\delta)$ we have $ \tau \cdot T\subseteq \bar F_{\mathfrak{V}}(V)$.  We can take $\delta$ small enough such that, for every $j$, the map $C_j(\tau):= \tau\cdot p_j\cdot\delta_{1/2}(\tau^{-1})$ is Euclidean biLipschitz on $B(0,\delta)$ and $C_j(B(0,\delta))\subseteq \bar F_{\mathfrak{V}}(V)$ for every $j=1,\ldots,2^Q$.
Indeed, the map $C_j$ is a polynomial in the coordinates of $\tau$ and
$$DC_j(0)=\begin{pmatrix}
\mathcal{D}_{V_1} &\dots& *\\
0&\ddots&\vdots\\
0&\dots&\mathcal{D}_{V_s}
\end{pmatrix},$$
where $\mathcal{D}_{V_j}$ is an upper triangular matrix with  $2^{-j}$ on the diagonal. 
    Let us assume by contradiction that for every $\tau\in B(0,\delta)$ we have that \eqref{intersectioncoordinateplanes} fails.
    For every $i=1,\ldots,n$ and $j=1,\ldots,2^Q$ the set 
    $$
    \{\tau\in B(0,\delta):C_j(\tau)\in \bar F_{\mathfrak{V}}(e_i^\perp\cap V)\cap C_j(B(0,\delta))\}= C_j^{-1}(\bar F_{\mathfrak{V}}(e_i^\perp\cap V)\cap C_j(B(0,\delta)))
    $$ 
    is Borel measurable. The fact that \eqref{intersectioncoordinateplanes} fails, implies that there exist $j\in\{1,\ldots,2^Q\}$ and $i\in\{1,\ldots,n\}$ such that $C_j^{-1}(\bar F_{\mathfrak{V}}(e_i^\perp)\cap V)\cap B(0,\delta)$ has positive measure. 
Since $C_j$ is biLipschitz between $B(0,\delta)$ and $C_j(B(0,\delta))$, we infer that $C_j^{-1}(\bar F_{\mathfrak{V}}(e_i^\perp\cap V)\cap C_j(B(0,\delta)))$ has positive measure if and only if $\bar F_{\mathfrak{V}}(e_i^\perp\cap V)\cap C_j(B(0,\delta))$ has positive measure. However, since $C_j(B(0,\delta))\subseteq\bar  F_{\mathfrak{V}}(V)$ and $\bar F_{\mathfrak{V}}^{-1}$ is Lipschitz on $\bar F_{\mathfrak{V}}(V)$, we see that this is not the case. This concludes the proof of the proposition.
\end{proof}

\begin{teorema}\label{prop:tilings1}
 Let $\mathfrak{V}:=\{v_1,\ldots,v_{n_1}\}\subseteq (\mathbb{S}^{n_1-1})^{n_1}$ be a basis of $V_1(\mathbb{G})$. Then, there is an open set $V\subseteq (0,1)^n$, indexes $i_1,\ldots,i_{2n}\in \{1,\ldots,n_1\}$, $\tau(\mathfrak{V})\in B(0,\lambda/4)$ and parameters
 $$0<\sigma(\mathfrak{V}),\rho(\mathfrak{V}), \xi(\mathfrak{V})\leq 1/10, \qquad \Lambda(\mathfrak{V})>0,$$
 such that for every basis $\mathfrak{W}$ of $V_1(\mathbb{G})$ for which
     $$\lvert w_i-v_i\rvert\leq \sigma(\mathfrak{V})\qquad\text{for every }i=1,\ldots,n_1,$$
we have
 \begin{itemize}
\item[(i)] $\tau\cdot p_j\cdot\delta_{1/2}(\tau)^{-1}\in B(p_j,\lambda/8)\subseteq \tau\cdot T$ for every $j=1,\ldots,2^Q$;
     \item[(ii)] the cones $C(w_i,\sigma)$ are separated;
     \item[(iii)] 
for  every $j\in\{1,\ldots,2^Q\}$ and every $y\in B(0,\rho(\mathfrak{V}))$ there are $s_1,\ldots,s_{2n}\in V$ for which
$$ y^{-1}\cdot\tau\cdot  p_j\cdot\delta_{1/2}(\tau)^{-1} = \delta_{\Lambda(\mathfrak{V})s_1}(w_{i_1})\cdot\ldots\cdot\delta_{\Lambda(\mathfrak{V})s_{2n}}(w_{i_{2n}}),$$
and there holds
\begin{itemize}
    \item[($\alpha$)]${\displaystyle \min\{\lvert s_i \rvert: i=1,\ldots,2n\}\geq \xi(\mathfrak{V}).}$
    \item[($\beta$)] ${\displaystyle \max\{\lvert s_i\rvert:i=1,\ldots,2n\}\leq \frac{n}{\mathrm{diam}(T)}\lVert y^{-1}\cdot\tau\cdot  p_j\cdot\delta_{1/2}(\tau)^{-1} \rVert}$ for every $y\in B(0,\rho)$.
\end{itemize}
 \end{itemize}
\end{teorema}

\begin{proof}
It is elementary to see that there exists $\epsilon>0$ such that if 
$$\tau\in B(0,\epsilon)\qquad\text{and}\qquad\sigma\leq\epsilon,$$
then items (i) and (ii) are trivially satisfied.
Let us now put ourselves in the notations of \cref{FTV} and let
us note that there exists $\rho_0>0$ depending only on $\epsilon>0$ such that
$$
B(0,\rho_0)\cdot\tau\cdot p_j\cdot\delta_{1/2}(\tau)^{-1}\subseteq B(T,\diam T)\subseteq \bar F_\mathfrak{W}(V),
$$
for every basis $\mathfrak{W}$ for which \eqref{eq:basivicine} is satisfied. Since we are in the notations of \cref{FTV} here we choose $\Lambda(\mathfrak{V})$ as in \eqref{sceltalambda}. In addition, it is immediate to see that thanks to the above considerations,  for every basis $\mathfrak{W}$ for which \eqref{eq:basivicine} is satisfied we also infer that item (ii,$\beta$)
is automatically verified.

Let us assume by contradiction that there exist infinitesimal sequences $\sigma_\kappa\leq \sigma_0/4$ and $\rho_\kappa\leq\rho_0/2 $ such that, denoting with $\{e_1,\ldots,e_n\}$ the standard basis in $\mathbb R^n$, 

\begin{equation}
\begin{split}
    \inf\{&\lvert \langle \bar F^{-1}_{\mathfrak{W}}(y^{-1}\cdot\tau\cdot p_j\cdot\delta_{1/2}(\tau)^{-1}),e_i\rangle\rvert:\,y\in B(0,\rho_\kappa),\,i\in\{1,\ldots,n\},\,j\in\{1,\ldots,2^Q\}\\
          &\qquad\qquad\text{ and }\mathfrak{W}=\{w_1,\ldots,w_{n_1}\}\text{ with }\lvert w_\ell-v_\ell\rvert\leq \sigma_\kappa\text{ for every }\ell=1,\ldots,n_1\}<\kappa^{-1}.\nonumber
          \end{split}
\end{equation}
Up to subsequences we can assume that there are $y_k\in B(0,\rho_\kappa)$, $i\in\{1,\ldots,n\},\,j\in\{1,\ldots,2^Q\}$ and a basis $\mathfrak{W}_\kappa$ such that $\lvert w_\ell-v_\ell\rvert\leq \sigma_\kappa$ and 
\begin{equation}
\lvert \langle \bar F^{-1}_{\mathfrak{W}_\kappa}(y^{-1}_\kappa\cdot\tau\cdot p_j\cdot\delta_{1/2}(\tau)^{-1}),e_i\rangle\rvert\leq \kappa^{-1}.
    \label{minutness}
\end{equation}
If we show that $\lim_{\kappa\to \infty} \bar F^{-1}_{\mathfrak{W}_\kappa}(y^{-1}_\kappa\cdot\tau\cdot p_j\cdot\delta_{1/2}(\tau)^{-1})=F^{-1}_{\mathfrak{V}}(\tau\cdot p_j\cdot\delta_{1/2}(\tau)^{-1})$, the proof of the proposition is achieved by the continuity of the scalar product and \eqref{minutness}, which would contradict our choice of $\tau$.
However, by definition the maps  $\bar F^{-1}_{\mathfrak{W}_\kappa}$ are uniformly equicontinuous and all defined on $\tau\cdot T$ and since the $\bar F_{\mathfrak{W}_\kappa}$ uniformly converge to $F_{\mathfrak{V}}$ on $V$, the proof of the proposition is achieved.
\end{proof}

\begin{proposizione}\label{prop:PD}
    Let $C(e_i,\sigma_i)$ for $i=1,\ldots,n_1$ be a family of $n_1$ separated cones with $e_i\in\mathbb{S}^{n_1-1}$ and $\sigma_i>0$. Then, there exists an $\ell_0\in\N$ such that for every $\ell\geq \ell_0$ there is a covering
    \begin{equation}
    \mathscr{C}_\ell:=\{C(w_{i,j}^\ell,\alpha_{i,j}^\ell):\alpha_{i,j}^\ell\leq \ell^{-1},\,w_{i,j}^\ell\in C(e_i,\sigma_i),\,\text{for every }i=1,\ldots,n_1\text{ and }j\in\N\},
        \label{eq:covering}
    \end{equation}
    of the cones $C(e_i,\sigma_i)$, namely
    \begin{equation}
         C(e_i,\sigma_i)\setminus \{0\}\subseteq \bigcup_{j\in\N}\mathrm{int} \,C(w_{i,j}^\ell,\alpha_{i,j}^\ell),\qquad\text{ for every }i=1,\ldots,n_1\text{ and every }\ell\in\N,
         \label{eq:covering2}
    \end{equation}
    such that the following property is true. 
    
    For every choice $j_1,\ldots,j_{n_1}\in\N$ the cones $C(w_{i,j_i}^\ell,\alpha_{i,j_i}^\ell)$ are separated and 
    there are    
    $$\tau\in B(0,\lambda/4),\quad i_1,\ldots,i_{2n}\in\{1,\ldots,n_1\},\quad 0<\xi\leq 1/10;\quad 0<\rho\leq 1/10,\quad \Lambda>0,$$
    for which the following holds. For every $\kappa=1,\ldots,2^Q$, $y\in B(0,\rho)$, $u_i\in C(w_{i,j_{i}}^\ell,\alpha_{i,{j_i}}^\ell)$, every $i=1,\ldots, n_1$ there are $s_1,\ldots, s_{2n}\in (0,1)$ such that we can write 
    \begin{equation}
        y^{-1}\tau p_\kappa\delta_{1/2}(\tau)^{-1}=\delta_{\Lambda s_1}(u_{i_1})\cdot \ldots\cdot \delta_{\Lambda s_n}(u_{i_n})\cdot \delta_{-\Lambda  s_{n+1}}(u_{i_{n+1}})\cdot\ldots\cdot \delta_{-\Lambda s_{2n}}(u_{i_{2n}}),
        \nonumber
    \end{equation}
    with $\min_{i=1,\ldots, 2n}s_i>\xi$ and
$$\max\{\lvert s_i\rvert:i=1,\ldots,2n\}\leq \frac{n}{\mathrm{diam}(T)}\lVert y^{-1}\cdot\tau\cdot  p_\kappa\cdot\delta_{1/2}(\tau)^{-1} \rVert,$$
where $p_1,\ldots,p_{2^Q}$ are the points yielded by Proposition~\ref{prop:tilings}.
\end{proposizione}

\begin{proof} Let us argue by contradiction, assuming that the conclusion of the proposition is false.
This means that, up to subsequences, for every $\ell\in\N$ and every covering $\mathscr{C}_{\iota_\ell}$ as in \eqref{eq:covering} and satisfying \eqref{eq:covering2}, there are $j_1^\ell,\ldots,j_{n_1}^\ell\in\N$ such that the cones $C(w_{i,j^\ell_{i}}^\ell,\alpha_{i,j^\ell_i}^\ell)$ are separated\footnote{Thanks to the definition of the covering $\mathscr{C}_\ell$, the fact that $\alpha_{i,j_i}\leq \ell^{-1}$ this is trivially satisfied for every $\ell$ sufficiently big.} and for every 
    $$\tau\in B(0,\lambda/4),\quad i_1,\ldots,i_{2n}\in\{1,\ldots,n_1\},\quad \Lambda>0,$$
    there are 
    $\kappa_\ell\in \{1,\ldots,2^Q\}$, $y_\ell\in B(0,\ell^{-1})$, $u_i^\ell\in C(w_{i,j^\ell_{i}}^\ell,\alpha_{i,j^\ell_i}^\ell)$ with $i=1,\ldots, n_1$ such that for any $s_1,\ldots, s_{2n}\in (0,1)$ we cannot write
    \begin{equation}
        y^{-1}_\ell\tau p_j\delta_{1/2}(\tau)^{-1}=\delta_{\Lambda s_1}(u_{i_1}^\ell)\cdot \ldots\cdot \delta_{\Lambda s_n}(u_{i_n}^\ell)\cdot \delta_{-\Lambda s_{n+1}}(u_{i_{n+1}}^\ell)\cdot\ldots\cdot \delta_{-\Lambda  s_{2n}}(u_{i_{2n}}^\ell),
        \nonumber
    \end{equation}
with $\min_{i=1,\ldots, 2n} s_i >\ell^{-1}$ and 
$\max\{\lvert s_i\rvert:i=1,\ldots,2n\}\leq \frac{n}{\mathrm{diam}(T)}\lVert y^{-1}\cdot\tau\cdot  p_j\cdot\delta_{1/2}(\tau)^{-1} \rVert$.
    Up to taking further subsequences of coverings, we can assume by compactness that 
\begin{equation}
    \lim_{j\to \infty}w^\ell_{i,j^\ell_i}=\bar{w}_i\in C(e_i,\sigma_i).
\end{equation}
Let us denote $\bar{\mathfrak{W}}:=\{\bar w_1,\ldots,\bar w_{n_1}\}$ and 
let $V\subseteq (0,1)^n$ be the open set, $i_1,\ldots,i_{2n}\in \{1,\ldots,n_1\}$ be the indexes, $\tau(\bar{\mathfrak{W}})\in B(0,\lambda/4)$ and finally let
$$0<\sigma(\bar{\mathfrak{W}}),\rho(\bar{\mathfrak{W}}), \xi(\bar{\mathfrak{W}})\leq 1/10, \qquad \Lambda(\bar{\mathfrak{W}})>0,$$
be the parameters yielded by \cref{prop:tilings1} relative to the basis $\bar{\mathfrak{W}}$.

The absurd assumption gives us a $\kappa \in \{1,\ldots,2^Q\}$ which we can consider to be fixed up to subsequences, a sequence $y_\ell\to 0$ and a sequence $u_i^\ell\in C(w_{i,j^\ell_{i}}^\ell,\alpha_{i,j^\ell_i}^\ell)$ for which for any $s_1,\ldots, s_{2n}\in (0,1)$ we cannot write
    \begin{equation}
        y^{-1}_\ell\tau p_j\delta_{1/2}(\tau)^{-1}=\delta_{\Lambda s_1}(u_{i_1}^\ell)\cdot \ldots\cdot \delta_{\Lambda s_n}(u_{i_n}^\ell)\cdot \delta_{-\Lambda s_{n+1}}(u_{i_{n+1}}^\ell)\cdot\ldots\cdot \delta_{-\Lambda  s_{2n}}(u_{i_{2n}}^\ell).
        \nonumber
    \end{equation}
with $\min_{i=1,\ldots, 2n} s_i >\ell^{-1}$  and 
$\max\{\lvert s_i\rvert:i=1,\ldots,2n\}\leq \frac{n}{\mathrm{diam}(T)}\lVert y^{-1}\cdot\tau\cdot  p_j\cdot\delta_{1/2}(\tau)^{-1} \rVert$.

However, it is clear that since $\alpha_{i,j^\ell_i}^\ell\leq \ell^{-1}$, we must have
$\lim_{i\to\infty}u_i^\ell=\bar w_i$. Therefore, there exists an $\ell_0\in\N$ such that 
$$\lvert u_i^\ell-\bar{w}_i\rvert\leq \sigma(\bar{\mathfrak{W}})\qquad\text{and}\qquad y_\ell\in B(0,\rho(\bar{\mathfrak{W}})),\qquad\text{for every $\ell\geq \ell_0$}.$$
This, however, would contradict \cref{prop:tilings1}.
\end{proof}

The above discussion implies, in particular, the following theorem. Let us first start with an observation. If we have a cube $T$ such that 
\[
T=\bigcup_{j=1}^{2^Q}p_j\cdot\delta_{1/2}(T),
\]
then, for every $\tau$, we have 
\[
\tau\cdot T=\bigcup_{j=1}^{2^Q}\tau\cdot p_j\cdot\delta_{1/2}(\tau)^{-1}\cdot\delta_{1/2}(\tau\cdot T).
\]
Hence, the translated cube $\tau\cdot T$ has $\tau\cdot p_j\cdot\delta_{1/2}(\tau)^{-1}$ as new centers of dyadic subcubes. Putting together all the information gathered in this subsection, translating the cube in \cref{prop:tilings},  and in particular using \cref{prop:tilings1}, we get the following.

\begin{proposizione}\label{cubofigo}
     Let $\mathbb{G}$ be a Carnot group of homogeneous dimension $Q$, and topological dimension $n$. Let $C(e_i,\sigma_i)$ for $i=1,\ldots,n_1$ be a family of $n_1$ separated cones with $e_i\in\mathbb{S}^{n_1-1}$ and $\sigma_i>0$. 
     Then, there exists an $\ell_0\in\N$ such that for every $\ell\geq \ell_0$ there is a covering
    \begin{equation}
    \mathscr{C}_\ell:=\{C(w_{i,j}^\ell,\alpha_{i,j}^\ell):\alpha_{i,j}^\ell\leq \ell^{-1},\,w_{i,j}^\ell\in C(e_i,\sigma_i)\cap \mathbb{S}^{n_1-1},\,\text{for every }i=1,\ldots,n_1\text{ and }j\in\N\},
        \nonumber
    \end{equation}
    of the cones $C(e_i,\sigma_i)$, namely
    \begin{equation}
    C(e_i,\sigma_i)\setminus \{0\}\subseteq \bigcup_{j\in\N}\mathrm{int} \,C(w_{i,j}^\ell,\alpha_{i,j}^\ell),\qquad\text{ for every }i=1,\ldots,n_1\text{ and every }\ell\in\N,
         \nonumber
    \end{equation}
    such that for every choice $J:=\{j_1,\ldots,j_{n_1}\}\in\N^{n_1}$ there is a compact set $T_J$, $2^Q$ points $p_{J,1},\ldots,p_{J,2^Q}\in T_J$ and  
    $$
    i_1,\ldots,i_{2n}\in\{1,\ldots,n_1\},\quad 0<\xi\leq 1/10;\quad 0<\rho\leq 1/10,\quad \Lambda>0,
    $$
    for which we have
     \begin{itemize}
          \item[(i)] $T_J$ is the closure of an open set and thus $\mathcal{H}^Q(T_J)>0$ and $T_J$ is a left translate of the cube $T$ given in \cref{prop:tilings} by some suitable $\tau_J\in B(0,\lambda/4)$ and where $\lambda$ is the same constant given by \cref{prop:tilings};
    \item[(ii)] $T_J=\bigcup_{j=1}^{2^Q}p_j\cdot\delta_{1/2}(T_J)$;
    \item[(iii)] $\mathcal{H}^Q(p_j\cdot\delta_{1/2}(T_J)\cap p_k\cdot\delta_{1/2}(T_J))=0$ for every $j\neq k\in\{1,\ldots,2^Q\}$;
    \item[(iv)] $B(0,\lambda/4)\subseteq T_J$ ,
    \item[(v)] the cones $C(w_{i,j_i}^\ell,\alpha_{i,j_i}^\ell)$ are separated and  for every basis $\mathfrak{U}=\{u_1,\ldots,u_{n_1}\}$ with $u_i\in C(w_{i,j_i}^\ell,\alpha_{i,j_i}^\ell)\cap \mathbb{S}^{n_1-1}$, for every $j=1,\ldots,2^Q$, every $y\in B(0,\rho)$ we can write
 \[
 y^{-1}\cdot p_{J,j} = \delta_{s_1}(u_{i_1})\cdot\ldots\cdot\delta_{s_M}(u_{i_M});
 \]
 with $\min_{\{i=1,\ldots,M\}} \{\lvert s_i\rvert :s_i\neq 0\} > \xi$, and where $M=2n$. In addition, we also have
 \begin{equation}
      \max\{\lvert s_i\rvert:i=1,\ldots,M\}\leq \frac{n\Lambda}{\mathrm{diam(T)}}\lVert y^{-1}p_{J,j}\rVert=:c_0\lVert y^{-1}p_{J,j}\rVert,
      \label{eq:stima}
 \end{equation}
 for every $j=1,\ldots,2^Q$.
     \end{itemize}
\end{proposizione}

\printbibliography[title=References, heading=bibintoc]

\end{document}